\documentclass[a4paper,11pt]{article}
\usepackage{amsmath,amsthm,amssymb,enumitem}
\usepackage[mathscr]{eucal}
\usepackage{tocloft}

\usepackage{tikz}
\usetikzlibrary{decorations.pathreplacing,shapes.geometric,petri}
\usetikzlibrary{calc,positioning}

\usepackage[nosort,nocompress,noadjust]{cite}

\usepackage[linktocpage=true,bookmarks=false,hyperfootnotes=false,colorlinks,
    linkcolor={red!60!black},
    citecolor={blue!50!black},
    urlcolor={blue!80!black}]{hyperref}

\renewcommand{\eqref}[1]{\hyperref[#1]{(\ref{#1})}}

\newlist{enumlist}{enumerate}{1}
\setlist[enumlist]{labelindent=0cm,label=\arabic*.,labelwidth=2.5ex,labelsep=0.5ex,leftmargin=3ex,align=left,topsep=0.5ex,itemsep=1ex,parsep=1ex}

\newlist{itemlist}{itemize}{1}
\setlist[itemlist]{labelindent=0cm,label=$\bullet$,labelwidth=2.5ex,labelsep=0.5ex,leftmargin=3ex,align=left,topsep=0.5ex,itemsep=1ex,parsep=1ex}

\numberwithin{equation}{section}

{\theoremstyle{definition}\newtheorem{definition}{Definition}[section]

\newtheorem{remark}[definition]{Remark}
}

\newtheorem{proposition}[definition]{Proposition}
\newtheorem{lemma}[definition]{Lemma}
\newtheorem{theorem}[definition]{Theorem}
\newtheorem{corollary}[definition]{Corollary}

\newcommand{\bim}[3]{\mathord{\raisebox{-0.4ex}[0ex][0ex]{\scriptsize $#1$}{#2}\hspace{-0.25ex}\raisebox{-0.4ex}[0ex][0ex]{\scriptsize $#3$}}}

\newcommand{\coms}[4]{\begin{array}{ccc}
#1 & \subset & #2 \\ \cup & & \cup \\ #3 & \subset & #4
\end{array}}

\newcommand{\C}{\mathbb{C}}
\newcommand{\cC}{\mathcal{C}}
\newcommand{\eps}{\varepsilon}
\newcommand{\al}{\alpha}
\newcommand{\be}{\beta}
\newcommand{\albar}{\overline{\alpha}}
\newcommand{\End}{\operatorname{End}}
\newcommand{\Irr}{\operatorname{Irr}}
\newcommand{\Mor}{\operatorname{Mor}}

\newcommand{\ot}{\otimes}
\newcommand{\recht}{\rightarrow}

\newcommand{\Z}{\mathbb{Z}}
\newcommand{\vphi}{\varphi}

\newcommand{\op}{^\text{\rm op}}
\newcommand{\bG}{\mathbb{G}}
\newcommand{\cO}{\mathcal{O}}
\newcommand{\id}{\mathord{\text{\rm id}}}
\newcommand{\om}{\omega}

\newcommand{\N}{\mathbb{N}}
\newcommand{\cL}{\mathcal{L}}
\newcommand{\ovt}{\mathbin{\overline{\otimes}}}
\newcommand{\mult}{\operatorname{mult}}
\newcommand{\Tr}{\operatorname{Tr}}
\newcommand{\real}{\operatorname{Re}}

\newcommand{\cD}{\mathcal{D}}

\newcommand{\counit}{\epsilon}

\newcommand{\cH}{\mathcal{H}}

\newcommand{\cZ}{\mathcal{Z}}

\newcommand{\cG}{\mathcal{G}}

\newcommand{\cK}{\mathcal{K}}

\newcommand{\cJ}{\mathcal{J}}

\newcommand{\cF}{\mathcal{F}}
\newcommand{\T}{\mathbb{T}}
\newcommand{\actson}{\curvearrowright}

\newcommand{\cS}{\mathscr{S}}
\newcommand{\cA}{\mathcal{A}}
\newcommand{\TLJ}{\operatorname{TLJ}}
\newcommand{\cB}{\mathcal{B}}
\newcommand{\zdiml}{\mathop{\text{\rm zd}_\ell}}
\newcommand{\zdimr}{\mathop{\text{\rm zd}_r}}
\newcommand{\cHbar}{\overline{\cH}}
\newcommand{\cW}{\mathcal{W}}
\newcommand{\cU}{\mathcal{U}}
\newcommand{\QN}{\operatorname{QN}}
\newcommand{\drank}{d_{\text{\rm rank}}}
\newcommand{\Ker}{\operatorname{Ker}}
\newcommand{\cM}{\mathcal{M}}
\newcommand{\lspan}{\operatorname{span}}
\newcommand{\cN}{\mathcal{N}}
\newcommand{\cR}{\mathcal{R}}
\newcommand{\cKreg}{\mathcal{K}^{\text{\rm reg}}}
\newcommand{\Creg}{C_{\text{\rm reg}}}
\newcommand{\orbit}{\operatorname{orbit}}
\newcommand{\indC}{\mathord{\text{\rm ind-}\mathcal{C}}}
\newcommand{\onb}{\operatorname{onb}}
\newcommand{\dpr}{^{\prime\prime}}

\newcommand{\cI}{\mathcal{I}}
\newcommand{\cV}{\mathcal{V}}

\newcommand{\cE}{\mathcal{E}}
\newcommand{\Tor}{\operatorname{Tor}}
\newcommand{\Ext}{\operatorname{Ext}}
\newcommand{\Hom}{\operatorname{Hom}}
\newcommand{\Ltil}{\widetilde{L}}
\newcommand{\Rtil}{\widetilde{R}}
\newcommand{\ftil}{\widetilde{f}}
\newcommand{\gtil}{\widetilde{g}}
\newcommand{\sub}{\text{\rm sub}}
\newcommand{\cHreg}{\mathcal{H}_{\text{\rm reg}}}
\newcommand{\HH}{\operatorname{HH}}
\newcommand{\Xbar}{\overline{X}}
\newcommand{\bns}{\beta_n^{\text{\rm\tiny (2)}}}
\newcommand{\bnl}{\beta_n^{\text{\rm\tiny (2)}}}
\newcommand{\bes}{\beta^{\text{\rm\tiny (2)}}}
\newcommand{\bel}{\beta^{\text{\rm\tiny (2)}}}
\newcommand{\rd}{\operatorname{d}}
\newcommand{\rdl}{\mathop{\text{\rm d}_\ell}}
\newcommand{\rdr}{\mathop{\text{\rm d}_r}}
\newcommand{\Aut}{\operatorname{Aut}}
\newcommand{\cCtil}{\widetilde{\mathcal{C}}}
\newcommand{\cAtil}{\widetilde{\mathcal{A}}}
\newcommand{\vphitil}{\widetilde{\varphi}}
\newcommand{\psitil}{\widetilde{\psi}}
\newcommand{\deltatil}{\widetilde{\delta}}
\newcommand{\mtil}{\widetilde{m}}
\newcommand{\cP}{\mathcal{P}}

\renewcommand{\Im}{\operatorname{Im}}

\tikzset{Box/.style={very thick, rounded corners}}
\tikzset{marked/.style={star, star point height = .75mm, star points =5, fill=black,minimum size=2mm, inner sep=0mm} }
\tikzset{verythickline/.style = {line width=7pt}}
\tikzset{thickline/.style = {line width=5pt}}
\tikzset{medthick/.style = {line width=3pt}}
\tikzset{med/.style = {line width=2pt}}
\tikzset{count/.style = {fill=white,circle,draw,thin, inner sep=2pt}}
\tikzset{rcount/.style = {fill=white,rectangle,draw,thin,inner sep=2pt, rounded corners}}
\tikzset{cpr/.style = {draw,fill=white,rectangle,thin, rounded corners}}

\begin{document}

\begin{center}
{\boldmath\LARGE\bf Cohomology and $L^2$-Betti numbers for subfactors \vspace{0.5ex}\\ and quasi-regular inclusions}

\bigskip

{\sc by Sorin Popa\footnote{Mathematics Department, UCLA, Los Angeles, CA 90095-1555 (United States), popa@math.ucla.edu\\
Supported in part by NSF Grant DMS-1400208}, Dimitri Shlyakhtenko\footnote{Mathematics Department, UCLA, Los Angeles, CA 90095-1555 (United States), shlyakht@math.ucla.edu\\ Supported in part by NSF Grant DMS-1500035} and Stefaan Vaes\footnote{KU~Leuven, Department of Mathematics, Leuven (Belgium), stefaan.vaes@kuleuven.be \\
    Supported in part by European Research Council Consolidator Grant 614195, and by long term structural funding~-- Methusalem grant of the Flemish Government.}}
\end{center}

\begin{abstract}\noindent
We introduce $L^2$-Betti numbers, as well as a general homology and cohomology theory for the standard invariants of subfactors, through the associated quasi-regular symmetric enveloping inclusion of II$_1$ factors. We actually develop a (co)homology theory for arbitrary quasi-regular inclusions of von Neumann algebras. For crossed products by countable groups $\Gamma$, we recover the ordinary (co)homology of $\Gamma$. For Cartan subalgebras, we recover Gaboriau's $L^2$-Betti numbers for the associated equivalence relation. In this common framework, we prove that the $L^2$-Betti numbers vanish for amenable inclusions and we give cohomological characterizations of property~(T), the Haagerup property and amenability. We compute the $L^2$-Betti numbers for the standard invariants of the Temperley-Lieb-Jones subfactors and of the Fuss-Catalan subfactors, as well as for free products and tensor products.
\end{abstract}

\mbox{}\vspace{-6ex}

{\footnotesize\tableofcontents}

\vfill

\mbox{}\clearpage

\section{Introduction}

It has been a longstanding problem to define a suitable (co)homology theory, including the theory of $L^2$-cohomology and of $L^2$-Betti numbers, for objects encoding  ``quantum symmetries'' that arise in Jones's theory of subfactors, \cite{Jo82}.  Such objects include the standard invariant in Jones subfactor theory ($\lambda$-lattice or Jones planar algebra), rigid $C^*$-tensor categories as well as representation categories of compact quantum groups.  The main goal of the present paper is to give a definition of such a (co)homology theory.  In fact, our approach gives a unified way of defining (co)homology for discrete groups, measure preserving discrete groupoids and equivalence relations as well as such quantum symmetries.  In this way, we present a common approach to $L^2$-Betti numbers, which includes Atiyah-Cheeger-Gromov $L^2$-Betti numbers of groups, Gaboriau's $L^2$-Betti numbers for equivalence relations, as well as (new) $L^2$-invariants such as $L^2$-Betti numbers associated to a Jones subfactor.

The importance of a suitable definition of $L^2$-Betti numbers in the context of quantum symmetries is apparent already from the case of discrete groups.  Indeed, the theory $L^2$-invariants has had a wide range of applications in  geometry, topology, geometric group theory, ergodic theory and von Neumann algebras, see \cite{Lu02,Po01,Ga01}. They were originally defined by Atiyah \cite{At74} for $\Gamma$-coverings $p : \Xbar \recht X$ of compact Riemannian manifolds, in the context of equivariant index theory, and they were generalized to measurable foliations in \cite{Co78}. When $\Xbar$ is contractible, these are invariants of the group $\Gamma$. For general countable groups $\Gamma$, not necessarily having a nice classifying space, the $L^2$-Betti numbers $\bns(\Gamma)$ were introduced in \cite{CG85}, as the $L(\Gamma)$-dimension of the usual group homology of $\Gamma$ with coefficients in $\ell^2(\Gamma)$. A remarkable result of Gaboriau \cite{Ga01} shows that these numbers are orbit equivalence invariants, and his introduction of these invariants in ergodic theory has led to a number of striking advances in that field.

Key to our approach is the definition of a Hochschild type (co)homology for general {\em quasi-regular} inclusions of von Neumann algebras, which in the irreducible case we show to be equivalent to a Hochschild type (co)homology of an algebra that we canonically associate to such an inclusion. We call it the {\em tube algebra}, the terminology being motivated by the particular case of the symmetric enveloping inclusion of a finite depth subfactor \cite{Oc93}.  When we compute our cohomology theory with $L^2$-coefficients, the resulting cohomology groups are naturally modules over a semifinite von Neumann algebra and in this way, we obtain the notion of $L^2$-Betti numbers of a quasi-regular inclusion.

Quasi-regular inclusions of von Neumann algebras $T\subset S$ are generalizations of crossed product inclusions in which $S=T\rtimes \Gamma$ is a crossed product by a discrete group $\Gamma$ acting by automorphisms of $T$, so that the normalizer $\cN_S(T) = \{u \in \cU(S) \mid uTu^* = T\}$ generates the entire von Neumann algebra $S$. For quasi-regular inclusions, $S$ is generated by finite index $T$-bimodules.

Let us explain how our construction can be used to yield (co)homology theories and $L^2$-Betti numbers for subfactors, groups and equivalence relations.

{\bf Subfactors.}  A subfactor $N \subset M$ gives rise to the group like standard invariant $\cG_{N,M}$ that ``acts'' on $M$. The corresponding crossed product type inclusion, which will be a crucial tool for us, is the \emph{symmetric enveloping (SE) inclusion} $T \subset S$ defined in \cite{Po94a,Po99}. Here, $T = M \ovt M\op$ and $S$ should be thought of as a crossed product of $T$ by an action of $\cG_{N,M}$. Indeed, in the particular case of diagonal subfactors defined by finitely many automorphisms, the standard invariant encodes the discrete group $\Gamma \subset \Aut(M)$ generated by these automorphisms as well as the generating set. The corresponding SE-inclusion is then precisely the inclusion of  $T= M \ovt M\op$ into the crossed product $S=T\rtimes \Gamma$. With this example in mind, the SE-inclusion has been successfully used to define and study several group like properties for standard invariants of subfactors, including amenability, the Haagerup property, property~(T), etc., see \cite{Po94a,Po94b,Po01,PV14}.

Since the inclusion $T\subset S$ is quasi-regular, our definition yields a (co)homology theory and the notion of $L^2$-Betti numbers for the subfactor.  Our tube algebra is then up to Morita equivalence the same as  Ocneanu's tube algebra \cite{Oc93}; in fact, this case was our motivating example for the general definition of the tube algebra.  Since this algebra only depends on the standard invariant, our (co)homology theory and $L^2$-Betti numbers also depend only on the standard invariant $\cG_{N,M}$.  Actually, the definition makes sense in other related contexts, including planar algebras and  rigid $C^*$-tensor categories, such as representation categories of compact quantum groups. In that case, our (co)homology corresponds to quantum group (co)homology for the quantum double of $\bG$. In particular, the $L^2$-Betti numbers should be viewed as the $L^2$-Betti numbers of this quantum double.

{\bf Discrete groups.}
If in the previous case, $N\subset M$ is the diagonal subfactor defined by a finite family of automorphisms of $N$, or more generally for crossed product inclusions $T\subset T\rtimes \Gamma=S$, with $\Gamma$ a discrete group, our (co)homology of the inclusion $T\subset S$ is equivalent to ordinary group (co)homology with coefficients in a unitary representation. The $L^2$-Betti numbers are exactly the Cheeger-Gromov $L^2$-Betti numbers of $\Gamma$. Our tube algebra in this case becomes (essentially) the group algebra of $\Gamma$.

{\bf Measured equivalence relations.}
Given a probability measure preserving equivalence relation $\cR$ on a probability measure space $(X,\mu)$, the associated Cartan subalgebra inclusion  $T=L^\infty(X)\subset S = L(\cR)$ is quasi-regular.  Applying our definition, we recover Gaboriau's $L^2$-Betti numbers of $\cR$, see \cite{Ga01}, as well as groupoid cohomology with coefficients in a unitary representation.  In fact, this example was one of the original motivations for our definition of (co)homology for quasi-regular inclusions.

In the last two sections, we compute $L^2$-Betti numbers in several interesting cases  and show that
the resulting theory goes well with various approximation properties of the quasi-regular inclusion.
We prove that they vanish for amenable irreducible quasi-regular inclusions, as well as for the Temperley-Lieb-Jones subfactors/planar algebras. We prove a formula for the $L^2$-Betti numbers of the free product of quasi-regular inclusions and deduce that the first $L^2$-Betti number of the Fuss-Catalan subfactors is nonzero, while the others vanish. We also briefly discuss homology with trivial coefficients. Finally, we prove that for an irreducible quasi-regular inclusion $T \subset S$, the Haagerup property is equivalent with the existence of a proper $1$-cocycle, while property~(T) is equivalent with all $1$-cocycles being inner. In particular, for property~(T) inclusions, the first $L^2$-Betti number vanishes.

{\bf Acknowledgment.} We are grateful to the Mittag-Leffler Institute for their hospitality during the program {\it Classification of operator algebras: complexity, rigidity, and dynamics}, where the latest version of this paper has been finalized.

\section{Preliminaries}

\subsection{Bimodules over tracial von Neumann algebras}

We fix a von Neumann algebra $T$ with a normal faithful tracial state $\tau$. When $\cH$ is a right Hilbert $T$-module, we denote by $\zdimr(\cH)$ its center valued dimension. In principle, $\zdimr(\cH)$ belongs to the extended positive cone of $\cZ(T)$, but we will only use this notation when $\cH$ is finitely generated as a Hilbert $T$-module. This precisely corresponds to $\zdimr(\cH)$ being a bounded operator. We similarly use the notation $\zdiml(\cH)$ when $\cH$ is a left Hilbert $T$-module.

We call a Hilbert $T$-bimodule $\cH$ \emph{bifinite} if both $\zdiml(\cH)$ and $\zdimr(\cH)$ are bounded. We call their support projections the left, resp.\ right support of $\cH$. When $T$ is a II$_1$ factor, we have $\zdiml(\cH) = \rdl(\cH) 1$ and $\zdimr(\cH) = \rdr(\cH) 1$, where $\rdl(\cH)$ and $\rdr(\cH)$ denote the usual left, resp.\ right, Murray-von Neumann dimension of $\cH$. When $T$ is a II$_1$ factor, it is more common to say that a $T$-bimodule has \emph{finite index} if both $\rdl(\cH)$ and $\rdr(\cH)$ are finite.

Let $z_1,z_2 \in \cZ(T)$ be central projections and $\al : \cZ(T) z_1 \recht \cZ(T) z_2$ a bijective $*$-isomorphism. We say that $\cH$ is an \emph{$\al$-$T$-bimodule} if the right support of $\cH$ equals $z_1$, the left support equals $z_2$ and
$$\xi a = \al(a) \xi \quad\text{for all}\;\; a \in \cZ(T) z_1 \; .$$

\begin{definition}\label{def.irred}
Let $(T,\tau)$ be a von Neumann algebra with a normal faithful tracial state.
We say that a bifinite Hilbert $T$-bimodule $\cH$ with right support $z_1$ and left support $z_2$ is \emph{irreducible} if the space $\End_{T-T}(\cH)$ of $T$-bimodular bounded operators  equals $\cZ(T)z_1$ represented by its right action on $\cH$, and also equals $\cZ(T)z_2$ represented by its left action on $\cH$.
\end{definition}

Note that in the situation of Definition \ref{def.irred}, there is a unique bijective $*$-isomorphism $\al : \cZ(T) z_1 \recht \cZ(T) z_2$ satisfying $\xi a = \al(a) \xi$ for all $\xi \in \cH$, $a \in \cZ(T) z_1$, so that $\cH$ is in particular an $\al$-$T$-bimodule.

Whenever $p \in M_n(\C) \ot T$ is a projection and $\psi : T \recht p(M_n(\C) \ot T)p$ is a normal unital $*$-homomorphism, we define the $T$-bimodule $\cH(\psi)$ given by
$$\cH(\psi) = p(\C^n \ot L^2(T)) \quad\text{and}\quad a \cdot \xi \cdot b = \psi(a) \xi b \quad\text{for all}\;\; a,b \in T, \xi \in \cH(\psi) \; .$$
Denote by $\Tr$ the non-normalized trace on $M_n(\C)$ and by $E_\cZ$ the unique trace preserving conditional expectation of $T$ onto $\cZ(T)$. Then, $\zdimr(\cH) = (\Tr \ot E_\cZ)(p)$. Also, the left support of $\cH(\psi)$ equals the support of the homomorphism $\psi$, i.e.\ the smallest projection $z \in T$ with the property that $\psi(1-z) = 0$.

When $T$ is a II$_1$ factor and $\cH$ a Hilbert $T$-bimodule, then the product $\rdl(\cH) \cdot \rdr(\cH)$ of the left and right dimension of $\cH$ is at least $1$. This follows for instance by using the categorical dimension function on bifinite Hilbert $T$-bimodules. The non-factorial version of this observation is provided by the following lemma.

\begin{lemma}\label{lem.dim-estimate}
Let $\al : \cZ(T) z_1 \recht \cZ(T) z_2$ be a bijective $*$-isomorphism and let $\cH$ be a bifinite $\al$-$T$-bimodule with right support $z_1$ and left support $z_2$. Then,
$$z_2 \leq \zdiml(\cH) \; \al(\zdimr(\cH)) \; .$$
\end{lemma}
\begin{proof}
Take unital $*$-homomorphisms $\vphi : T \recht p (M_n(\C) \ot T)p$ and $\psi : T \recht q (M_n(\C) \ot T)q$ such that $\cH \cong \cH(\vphi)$ and $\cHbar \cong \cH(\psi)$ as Hilbert $T$-bimodules. For any tracial state $\tau_1$ on $T$, we have that $(\Tr \ot \tau_1) \circ \psi$ is a trace on $T$. Therefore,
$$(\Tr \ot E_\cZ) \circ \psi = (\Tr \ot E_\cZ) \circ \psi \circ E_\cZ \; .$$
Since $\cH$ is an $\al$-$T$-bimodule, we have $\psi(a) = (1 \ot \al(az_1))q$ for all $a \in \cZ(T)$. We conclude that for all $x \in T$, we have that
\begin{equation}\label{eq.center-trace}
\begin{split}
(\Tr \ot E_\cZ)(\psi(x)) &= (\Tr \ot E_\cZ)(\psi(E_\cZ(x))) = (\Tr \ot E_\cZ)(q (1 \ot \al(E_\cZ(x)))) \\ & = \zdimr(\cH(\psi)) \, \al(E_\cZ(x)) = \zdiml(\cH) \, \al(E_\cZ(x)) \; .
\end{split}
\end{equation}
We now use the Connes tensor product $\cH \ot_T \cHbar$. Note that $\cH \ot_T \cHbar \cong \cH((\id \ot \psi)\vphi)$. Therefore,
$$\zdimr(\cH \ot_T \cHbar) = (\Tr \ot \Tr \ot E_\cZ)((\id \ot \psi)(p)) \; .$$
Using \eqref{eq.center-trace}, it follows that
\begin{align*}
\zdimr(\cH \ot_T \cHbar) &= (\Tr \ot E_\cZ)(\psi((\Tr \ot \id)(p))) = \zdiml(\cH) \, \al((\Tr \ot E_\cZ)(p)) \\ &= \zdiml(\cH) \, \al(\zdimr(\cH)) \; .
\end{align*}
Since $z_2$ is the left support of $\cH$, the $T$-bimodule $L^2(T z_2)$ is a sub-$T$-bimodule of $\cH \ot_T \cHbar$. Therefore, $z_2 \leq \zdimr(\cH \ot_T \cHbar)$ and the lemma is proved.
\end{proof}

When $T$ is a II$_1$ factor, the bifinite Hilbert $T$-bimodules form a rigid C$^*$-tensor category. In particular, every bifinite Hilbert $T$-bimodule decomposes as a direct sum of finitely many irreducible $T$-bimodules. We need the following non-factorial version of this fact.

\begin{proposition}\label{prop.decompose-irred}
Let $(T,\tau)$ be a von Neumann algebra with a normal faithful tracial state. Every bifinite Hilbert $T$-bimodule is a direct sum of \emph{finitely many} Hilbert $T$-bimodules that are irreducible in the sense of Definition \ref{def.irred}.
\end{proposition}
\begin{proof}
Take a bifinite Hilbert $T$-bimodule $\cH$. Take a positive number $\kappa > 0$ such that $\zdiml(\cH) \leq \kappa \, 1$ and $\zdimr(\cH) \leq \kappa \, 1$. We first prove that $\cH$ is a, possibly infinite, direct sum of irreducible Hilbert $T$-bimodules. Write $\cH$ as $\cH(\psi)$ for some normal unital $*$-homomorphism $\psi : T \recht p(M_n(\C) \ot T)p$. Since $\zdiml(\cH) \leq \kappa \, 1$, we get that $\psi(T)$ has finite index as a von Neumann subalgebra of $p(M_n(\C) \ot T)p$ equipped with the trace $\Tr \ot \tau$. Using e.g.\ \cite[Lemma A.3]{Va07}, also
$$\psi(\cZ(T)) = \psi(T)' \cap \psi(T) \subset \psi(T)' \cap p (M_n(\C) \ot T)p$$
has finite index. We identify $\psi(T)' \cap p (M_n(\C) \ot T)p = \End_{T-T}(\cH)$. Since $\cZ(T)$ is abelian, it follows that $\End_{T-T}(\cH)$ is of type~I and that we can find projections $p_k \in \End_{T-T}(\cH)$ with $\sum_k p_k = 1$ such that for every $k$, $\End_{T-T}(p_k \cH)$ equals the image of $\cZ(T)$ by its left action. By symmetry, we can further decompose and find that $\cH$ is the orthogonal direct sum of irreducible Hilbert $\al_k$-$T$-bimodules $\cH_k \subset \cH$, where $\al_k : \cZ(T) z_{1,k} \recht \cZ(T) z_{2,k}$ are bijective $*$-isomorphisms.

Write $(\cZ(T),\tau) = L^\infty(X,\mu)$ for some standard probability space $(X,\mu)$. We then view each $\al_k$ as a nonsingular partial automorphism of $X$, with domain $D_k \subset X$ and range $R_k \subset X$. Define the set $\cW = \sqcup_k D_k$ as the disjoint union of the sets $D_k$. Define the maps $\pi_1,\pi_2 : \cW \recht X$ given by $\pi_1(x) = x$ and $\pi_2(x) = \al_k(x)$ when $x \in D_k$. The positive measurable function $x \mapsto |\pi_1^{-1}(x)|$ is equal to $\sum_k z_{1,k}$. Similarly, the function $x \mapsto |\pi_2^{-1}(x)|$ equals $\sum_k z_{2,k}$.

Recall that we have chosen $\kappa > 0$ such that $\zdiml(\cH) \leq \kappa \, 1$ and $\zdimr(\cH) \leq \kappa \, 1$. We claim that $\sum_k z_{1,k} \leq \kappa^2 \, 1$. By Lemma \ref{lem.dim-estimate}, we get for all $k$ that
$$z_{1,k} \leq \zdimr(\cH_{\al_k}) \, \al_k^{-1}(\zdiml(\cH_{\al_k})) \leq \kappa \, \zdimr(\cH_{\al_k}) \; .$$
Summing over $k$, it follows that
$$\sum_k z_{1,k} \leq \kappa \, \zdimr\bigr(\bigoplus_k \cH_{\al_k}\bigr) = \kappa \, \zdimr(\cH) \leq \kappa^2 \, 1 \; .$$
So, the claim is proved. Similarly, we get that $\sum_k z_{2,k} \leq \kappa^2 \, 1$.

So after removing from $X$ a set of measure zero, both functions $x \mapsto |\pi_1^{-1}(x)|$ and $x \mapsto |\pi_2^{-1}(x)|$ are bounded on $X$. We can thus write $\cW$ as the disjoint union of finitely many Borel sets $\cW_1,\ldots,\cW_\eta$ such that for each $j$, both the restriction of $\pi_1$ to $\cW_j$ and the restriction of $\pi_2$ to $\cW_j$ are $1$-to-$1$. Denote by $r_{1,j,k} \leq z_{1,k}$ the projection that corresponds to $D_k \cap \cW_j$. Define $r_{2,j,k} = \al_k(r_{1,j,k})$. For every fixed $j$, the sets $(\cW_j \cap D_k)_k$ form a partition of $\cW_j$. Since both the restriction of $\pi_1$ to $\cW_j$ and the restriction of $\pi_2$ to $\cW_j$ are $1$-to-$1$, the projections $(r_{1,j,k})_k$ are orthogonal, as well as the projections $(r_{2,j,k})_k$.

Define the Hilbert $T$-bimodules $\cH_1,\ldots,\cH_\eta \subset \cH$ given by
$$\cH_j = \bigoplus_k \cH_k r_{1,j,k} = \bigoplus_k r_{2,j,k} \cH_k \; .$$
By the orthogonality of the projections $(r_{1,j,k})_k$, we get that $\cH_j r_{1,j,k} = \cH_k r_{1,j,k}$. Similarly, $r_{2,j,k} \cH_j = \cH_k r_{1,j,k}$. The irreducibility of the $\cH_k$ implies that all $\cH_j$ are irreducible.

Since the sets $\cW_1,\ldots,\cW_\eta$ form a partition of $\cW$, we get that
$$\cH = \bigoplus_{j=1}^\eta \cH_j \; .$$
\end{proof}

If $(T,\tau)$ is a von Neumann algebra with a normal faithful tracial state and $\cH$ is a bifinite Hilbert $T$-bimodule, then $\tau$ induces a canonical trace $\Tr^r_\cH$ on $\End_{-T}(\cH)$, i.e.\ the commutant of the right $T$-action on $\cH$, as well as a canonical trace $\Tr^\ell_\cH$ on $\End_{T-}(\cH)$. Both restrict to faithful traces on $\End_{T-T}(\cH)$ and these might be different. We denote by $\Delta_\cH$ the, possibly unbounded, positive, self-adjoint operator affiliated with $\cZ(\End_{T-T}(\cH))$ such that $\Tr^r_\cH = \Tr^\ell_\cH(\, \cdot \, \Delta_\cH)$. The \emph{canonical trace} on $\End_{T-T}(\cH)$, denoted by $\Tr_\cH$ is then defined as
\begin{equation}\label{eq.canonical-trace}
\Tr_\cH = \Tr^\ell_\cH(\, \cdot \, \Delta_\cH^{1/2}) = \Tr^r_\cH(\, \cdot \, \Delta_\cH^{-1/2}) \; .
\end{equation}

\subsection{Quasi-regular inclusions of von Neumann algebras}

We recall here the definition of quasi-regular inclusions (see \cite{Po99,Po01}) and their basic properties, with special emphasis on the case
of irreducible subfactors.

\begin{definition}\label{def.quasiregular}
Let $(S,\tau)$ be a tracial von Neumann algebra and $T \subset S$ a von Neumann subalgebra. The \emph{quasi-normalizer} of $T$ inside $S$ is defined as
\begin{align*}
\QN_S(T) = \Bigl\{x \in S \Bigm| \; & \exists x_1,\ldots,x_n, y_1,\ldots,y_m \in S \;\;\text{such that}\;\; x T \subset \sum_{i=1}^n T x_i \;\;\text{and}\\ & T x \subset \sum_{j=1}^m y_j T \Bigr\} \; .
\end{align*}
We say that $T \subset S$ is \emph{quasi-regular} if $\QN_S(T)\dpr = S$.
\end{definition}

For irreducible subfactors $T \subset S$, the quasi-normalizer is particularly well behaved, as can be seen from the following lemma. All the results in the lemma can be deduced from \cite[Section 1]{PP84}. For the convenience of the reader, we give a self contained proof.

\begin{lemma}\label{lem.well-behaved}
Let $(S,\tau)$ be a II$_1$ factor and $T \subset S$ an irreducible subfactor. Denote by $E : S \recht T$ the unique trace preserving conditional expectation. Let $\cK \subset L^2(S)$ be a finite index $T$-subbimodule.
\begin{enumerate}
\item There exists a basis (in the sense of \cite{PP84}) of $\cK \cap S$ as a right $T$-module: elements $x_1,\ldots,x_n \in \cK \cap S$ satisfying
$$p_\cK(x) = \sum_{i=1}^n x_i E(x_i^* x) \;\;\text{for all}\;\; x \in S \;\; ,$$
where $p_\cK$ is the orthogonal projection of $L^2(S)$ onto $\cK$.

\item Similarly, there exists a basis of $\cK \cap S$ as a left $T$-module: elements $y_1,\ldots,y_m \in \cK \cap S$ such that
$$p_\cK(x) = \sum_{j=1}^m E(x y_j^*) y_j \;\;\text{for all}\;\; y \in S \;\; .$$

\item The space of $T$-bounded vectors in $\cK$ equals $\cK \cap S$.

\item The densely defined linear maps $\cK \ot_T L^2(S) \recht L^2(S)$ and $L^2(S) \ot_T \cK \recht L^2(S)$ given by multiplication are well defined bounded operators.

\item If $\cK$ is irreducible, the multiplicity of $\cK$ in $L^2(S)$ is bounded above by both $\rdl(\cK)$ and $\rdr(\cK)$.
\end{enumerate}

Let $(\cK_i)_{i \in I}$ be a maximal family of inequivalent irreducible finite index $T$-bimodules that appear in $L^2(S)$. The map
\begin{equation}\label{eq.map}
\bigoplus_{i \in I} \Bigl( (L^2(S),\cK_i) \ot \cK_i^0 \Bigr) \recht \QN_S(T) : V \ot \xi \mapsto V(\xi)
\end{equation}
is an isomorphism of vector spaces. Here, $(L^2(S),\cK_i)$ denotes the space of $T$-bimodular bounded operators from $\cK_i$ to $L^2(S)$ and $\cK_i^0 \subset \cK_i$ is the subspace of $T$-bounded vectors. All tensor products and direct sums are algebraic. Also, by 5, the vector spaces $(L^2(S),\cK_i)$ are finite dimensional.
\end{lemma}
\begin{proof}
To prove 1-5, we may assume that $\cK$ is irreducible. Since $\cK$ is a finite index $T$-bimodule, we can choose a $T$-bimodular unitary operator
$$V : p(\C^n \ot L^2(T)) \recht \cK \; ,$$
where the left $T$-module structure on $p(\C^n \ot L^2(T))$ is given by left multiplication with $\psi(a)$, $a \in T$ and $\psi : T \recht p (M_n(\C) \ot T)p$ is a finite index inclusion. Define the elements $x_i \in \cK$ given by $x_i := V(p(e_i \ot 1))$. Define $\cV \in (\overline{\C^n} \ot L^2(S))p$ given by
$$\cV = \sum_{i=1}^n \overline{e_i} \ot x_i \; .$$
The $T$-bimodularity of $V$ means that $a \cV = \cV \psi(a)$ for all $a \in T$. In particular, $\cV = \cV p$. Then, $\cV \cV^*$ is an element of $L^1(S)$ that commutes with $T$. So, $\cV \cV^*$ is a multiple of $1$. Therefore, $\cV \in (\overline{\C^n} \ot S)p$ and $x_i \in S$ for all $i$.

1.\ View $V$ as a partial isometry from $\C^n \ot L^2(T)$ to $L^2(S)$ with initial projection $p$ and final projection $p_\cK$. A direct computation gives that
$$V^*(x) = \sum_{i=1}^n e_i \ot E(x_i^* x) \quad\text{for all}\;\; x \in S \; .$$
From this, 1 follows immediately.

2.\ This is analogous to 1.

3.\ Denote by $\cK^0 \subset \cK$ the space of $T$-bounded vectors. The inclusion $\cK \cap S \subset \cK^0$ is obvious. On the other hand, $\cK^0 = V(p(\C^n \ot T))$. Since all $x_i \in S$, it follows that $\cK^0 \subset S$.

4.\ Identifying
$$p(\C^n \ot L^2(T)) \ot_T L^2(S) = p(\C^n \ot L^2(S)) \; ,$$
the multiplication operator $\cK \ot_T L^2(S) \recht L^2(S)$ is the composition of the unitary operator
$$V^* \ot 1 : \cK \ot_T L^2(S) \recht p(\C^n \ot L^2(S))$$
and the bounded operator $p(\C^n \ot L^2(S)) \recht L^2(S)$ given by left multiplication with $\cV$.

5.\ Assume that the $T$-bimodule $p(\C^n \ot L^2(T))$ appears at least $k$ times in $L^2(S)$. We then find, for $i=1,\ldots,k$, $T$-bimodular isometries
$$V_i : p(\C^n \ot L^2(T)) \recht L^2(S)$$
with orthogonal ranges. The corresponding elements $\cV_i \in (\overline{\C^n} \ot S)p$ then satisfy
$$(\id \ot E)(\cV_j^* \cV_i) = \delta_{i,j} p \; .$$
Since $\cV_i \cV_j^*$ belongs to $T' \cap S = \C 1$, it follows that
$$\cV_i \cV_j^* = \delta_{i,j} (\Tr \ot \tau)(p) \, 1 = \delta_{i,j} \rdr(\cK) \, 1 \; .$$
The elements $\rdr(\cK)^{-1/2} \cV_i$ are thus partial isometries in $\overline{\C^n} \ot S$ with left support equal to $1$ and orthogonal right supports below $p$. It follows that $k \leq (\Tr \ot \tau)(p) = \rdr(\cK)$. By symmetry, also $k \leq \rdl(\cK)$.

To prove the remaining statement in the lemma, observe that the map in \eqref{eq.map} is injective and has its range in $\QN_S(T)$. When $x \in \QN_S(T)$, define $\cK$ as the closed linear span of $TxT$. Then, $\cK$ is a finite index $T$-bimodule and $x$ is a $T$-bounded vector in $\cK$. So, $x$ lies in the range of the map in \eqref{eq.map}.
\end{proof}

\subsection{Rank completions}

Let $(A,\tau)$ be a von Neumann algebra with a normal faithful tracial state and let $\cH$ be a (purely algebraic) $A$-bimodule. In \cite[Section 2]{T06}, the following quasi-metric is defined on $\cH$. For all $\xi,\eta \in \cH$, we put
$$
[\xi] := \inf\{ \tau(p) + \tau(q) \mid (1-p) \xi (1-q) = 0 \} \quad\text{and}\quad
\drank(\xi,\eta) := [\xi - \eta] \; .
$$
As explained in \cite[Section 2]{T06}, the separation/completion of $\cH$ w.r.t.\ the \emph{rank metric} $\drank$ is again an $A$-bimodule. It is called the \emph{rank completion} of $\cH$.

In the framework of quasi-regular inclusions $T \subset S$, we will use the rank completion w.r.t.\ $A = \cZ(T)$. The following lemma is then of crucial technical importance.

\begin{lemma}\label{lem.rank-cont}
Let $(S,\tau)$ be a tracial von Neumann algebra, $T \subset S$ a von Neumann subalgebra and $\cS$ a $*$-algebra with $T \subset \cS \subset \QN_S(T)$. Let $\cH$ be a (purely algebraic) $\cS$-bimodule. Consider the rank metric on $\cH$ viewed as a $\cZ(T)$-bimodule. Both the left and the right module actions of $\cS$ on $\cH$ are rank continuous. Hence, the rank completion of $\cH$ canonically is an $\cS$-bimodule.
\end{lemma}
\begin{proof}
Fix $x \in \cS$. By symmetry, it suffices to prove that the left action of $x$ on $\cH$ is a rank continuous map. Denote by $\cK \subset L^2(S)$ the Hilbert $T$-bimodule defined as the closed linear span of $TxT$. Since $x \in \QN_S(T)$, we see that $\cK$ is a bifinite Hilbert $T$-bimodule. By Proposition \ref{prop.decompose-irred}, we can decompose $\cK$ as the direct sum of finitely many irreducible $T$-bimodules $\cK_1,\ldots,\cK_\eta$. Each of these $\cK_j$ is an $\al_j$-$T$-bimodule, where $\al_j$ is a partial automorphism of $\cZ(T)$. It follows that for every projection $p \in \cZ(T)$, we have
\begin{equation}\label{eq.useful}
x p = p_0 x p \quad\text{where}\quad p_0 = \al_1(p) \vee \cdots \vee \al_\eta(p) \; .
\end{equation}

Fix $\eps > 0$. We construct $\delta > 0$ such that for every $\xi \in \cH$ with $[\xi] < \delta$, we have $[x \xi] < \eps$. Since the $\al_j$ are normal partial automorphisms, we can take $\delta_1 > 0$ such that for any projection $p \in \cZ(T)$ with $\tau(p) < \delta_1$, we have $\tau(\al_j(p)) < \eps / (2 \eta)$. Put $\delta = \min\{\delta_1,\eps/2\}$. Take $\xi \in \cH$ with $[\xi] < \delta$. We prove that $[x \xi] < \eps$.

Take projections $p,q \in \cZ(T)$ such that $\tau(p) + \tau(q) < \delta$ and $(1-p)\xi(1-q) = 0$. Define $p_0 = \al_1(p) \vee \cdots \vee \al_\eta(p)$. Since $\tau(p) < \delta_1$, we get that $\tau(p_0) < \eps/2$. From \eqref{eq.useful}, we get that $x p = p_0 x p$ and thus, $(1-p_0) x = (1-p_0)x (1-p)$. Therefore,
$$(1-p_0) x \xi (1-q) = (1-p_0)x (1-p)\xi (1-q) = 0 \; .$$
Since $\tau(p_0) < \eps/2$ and $\tau(q) < \eps/2$, we conclude that $[x \xi] < \eps$.
\end{proof}

\subsection[Rigid C$^*$-tensor categories]{\boldmath Rigid C$^*$-tensor categories}\label{subsec.categories}

Recall that a \emph{rigid C$^*$-tensor category} is a C$^*$-tensor category that is semisimple, with irreducible tensor unit $\eps \in \cC$ and with every object $\al \in \cC$ having an adjoint $\albar \in \cC$ that is both a left and a right dual of $\al$. For basic definitions and results on rigid C$^*$-tensor categories, we refer to \cite[Sections 2.1 and 2.2]{NT13}.

For $\al,\be \in \cC$, the finite dimensional Banach space of morphisms from $\al$ to $\be$ is denoted by $(\be,\al)$. Recall that $\End(\al) = (\al,\al)$ is a finite dimensional C$^*$-algebra. We denote the tensor product of $\al,\be \in \cC$ by juxtaposition $\al \be$. For every $\al \in \cC$, we choose a standard solution of the conjugate equations (in the sense of \cite{LR95}, see also \cite[Definition 2.2.12]{NT13}): $s_\al \in (\al \albar,\eps)$ and $t_\al \in (\albar \al,\eps)$ such that
\begin{equation}\label{eq.standard-sol}
(t_\al^* \ot 1) (1 \ot s_\al) = 1 \;\; , \;\; (s_\al^* \ot 1)(1 \ot t_\al)=1 \quad\text{and}\quad t_\al^*(1 \ot X) t_\al = s_\al^* (X \ot 1) s_\al
\end{equation}
for all $X \in \End(\al)$. These $s_\al,t_\al$ are unique up to unitary equivalence and the functional $\Tr_\al(X) = t_\al^*(1 \ot X) t_\al = s_\al^* (X \ot 1) s_\al$ on $\End(\al)$ is uniquely determined and tracial. The trace $\Tr_\al$ is non-normalized: $\Tr_\al(1) = \rd(\al)$, the categorical dimension of $\al$.

We also consider the partial traces
\begin{align*}
\Tr_\al \ot \id & : (\al \be, \al \gamma) \recht (\be,\gamma) : (\Tr_\al \ot \id)(S) = (t_\al^* \ot 1) (1 \ot S)(t_\al \ot 1) \;\; , \\
\id \ot \Tr_\al & : (\be \al , \gamma \al) \recht (\be,\gamma) : (\id \ot \Tr_\al)(S) = (1 \ot s_\al^*) (S \ot 1) (1 \ot s_\al) \;\; .
\end{align*}
Note that $\Tr_\be \circ (\Tr_\al \ot \id) = \Tr_{\al \be} = \Tr_\al \circ (\id \ot \Tr_\be)$ on $(\al \be,\al \be)$.

We denote by $\Irr(\cC)$ a set of representatives of all irreducible objects in $\cC$. The fusion $*$-algebra $\C[\cC]$ of $\cC$ has $\Irr(\cC)$ as a vector space basis, with
\begin{equation}\label{eq.fusion-alg}
\al \cdot \be = \sum_{\gamma \in \Irr(\cC)} \mult(\al \be,\gamma) \, \gamma \quad\text{and}\quad \al^\# = \albar \; ,
\end{equation}
where $\mult(\al \be,\gamma)$ denotes the multiplicity of $\gamma \in \Irr(\cC)$ in $\al \be$, i.e.\ the dimension of the vector space $(\al \be , \gamma)$.

Using the categorical trace, all spaces of morphisms $(\be,\al)$, for $\al,\be \in \cC$, are finite dimensional Hilbert spaces with scalar product
$$\langle V,W\rangle = \Tr_\be(V W^*) = \Tr_\al(W^* V) \quad\text{for all}\;\; V,W \in (\be,\al) \; .$$
We denote by $\onb(\be,\al)$ any choice of orthonormal basis of the Hilbert space $(\be,\al)$. When $\al \in \Irr(\cC)$ and $V,W \in (\be,\al)$, we have that $W^* V \in (\al,\al) = \C 1$. Therefore, for all $\be \in \cC$, we have
\begin{equation}\label{eq.onb}
\sum_{\al \in \Irr(\cC)} \sum_{V \in \onb(\be,\al)} \rd(\al) V V^* = 1 \; .
\end{equation}

\subsection{Subfactors, their standard invariant and symmetric enveloping algebra}\label{subsec.SEinclusion}

Let $N \subset M$ be an inclusion of II$_1$ factors with finite Jones index, $[M:N] < \infty$. Let $N \subset M \subset M_1 \subset \cdots$ be the associated Jones tower and use the convention that $M_0 = M$, $M_{-1} = N$. Recall that for all $n \geq 0$, $M_{n+1}$ is generated by $M_n$ and the Jones projection $e_n : L^2(M_n) \recht L^2(M_{n-1})$. The relative commutants $A_{ij} = M_i' \cap M_j$, $i \leq j$, form a lattice of multimatrix algebras called the \emph{standard invariant}. Together with the projections $e_n \in A_{ij}$, $i < n < j$, they form a \emph{$\lambda$-lattice} in the sense of \cite{Po94b}, with $\lambda = [M:N]^{-1}$.

Another axiomatization for the standard invariant of a subfactor is given by Jones  \cite{jones:planar}. Indeed, he showed that the axioms of a $\lambda$-lattice are equivalent to the existence of a \emph{planar algebra} structure on the linear spaces $A_{ij}$.  A key ingredient is the assignment to the isotopy invariance class of a {\em planar tangle} $T$ of a certain multi-linear map $Z_T$ between certain tensor products of these linear spaces.  We refer to  \cite{jones:planar} for details.

We also consider the C$^*$-tensor category $\cC$ of all $M$-bimodules that are isomorphic to a finite direct sum of $M$-subbimodules of $\bim{M}{L^2(M_n)}{M}$ for some $n$. Note that $\cC$ is a rigid C$^*$-tensor category.

For every extremal\footnote{This means that the natural anti-isomorphism between $M' \cap M_1$ and $N' \cap M$ is trace preserving. This is equivalent with all $M$-subbimodules of $L^2(M_n)$ having equal left and right $M$-dimension.} finite index subfactor $N \subset M$, we consider the symmetric enveloping (SE) algebra $S = M \boxtimes_{e_N} M\op$ introduced in \cite{Po94a,Po99}. By \cite{Po94a,Po99}, $S$ is the unique tracial von Neumann algebra generated by commuting copies of $M$ and $M\op$ together with an orthogonal projection $e_N$ that serves as the Jones projection for both $N \subset M$ and $N\op \subset M\op$. Writing $T = M \ovt M\op$, we refer to $T \subset S$ as the \emph{SE-inclusion} for the subfactor $N \subset M$. By \cite{Po99}, $T \subset S$ is irreducible and quasi-regular.

Denote by $\cC$ the C$^*$-tensor category of $M$-bimodules generated by $N \subset M$ as above. By \cite{Po99}, we have
\begin{equation}\label{eq.decomp}
L^2(S) = \bigoplus_{\al \in \Irr(\cC)} \bigl(\cH_\al \ot \overline{\cH_\al}\bigr)
\end{equation}
as $T$-bimodules. Given any C$^*$-tensor category $\cC$ of finite index $M$-bimodules having equal left and right dimension, one can define the SE-inclusion $T \subset S$ with $T = M \ovt M\op$ and such that \eqref{eq.decomp} holds; see \cite{LR94,Ma99} and see also \cite[Remark 2.7]{PV14}.

\section{\boldmath The tube $*$-algebra of an irreducible quasi-regular inclusion} \label{sec.tube-algebras}

Let $N \subset M$ be an extremal finite index subfactor with associated SE-inclusion $T \subset S$ (see Section \ref{subsec.SEinclusion}). In \cite{PV14}, the representation theory of the standard invariant of $N \subset M$ was defined as the class of \emph{SE-correspondences}, i.e.\ $S$-bimodules $\cH$ that are generated by $T$-central vectors. It was shown that this representation theory only depends on the standard invariant. Denoting by $\cC$ the tensor category of $M$-bimodules generated by $N \subset M$, the notions of an \emph{admissible state} on the fusion $*$-algebra $\C[\cC]$ (see \eqref{eq.fusion-alg}) and an \emph{admissible representation} of $\C[\cC]$ were defined in \cite{PV14} and characterized purely in terms of $\cC$ as a rigid C$^*$-tensor category. It was proved that there is a canonical bijection between SE-correspondences and admissible representations of $\C[\cC]$.

In \cite{NY15a}, a more categorical point of view on this representation theory was given. For any rigid C$^*$-tensor category $\cC$, the notion of a unitary half braiding on an ind-object of $\cC$ was defined (see Section \ref{sec.rep-tube-vs-half-braiding} for details). In the case where $\cC$ is the category of finite index $M$-bimodules generated by an extremal subfactor $N \subset M$ with associated SE-inclusion $T \subset S$, it was proved in \cite{NY15a} that there is a canonical bijection between the class of these unitary half-braidings and the \emph{generalized SE-correspondences}, i.e.\ the $S$-bimodules $\cH$ that, as a $T$-bimodule, are a direct sum of $T$-bimodules of the form $\cH_\al \ot \overline{\cH_\be}$, $\al,\be \in \cC$ (recall that $T = M \ovt M\op$). In this picture, one should think of the SE-correspondences as the \emph{spherical} part of the representation theory given by all generalized SE-correspondences.

In \cite{GJ15}, the representation theory of a rigid C$^*$-tensor category has been developed further and linked to the \emph{tube $*$-algebra} $\cA$ of Ocneanu \cite{Oc93}. This $*$-algebra $\cA$, whose construction is recalled in Section \ref{sec.tube-tensor} below, comes with a family of projections $(p_i)_{i \in \Irr(\cC)}$ and a canonical isomorphism $p_\eps \cdot \cA \cdot p_\eps \cong \C[\cC]$. It was proved in \cite{GJ15} that a state $\om$ on $\C[\cC]$ is admissible if and only if $\om$ remains positive on $\cA$.

The main parts of this section are \ref{sec.tube-quasireg} and \ref{sec.rep-tube-quasireg-vs-bimodules}. Inspired by Ocneanu's tube algebra of a tensor category (see \cite{Oc93}) and the above connection between representations of the tube algebra and SE-correspondences, we define a tube $*$-algebra $\cA$ for an \emph{arbitrary} irreducible quasi-regular inclusion $T \subset S$ of II$_1$ factors, see Section \ref{sec.tube-quasireg}. In the special case where $T \subset S$ is an SE-inclusion, our tube algebra is Morita equivalent with Ocneanu's, see Proposition \ref{prop.morita}.

We actually define the tube $*$-algebra $\cA$ for an irreducible quasi-regular inclusion $T \subset S$ together with a choice of tensor category $\cC$ of finite index $T$-bimodules containing all finite index $T$-subbimodules of $L^2(S)$. A canonical choice for $\cC$ is of course the tensor category generated by all finite index $T$-subbimodules of $L^2(S)$, but it is convenient to also allow larger choices of $\cC$. In Section \ref{sec.rep-tube-quasireg-vs-bimodules}, we construct a canonical bijection between Hilbert space representations of the tube $*$-algebra $\cA$ and Hilbert $S$-bimodules $\cH$ that, as a $T$-bimodule, are a direct sum of $T$-bimodules in $\cC$. In this way, the $*$-algebra $\cA$ exactly encodes the $S$-bimodules that are ``discrete'' as a $T$-bimodule (i.e.\ a direct sum of finite index $T$-subbimodules).

In the second part of this section, we unify and complete the different pictures of the representation theory mentioned above. For general rigid C$^*$-tensor categories $\cC$, we construct in Section \ref{sec.rep-tube-vs-half-braiding} a canonical bijection between Hilbert space representations of the tube $*$-algebra $\cA$ of $\cC$ and unitary half braidings for $\cC$ in the sense of \cite{NY15a}. When $\cC$ is a category of $M$-bimodules with corresponding SE-inclusion $T \subset S$, we prove in Section \ref{sec.tube-tensor} that there is a canonical bijection between generalized SE-correspondences and Hilbert space representations of the tube $*$-algebra (see Corollary \ref{cor.corr-vs-tube-category}). Finally, in Section \ref{sec.affine-cat}, we explain the relation with the approach of \cite{jones:annular}, where representations of a planar algebra (i.e.\ standard invariant of a subfactor) are viewed as Hilbert space representations of the associated affine category.

\subsection[Construction of the tube $*$-algebra]{\boldmath Construction of the tube $*$-algebra}\label{sec.tube-quasireg}

Let $S$ be a II$_1$ factor and $T \subset S$ an irreducible quasi-regular subfactor.
Given Hilbert $T$-bimodules $\cH_1, \cH_2$, we say that a $T$-bimodular bounded operator $V : \cH_2 \recht \cH_1$ has \emph{finite rank} if the closure of $V(\cH_2)$ is a finite index $T$-bimodule. We denote the vector space of these finite rank $T$-bimodular operators as $(\cH_1,\cH_2)$.

Fix a tensor category $\cC$ of finite index $T$-bimodules containing all finite index $T$-subbimodules of $L^2(S)$. Realize every $i \in \Irr(\cC)$ as an irreducible $T$-bimodule $\cH_i$. Write $\cS = \QN_S(T)$. With some abuse of notation, we denote for all $i,j \in \Irr(\cC)$,
$$(i \cS, \cS j) := (\cH_i \ovt_T L^2(S) , L^2(S) \ovt_T \cH_j) \; .$$
For every finite subset $\cF \subset \Irr(\cC)$, denote by $e_\cF$ the orthogonal projection of $L^2(S)$ onto the sum of the $T$-subbimodules of $L^2(S)$ that are equivalent with one of the $\cH_i$, $i \in \cF$. By Lemma \ref{lem.well-behaved}, every $e_\cF$ has finite rank. Also, a bounded $T$-bimodular operator $V : L^2(S) \ovt_T \cH_j \recht \cH_i \ovt_T L^2(S)$ has finite rank if and only if there exists a finite subset $\cF \subset \Irr(\cC)$ satisfying $V = V(e_\cF \ot 1) = (1 \ot e_\cF) V$.

We can then define the tube $*$-algebra $\cA$ associated with $T \subset S$ and $\cC$. As a vector space, $\cA$ is defined as the algebraic direct sum
$$\cA = \bigoplus_{i,j \in \Irr(\cC)} (i \cS, \cS j) \; .$$
The product of $V \in (i \cS, \cS k)$ and $W \in (k' \cS, \cS j)$ is denoted by $V \cdot W$, belongs to $(i \cS, \cS j)$ and is defined as
\begin{equation}\label{eq.prod-A}
\delta_{k,k'} \; (1 \ot m) (V \ot 1)(1 \ot W) (m^* \ot 1) \; .
\end{equation}
Here, $m : S \ot_T S \recht S$ denotes the multiplication map and $m^*$ is its adjoint w.r.t.\ the Hilbert space structures $L^2 (S) \ovt_T L^2(S)$ and $L^2(S)$. Since $m$ need not extend to a bounded operator from $L^2(S) \ovt_T L^2(S)$ to $L^2(S)$, one has to be careful in the interpretation of \eqref{eq.prod-A}. But since $V$ and $W$ are finite rank intertwiners, we can take a finite subset $\cF \subset \Irr(\cC)$ such that $W = W (e_\cF \ot 1)$ and $V = (1 \ot e_\cF) V$. By Lemma \ref{lem.well-behaved}, we have that $m(e_\cF \ot 1)$ and $m(1 \ot e_\cF)$ are bounded $T$-bimodular operators from $L^2(S) \ovt_T L^2(S)$ to $L^2(S)$. So, the expression in \eqref{eq.prod-A} is a well defined finite rank intertwiner.

The associativity of the product map gives us the associativity of the product on $\cA$.

We now define the adjoint operation on $\cA$. Denote by $\delta : L^2(T) \recht L^2(S)$ the inclusion map. Then denote $a = m^* \delta$. Again, $a$ need not be a well defined intertwiner from $L^2(T)$ to $L^2(S) \ovt_T L^2(S)$. But, whenever $\cF \subset \Irr(\cC)$ is a finite subset, we have that $(e_\cF \ot 1) a = (1 \ot e_{\overline{\cF}}) a$ is well defined and given by
$$(e_\cF \ot 1) a = \sum_{i=1}^n x_i \ot x_i^* \;\; ,$$
where $x_1,\ldots,x_n$ is a basis of $e_\cF (L^2(S))$ as a right $T$-module (see Lemma \ref{lem.well-behaved}.1). One checks that
$$(a^* \ot 1)(1 \ot a) = 1 \;\; ,$$
which rigorously speaking only makes sense after multiplying with $e_\cF$ for an arbitrary finite subset $\cF \subset \Irr(\cC)$.

The adjoint of $V \in (i \cS,\cS j)$ is denoted by $V^\#$, belongs to $(j \cS, \cS i)$ and is defined as
$$V^\# = (a^* \ot 1 \ot 1)(1 \ot V^* \ot 1)(1 \ot 1 \ot a) \; .$$

The fundamental properties of $m$, $a$ and $\delta$ can be summarized as:
\begin{equation}\label{eq.m-a}
\begin{alignedat}{2}
& m(1 \ot m) = m(m \ot 1) \quad , \quad & & (a^* \ot 1)(1 \ot m^*) = m = (1 \ot a^*) (m^* \ot 1) \quad , \\
& m(1 \ot \delta) = 1 = m(\delta \ot 1) \quad , \quad & & (a^* \ot 1) (1 \ot a) = 1 = (1 \ot a^*)(a \ot 1) \; .
\end{alignedat}
\end{equation}
As above, these formulas only make sense after multiplication with enough projections $e_\cF$, $\cF \subset \Irr(\cC)$ finite.

Using \eqref{eq.m-a}, one easily checks that $\cA$ is a $*$-algebra. When $\Irr(\cC)$ is infinite, the $*$-algebra $\cA$ is non-unital. But, for every $i \in \Irr(\cC)$, the element $(1 \ot \delta)(\delta^* \ot 1) \in (i \cS, \cS i)$ is a self-adjoint projection in $\cA$ that we denote as $p_i$. Note that $(i \cS, \cS j) = p_i \cdot \cA \cdot p_j$. So, $\cA$ always has enough self-adjoint idempotents.

When $\cK$ is a Hilbert $T$-bimodule that is a direct sum of finite index $T$-bimodules, then the algebra $(\cK,\cK)$ of finite rank intertwiners has two natural faithful traces:
$$\Tr^\ell_\cK(W) = \sum_i \langle W(\xi_i) , \xi_i \rangle \quad\text{and}\quad \Tr^r_\cK(W) = \sum_j \langle W(\eta_j),\eta_j \rangle \;\;,$$
where the $\xi_i$, resp.\ $\eta_j$, form an orthonormal basis of $\cK$ as a left, resp.\ right, $T$-module. We have $\Tr^\ell_\cK(1) = \rdl(\cK)$ and $\Tr^r_\cK(1) = \rdr(\cK)$. We have $\Tr^\ell = \Tr^r$ if and only if all subbimodules of $\cK$ have equal left and right dimension. We denote by $\Delta_\cK$ the positive, self-adjoint, but generally unbounded, operator on $\cK$ such that $$\Tr^r(\, \cdot \,) = \Tr^\ell( \Delta_\cK \, \cdot \, ) \; .$$
For every finite index intertwiner $V \in (\cK , \cK')$, we have that $\Delta_\cK V$ and $V \Delta_{\cK'}$ are equal and bounded. When $\cK$ is an irreducible finite index $T$-bimodule, $(\cK,\cK)$ is one-dimensional and $\Delta_\cK$ equals the ratio $\rdr(\cK)/\rdl(\cK)$ between the right and left $T$-dimension of $\cK$.

In particular, we consider the positive self-adjoint, but generally unbounded, operator $\Delta_\cS$ on $L^2(S)$. For every finite subset $\cF \subset \Irr(\cC)$, we have that $\Delta_\cS e_\cF$ is bounded and given by
$$\Delta_\cS = \sum_{\al \in \cF} \Delta_\al e_{\{\al\}} \; .$$

Since intertwiner spaces have a left and a right trace, we also have a left and a right scalar product on all our intertwiner spaces, defined as
$$\langle V, W \rangle_\ell = \Tr^\ell_\cK(V W^*) = \Tr^\ell_\cH(W^* V) \quad , \quad \langle V, W \rangle_r = \Tr^r_\cK(V W^*) = \Tr^r_\cH(W^* V)$$
for all $V, W \in (\cK,\cH)$.

Finally, note that
\begin{equation}\label{eq.trace-left-right}
\Tr^\ell(V) = a^* (1 \ot V) a \quad\text{and}\quad \Tr^r(V) = a^* (V \ot 1) a \quad\text{for all}\;\; V \in (\cS,\cS) \; .
\end{equation}

The following lemma implies that every $*$-representation of $\cA$ on a pre-Hilbert space is automatically by bounded operators, and that $\cA$ has a universal enveloping C$^*$-algebra.

\begin{lemma}\label{lem.other-sum}
For every $i \in \Irr(\cC)$ and every finite subset $\cF \subset \Irr(\cC)$, we have
\begin{equation}\label{eq.other-sum}
\sum_{j \in \Irr(\cC)} \;\;\; \sum_{W \in \onb_\ell((1 \ot e_\cF)(i \cS , \cS j) (e_\cF \ot 1))} \;\; \rdl(j) \; W \cdot W^\# = \rdr(e_\cF(L^2(S)))^2 \; p_i \; .
\end{equation}
\end{lemma}

Here, we denote by $\onb_\ell$ any choice of orthonormal basis w.r.t.\ the left scalar product. Also note that the sum in \eqref{eq.other-sum} only has finitely many terms~: since $\cF$ is finite and $i$ is fixed, there are only finitely many $j \in \Irr(\cC)$ for which $(1 \ot e_\cF)(i \cS , \cS j) (e_\cF \ot 1)$ is non-zero, and each of these is a finite dimensional Hilbert space.

\begin{proof}
Note that the map
$$(e_{\overline{\cF}} \ot 1 \ot e_\cF)(\cS i \cS , j) \recht (1 \ot e_\cF)(i \cS , \cS j) (e_\cF \ot 1) : W \mapsto (a^* \ot 1 \ot 1)(1 \ot W)$$
is a unitary w.r.t.\ the left scalar products and that $\bigl((a^* \ot 1 \ot 1)(1 \ot W)\bigr)^\# = (W^* \ot 1)(1 \ot 1 \ot a)$. Therefore, the left hand side of \eqref{eq.other-sum} equals
\begin{align*}
\sum_{j \in \Irr(\cC)} \;\;\;  & \sum_{W \in \onb_\ell ((e_{\overline{\cF}} \ot 1 \ot e_\cF)(\cS i \cS , j))} \;\; \rdl(j) \; (a^* \ot 1 \ot m)(1 \ot WW^* \ot 1) (m^* \ot 1 \ot a) \\
& = (a^* \ot 1 \ot m)(1 \ot e_{\overline{\cF}} \ot 1 \ot e_\cF \ot 1) (m^* \ot 1 \ot a) \\
& = \rdr(e_\cF(L^2(S)))^2 \;  (1 \ot \delta) (\delta^* \ot 1) = \rdr(e_\cF(L^2(S)))^2 \; p_i \;\; ,
\end{align*}
because $m(e_\cF \ot 1)a = m(1 \ot e_{\overline{\cF}}) a = \rdr(e_\cF(L^2(S))) \delta$.
\end{proof}

The $*$-algebra $\cA$ has the following natural weight $\tau : \cA \recht \C$ with corresponding von Neumann algebra completion $\cA\dpr$ acting on $L^2(\cA)$. In the unimodular case, i.e.\ when all $T$-subbimodules of $L^2(S)$ have equal left and right dimension, $\tau$ is a trace and $\cA\dpr$ is a semifinite von Neumann algebra.

\begin{proposition}\label{prop.trace-tube-quasi-reg}
Let $S$ be a II$_1$ factor, $T \subset S$ an irreducible quasi-regular subfactor and $\cC$ a tensor category of finite index $T$-bimodules containing all finite index $T$-subbimodules of $L^2(S)$. Define the $*$-algebra $\cA$ as above. The linear map
$$\tau : \cA \recht \C : \tau(V) = \sum_{i \in \Irr(\cC)} \Tr^\ell_i((1 \ot \delta^*) V_{ii} (\delta \ot 1))$$
is a faithful positive functional on $\cA$. Denote by $L^2(\cA)$ the completion of $\cA$ w.r.t.\ the norm $\|V\|_{2,\tau} = \sqrt{\tau(V^\# \cdot V)}$. Left multiplication extends to a $*$-representation of $\cA$ by bounded operators on $L^2(\cA)$ and $\tau$ extends uniquely to a normal semifinite faithful weight on $\cA\dpr$ with modular automorphism group
$$\sigma^\tau_t(V) =  (1 \ot \Delta_\cS^{it}) V \quad\text{for all}\;\; V \in (i \cS, \cS j) \; .$$
\end{proposition}
\begin{proof}
Take $V,W \in (i\cS, \cS j)$. A direct computation yields
$$\tau(V^\# \cdot W) = \Tr^\ell_{\cS j}(V^*W) \quad\text{and}\quad \tau(W \cdot V^\#) = \Tr^\ell_{\cS j}((\Delta_\cS \ot 1) V^* W) \; .$$
By Lemma \ref{lem.other-sum}, the representation of $\cA$ on $L^2(\cA)$ is indeed by bounded operators. The remaining statements follow by standard methods of modular theory.
\end{proof}

The following definition is now a natural one and corresponds exactly to the case where $\tau$ is a trace.

\begin{definition}\label{def.inclusion-unimodular}
Let $S$ be a II$_1$ factor and $T \subset S$ an irreducible quasi-regular subfactor. We say that the inclusion $T \subset S$ is \emph{unimodular} when all $T$-subbimodules of $L^2(S)$ have equal left and right dimension.
\end{definition}

\subsection[Representations of the tube $*$-algebra and Hilbert bimodules]{\boldmath Representations of the tube $*$-algebra and Hilbert bimodules}\label{sec.rep-tube-quasireg-vs-bimodules}

We say that a Hilbert space $\cK$ is a \emph{right Hilbert $\cA$-module} when we are given a $*$-anti-homomorphism from $\cA$ to $B(\cK)$. We denote the right action of $V \in \cA$ on $\xi \in \cK$ as $\xi \cdot V$. We say that $\cK$ is \emph{nondegenerate} when $\cK \cdot \cA$ has dense linear span in $\cK$. Note that $\cK$ is nondegenerate if and only if the linear span of the subspaces $\cK \cdot p_i$, $i \in \Irr(\cC)$, is dense in $\cK$.

\begin{theorem}\label{thm.tube-vs-corr}
Let $S$ be a II$_1$ factor, $T \subset S$ an irreducible quasi-regular subfactor and $\cC$ a tensor category of finite index $T$-bimodules containing all finite index $T$-subbimodules of $L^2(S)$. Let $\cA$ be the associated tube $*$-algebra. The formulas below provide a natural bijection between
\begin{itemize}
\item Hilbert $S$-bimodules $\cH$ that, as a $T$-bimodule, are a direct sum of $T$-bimodules contained in $\cC$~;
\item nondegenerate right Hilbert $\cA$-modules.
\end{itemize}
\end{theorem}

Given a Hilbert $S$-bimodule $\cH$ that, as a $T$-bimodule, is a direct sum of $T$-bimodules contained in $\cC$, define for all $i \in \Irr(\cC)$, the space $\cK_i := (\cH,i)$ and turn $\cK_i$ into a Hilbert space using the right scalar product $\langle \xi,\eta \rangle = \Tr^r_i(\eta^* \xi)$. Denote $m_{lr} : \cS \ot_T \cH \ot_T \cS \recht \cH : m_{lr}(x \ot \xi \ot y) = x \cdot \xi \cdot y$. Then,
\begin{equation}\label{eq.Amodule}
\xi \cdot V = m_{lr} (1 \ot \xi \ot 1) (1 \ot V (\Delta_\cS^{1/2} \ot 1)) (a \ot 1)
\end{equation}
for all $V \in (i \cS, \cS j)$ and $\xi \in \cK_i = (\cH,i)$, turns the direct sum $\cK = \oplus_{i \in \Irr(\cC)} \cK_i$ into a nondegenerate right Hilbert $\cA$-module.

Given a nondegenerate right Hilbert $\cA$-module $\cK$, denote $\cK_i := \cK \cdot p_i$ for all $i \in \Irr(\cC)$ and define $\cH^0$ as the algebraic direct sum of all $\cK_i \ot \cH^0_i$, $i \in \Irr(\cC)$, where $\cH^0_i$ is the set of $T$-bounded vectors in the irreducible $T$-bimodule $\cH_i$. The formulas
\begin{equation}\label{eq.Sbimodule}
\begin{split}
& (\xi \ot \mu) \cdot x = \sum_{j \in \Irr(\cC)} \; \sum_{U \in \onb_r(i \cS,j)} \; \rdr(j) \;\; \xi \cdot (U(\delta^* \ot 1)) \ot U^*(\mu \ot x) \;\; , \\
& x \cdot (\xi \ot \mu) = \sum_{j \in \Irr(\cC)} \; \sum_{U \in \onb_\ell(j , \cS i)} \; \rdl(j) \;\; \xi \cdot ((1 \ot \delta)U(\Delta_\cS^{1/2} \ot 1))^\# \ot U(x \ot \mu)
\end{split}
\end{equation}
for all $\xi \in \cK_i , \mu \in \cH^0_i$ and $x \in \cS$, together with the scalar product
$$\langle \xi_1 \ot \mu_1 , \xi_2 \ot \mu_2 \rangle = \frac{1}{\rdr(i)} \; \delta_{i,j} \; \langle \xi_1,\xi_2 \rangle \; \langle \mu_1,\mu_2 \rangle$$
turn the Hilbert space completion $\cH$ of $\cH^0$ into a well defined Hilbert $S$-bimodule that, as a $T$-bimodule, is a direct sum of copies of $\cH_i$, $i \in \Irr(\cC)$, with $(\cH,i) = \cK_i$.

\begin{proof}
Given a Hilbert $S$-bimodule $\cH$ and defining $\cK^0$ as the algebraic direct sum of the Hilbert spaces $\cK_i  := (\cH,i)$, a slightly tedious, but straightforward computation shows that \eqref{eq.Amodule} defines a $*$-anti-representation of $\cA$ on $\cK^0$. By Lemma \ref{lem.other-sum}, this anti-representation is by bounded operators on the Hilbert space completion $\cK$ of $\cK^0$ and we have found a nondegenerate right Hilbert $\cA$-module $\cK$.

Conversely, assume that $\cK$ is a nondegenerate right Hilbert $\cA$-module and define the pre-Hilbert space $\cH^0$ as above. It is again straightforward but slightly tedious to check that the formulas \eqref{eq.Sbimodule} turn $\cH^0$ into an $\cS$-bimodule satisfying
$$\langle x \cdot \mu \cdot y , \mu' \rangle = \langle \mu, x^* \cdot \mu' \cdot y^* \rangle$$
for all $x,y \in \cS$ and $\mu,\mu' \in \cH^0$. In order to prove that we can uniquely extend this to a Hilbert $S$-bimodule structure on the Hilbert space completion $\cH$ of $\cH^0$, it suffices to prove that for all $i , j \in \Irr(\cC)$, $\xi \in \cK_i$, $\xi' \in \cK_j$, $\mu \in \cH^0_i$ and $\mu' \in \cH^0_j$, the linear functionals
$$\cS \recht \C : x \mapsto \langle (\xi \ot \mu) \cdot x , \xi' \ot \mu' \rangle  \quad\text{and}\quad x \mapsto \langle x \cdot (\xi \ot \mu), \xi' \ot \mu' \rangle$$
extend to normal functionals on $S$. By symmetry, we only consider the first functional. It follows from \eqref{eq.Sbimodule} that it is a finite linear combination of functionals of the form
\begin{equation}\label{eq.my-functional}
x \mapsto \langle U^*(\mu \ot x) , \mu' \rangle
\end{equation}
with $\mu \in \cH^0_i$, $\mu' \in \cH^0_j$ and $U \in (i \cS,j)$. Since $\mu$ is a bounded vector, we can define the bounded operator $L_\mu : L^2(S) \recht \cH_i \ovt_T L^2(S)$ given by $L_\mu(x) = \mu \ot x$ for all $x \in S$. It follows that
$$\langle U^*(\mu \ot x) , \mu' \rangle = \langle x, L_\mu^*(U(\mu')) \rangle \; .$$
Since $L_\mu^*(U(\mu')) \in L^2(S)$, the functional in \eqref{eq.my-functional} is indeed normal.

By construction, the above correspondence between Hilbert $S$-bimodules and Hilbert $\cA$-mod\-ules is indeed bijective, in the sense of the theorem.
\end{proof}

Given an irreducible quasi-regular inclusion of II$_1$ factors $T \subset S$, we have two natural $S$-bimodules: the trivial $S$-bimodule $L^2(S)$ and the family of coarse $S$-bimodules
$L^2(S) \ovt_T L^2(S)$, as well as $L^2(S) \ovt_T \cH \ovt_T L^2(S)$ for an arbitrary $T$-bimodule $\cH$ that is a direct sum of finite index $T$-bimodules. Through Theorem \ref{thm.tube-vs-corr}, they correspond to the following representations of the tube algebra. The proof of this lemma is given by a direct computation.

\begin{lemma} \label{lem.reg-triv-A}
Let $T \subset S$ be an irreducible quasi-regular inclusion of II$_1$ factors and $\cC$ a tensor category of finite index $T$-bimodules containing all finite index $T$-subbimodules of $L^2(S)$. Denote by $\cA$ the associated tube $*$-algebra.

Under the bijection of Theorem \ref{thm.tube-vs-corr},
\begin{enumerate}
\item the $S$-bimodule $L^2(S) \ovt_T L^2(S)$ corresponds to the right Hilbert $\cA$-module $L^2(p_\eps \cdot \cA)$, where $L^2(\cA)$ is given by Proposition \ref{prop.trace-tube-quasi-reg} and the right action of $W \in \cA$ on $L^2(p_\eps \cdot \cA)$ is given by right multiplication with $\sigma^\tau_{-i/2}(W)$~;
\item given a $T$-bimodule $\cH$ that is a direct sum of $T$-bimodules in $\cC$, the $S$-bimodule $L^2(S) \ovt_T \cH \ovt_T L^2(S)$ corresponds to the right Hilbert $\cA$-module $\bigoplus_{i \in \Irr(\cC)} (\cH,i) \ot L^2(p_i \cdot \cA)$~; in particular, with $\cH = \bigoplus_{i \in \Irr(\cC)} \cH_i$, we find the right Hilbert $\cA$-module $L^2(\cA)$~;
\item the $S$-bimodule $L^2(S)$ corresponds to the right Hilbert $\cA$-module defined by completing
$$\cE^r := \bigoplus_{i \in \Irr(\cC)} (\cS,i)$$
w.r.t.\ the left scalar product on $(\cS,i)$ and right $\cA$-module structure given by
$$\xi \cdot V = m(1 \ot m) (1 \ot (\xi \ot \Delta_\cS^{1/2}) V) (a \ot 1)$$
for all $\xi \in (\cS,i)$ and $V \in (i\cS, \cS j)$.
\end{enumerate}
\end{lemma}

\begin{remark} \label{rem.left-triv}
\begin{enumerate}
\item The right $\cA$-module $\cE^r$ should be considered as the trivial representation of $\cA$. Its adjoint is the left $\cA$-module
$$\cE^\ell := \bigoplus_{i \in \Irr(\cC)} (i, \cS)$$
with left $\cA$-module structure given by
$$V \cdot \xi = (1 \ot a^*) ((1 \ot \Delta_\cS^{-1/2})V \ot 1) (1 \ot \xi \ot 1) (m^* \ot 1) m^*$$
for all $V \in (i \cS, \cS j)$ and $\xi \in (j,\cS)$.

\item Also this left $\cA$-module $\cE^\ell$ can be completed into a Hilbert $\cA$-module by using the left scalar product on each $(i,\cS)$.

\item In Remark \ref{rem.reg-triv-as-cp-maps-and-states}, we will see that $\cE^\ell$ and $\cE^r$ can also be viewed as the GNS-spaces of $\cA$ w.r.t.\ a canonical state on $\cA$.
\end{enumerate}
\end{remark}

\begin{corollary}\label{cor.cp-maps-vs-states-tube-algebra}
Let $S$ be a II$_1$ factor, $T \subset S$ an irreducible quasi-regular subfactor and $\cC$ a tensor category of finite index $T$-bimodules containing all finite index $T$-subbimodules of $L^2(S)$. Let $\cA$ be the associated tube $*$-algebra. Then \eqref{eq.cp-maps-vs-states-tube-algebra} below gives a bijection between
\begin{itemize}
\item unital, completely positive, trace preserving $T$-bimodular maps $\vphi : S \recht S$~;
\item states $\omega_\vphi$ on $\cA$ with the property that $\om_\vphi(p_\eps) = 1$.
\end{itemize}
This bijection is given by
\begin{equation}\label{eq.cp-maps-vs-states-tube-algebra}
\begin{split}
\om_\vphi(V) &= \Tr(\vphi \circ V) \quad\text{for all}\;\; V \in (\cS,\cS) \quad\text{and}\\
\om_\vphi(V) &= 0 \quad\text{when}\;\; V \in (i \cS, \cS j) \;\;\text{with $i \neq \eps$ or $j \neq \eps$.}
\end{split}
\end{equation}
Note that for all $V \in (\cS,\cS)$ and for every $T$-bimodular linear map $\vphi : S \recht S$, we can view $\vphi \circ V$ as a finite rank $T$-bimodular map, i.e.\ as an element of $(\cS,\cS)$. We denote by $\Tr$ the categorical trace on $(\cS,\cS)$ given by $\Tr(\,\cdot\,) = \Tr^r(\Delta_{\cS}^{-1/2} \, \cdot \,) = \Tr^\ell(\Delta_{\cS}^{1/2} \, \cdot \,)$.
\end{corollary}

Note that whenever $\vphi : S \recht S$ is a normal, completely positive, $T$-bimodular map, the irreducibility of $T \subset S$ implies that $\vphi(1) = \lambda 1$ and $\tau \circ \vphi = \lambda \, \tau$ for some $\lambda \geq 0$. It is therefore not restrictive to only consider unital, trace preserving maps.

\begin{proof}
Given a unital, completely positive, trace preserving $T$-bimodular map $\vphi : S \recht S$, define the $S$-bimodule $\cH$ as the separation/completion of $S \ot S$ w.r.t.\ the scalar product $\langle x \ot y, a \ot b \rangle = \tau(x \vphi(yb^*)a^*)$. Note that by construction, as a $T$-bimodule, $\cH$ is isomorphic with a direct sum of irreducible $T$-subbimodules of $L^2(S) \ovt_T L^2(S)$, which thus belong to $\cC$. By Theorem \ref{thm.tube-vs-corr}, we find a $*$-representation of $\cA$ on a Hilbert space $\cK$ and a unit vector $\xi_0 \in \cK \cdot p_\eps$ corresponding to the $T$-central vector $1 \ot 1 \in \cH$. Define $\om_\vphi$ as the vector state on $\cA$ given by $\xi_0$. A direct computation shows that \eqref{eq.cp-maps-vs-states-tube-algebra} holds.

Conversely, given a state $\om : \cA \recht \C$ with $\om(p_\eps) = 1$, combining the GNS-construction and Theorem \ref{thm.tube-vs-corr}, we find an $S$-bimodule $\cH$ and a $T$-central unit vector $\xi_1 \in \cH$. Denote by $\vphi$ the unique unital, completely positive, trace preserving $T$-bimodular map $\vphi : S \recht S$ satisfying $\langle x \cdot \xi_1 \cdot y , \xi_1 \rangle = \tau(x \vphi(y))$ for all $x,y \in S$. A direct computation shows that $\om_\vphi = \om$.
\end{proof}

\begin{remark}\label{rem.reg-triv-as-cp-maps-and-states}
The trivial and the regular representation of Lemma \ref{lem.reg-triv-A} can also be understood in the context of Corollary \ref{cor.cp-maps-vs-states-tube-algebra}. The identity map $S \recht S : x \mapsto x$, corresponds to the state $\counit : \cA \recht \C$ given by $\counit(V) = \Tr(V)$ for all $V \in (\cS,\cS)$ and $\counit(V) = 0$ if $V \in (i \cS, \cS j)$ with $i \neq \eps$ or $j \neq \eps$. Performing the GNS-construction with this state $\counit$, we obtain the right Hilbert $\cA$-module $\cE^r$ of Lemma \ref{lem.reg-triv-A}. Also note that $\counit$ is a character when restricted to $p_\eps \cdot \cA \cdot p_\eps$, but it is not a character on the entire $*$-algebra $\cA$.

The map $S \recht S : x \mapsto \tau(x) 1$ corresponds to the state $\cA \recht \C : V \mapsto \tau(p_\eps \cdot V \cdot p_\eps)$. Performing the GNS-construction with this state, we obtain the right Hilbert $\cA$-module $L^2(p_\eps \cdot \cA)$.
\end{remark}

\subsection[Ocneanu's tube $*$-algebra of a rigid C$^*$-tensor category]{\boldmath Ocneanu's tube $*$-algebra of a rigid C$^*$-tensor category}\label{sec.tube-tensor}

Let $\cC$ be a rigid C$^*$-tensor category. We recall  the construction of \emph{Ocneanu's tube $*$-algebra}, introduced in \cite{Oc93} when $\cC$ has only finitely many irreducible objects. As a vector space, $\cA$ is given as the algebraic direct sum
$$\cA = \bigoplus_{i,j,\al \in \Irr(\cC)} (i\al,\al j) \; .$$
So, an element $V \in \cA$ is given by elements $V^\al_{ij} \in (i \al, \al j)$, with only finitely many of these elements being nonzero. Whenever $V \in (i \al,\al j)$, we also view $V$ as an element of $\cA$ living in the corresponding direct summand. When $i,j \in \Irr(\cC)$ and $\be \in \cC$ is a not necessarily irreducible object, every $V \in (i \be,\be j)$ defines an element in $\cA$ concretely given by
\begin{equation}\label{eq.general-elements-cA}
V^\al_{kl} = \delta_{i,k} \; \delta_{j,l} \; \sum_{U \in \onb (\be,\al)} \rd(\al) \; (1 \ot U^*) V (U \ot 1) \; .
\end{equation}
Here, we use the same conventions for the orthonormal basis $\onb(\be,\al)$ as in \eqref{eq.onb}.

We then turn $\cA$ into a $*$-algebra:
\begin{align*}
V \cdot W = \delta_{k,k'} \; (V \ot 1)(1 \ot W) &\quad\text{if $V \in (i \al, \al k)$ and $W \in (k' \be, \be, j)$~;}\\
V^\# = (t_\al^* \ot 1 \ot 1) (1 \ot V^* \ot 1) (1 \ot 1 \ot s_\al) &\quad\text{if $V \in (i \al, \al j)$.}
\end{align*}
Note that $V^\# \in (j \albar,\albar i)$ when $V \in (i \al, \al j)$. To avoid confusion with composition and adjoints of morphisms, we systematically write the dot and use the symbol $^\#$ for the operations in $\cA$.

The identity morphism $1 \in (i \eps, \eps i)$, when viewed as an element of $\cA$, is denoted as $p_i$. Note that the $p_i$ are self-adjoint idempotents and that
$$p_i \cdot \cA \cdot p_j = \bigoplus_{\al \in \Irr(\cC)} (i \al , \al j)$$
as vector spaces.

Identifying $\al \in \Irr(\cC)$ with the identity map $1 \in (\eps \al , \al \eps)$, we get an isomorphism $p_\eps \cdot \cA \cdot p_\eps \cong \C[\cC]$, where $\C[\cC]$ is the fusion $*$-algebra of $\cC$ (see \eqref{eq.fusion-alg}).

The \emph{co-unit} $\counit : \cA \recht \C$ is the unital $*$-homomorphism given by $\counit(p_i) = 0$ for all $i \neq \eps$ and $\counit(\al) = \rd(\al)$ for all $\al \in \Irr(\cC)$ viewed as the identity map $1 \in (\eps \al , \al \eps)$.

The following lemma ensures purely algebraically that there is a universal C$^*$-norm on $\cA$, a fact that was proved already in \cite{GJ15}. The proof is identical to the proof of Lemma \ref{lem.other-sum}.

\begin{lemma}\label{lem.sum}
For all $i \in \Irr(\cC)$ and $\al \in \cC$, we have that
$$\sum_{j \in \Irr(\cC)} \sum_{W \in \onb(i \al, \al j)} \rd(j) \;  W \cdot W^\# = \rd(\al) p_i \; .$$
For every $*$-representation $\pi$ of $\cA$ as linear operators on a pre-Hilbert space $\cH$, we have that $\|\pi(V)\| \leq \rd(\al) \|V\|$ for all $i,j \in \Irr(\cC)$, $\al \in \cC$, $V \in (i\al,\al j)$. Here, $\|V\|$ denotes the operator norm of $V \in (i \al, \al j)$.
\end{lemma}

As in Proposition \ref{prop.trace-tube-quasi-reg}, we have a natural trace on the tube $*$-algebra $\cA$ with corresponding von Neumann algebra completion $\cA\dpr$.

\begin{proposition}\label{prop.tube-von-neumann-tensor}
The map
$$\tau : \cA \recht \C : \tau(V) = \sum_{i \in \Irr(\cC)} \Tr_i(V^\eps_{ii})$$
is a positive faithful trace on $\cA$~: $\tau(V \cdot W) = \tau(W \cdot V)$ and $\tau(V^\# \cdot V) \geq 0$ for all $V,W \in \cA$ with $\tau(V^\# \cdot V) = 0$ if and only if $V = 0$.

Denote by $L^2(\cA)$ the completion of $\cA$ w.r.t.\ the norm $\|V\|_{2,\tau} = \sqrt{\tau(V^\# \cdot V)}$. For every $V \in \cA$, left multiplication as well as right multiplication with $V$ extend to bounded operators on $L^2(\cA)$.

Denote by $\cA\dpr$ the von Neumann algebra generated by the left action of $\cA$ on $L^2(\cA)$. Then $\tau$ uniquely extends to a normal semifinite faithful trace on $\cA\dpr$.
\end{proposition}
\begin{proof}
A direct computation gives for all $i,j,\al,\be \in \Irr(\cC)$ and $V \in (i\al,\al j)$, $W \in (i \be, \be j)$ that
$$\tau(V \cdot W^\#) = \delta_{\be,\al} \; \frac{1}{\rd(\al)} \; \Tr_{i \al}(V W^*) \quad\text{and}\quad \tau(W^\# \cdot V) =  \delta_{\be,\al} \; \frac{1}{\rd(\al)} \; \Tr_{\al j}(W^* V) \; .$$
It follows that $\tau$ is a trace. The remaining statements follow from Lemma \ref{lem.sum}.
\end{proof}

\begin{remark} \label{rem.confusion}
Given $i,j,\al \in \Irr(\cC)$, we have considered $(i \al, \al j)$ as a Hilbert space using the scalar product $\langle V, W \rangle = \Tr_{i \al}(V W^*)$. Now, we can also view $(i \al, \al j)$ as a subspace of $\cA$ and thus, of $L^2(\cA)$. Then, the scalar product is scaled with the factor $\rd(\al)$.

When using the notation $\onb(i\al, \al j)$, we always refer to an orthonormal basis for the first mentioned scalar product. This is the most convenient, since we also use such orthonormal bases for arbitrary spaces of morphisms $(\be,\gamma)$ with $\be, \gamma \in \cC$.
\end{remark}

Assume now that $M$ is a II$_1$ factor and that $\cC$ is a tensor category of finite index $M$-bimodules having equal left and right dimension. Consider the SE-inclusion $T \subset S$ defined in Section \ref{subsec.SEinclusion}. We then have two tube $*$-algebras: Ocneanu's tube algebra of the tensor category $\cC$ that we recalled above and the tube algebra of the quasi-regular inclusion $T \subset S$ defined in Section \ref{sec.tube-quasireg}. We prove that both tube algebras are naturally strongly Morita equivalent.

\begin{proposition}\label{prop.morita}
Let $M$ be a II$_1$ factor and $\cC$ a tensor category of finite index $M$-bimodules having equal left and right dimension. Put $T = M \ovt M\op$ and let $\cC_1$ be the tensor category of $T$-bimodules generated by $\al \ot \overline{\be}$, $\al,\be \in \cC$. The formula \eqref{eq.concrete-Morita} below defines a Morita equivalence between Ocneanu's tube $*$-algebra associated with $\cC$ (as defined in this section) and the tube $*$-algebra associated with the quasi-regular SE-inclusion $T \subset S$ and $\cC_1$ (as defined in Section \ref{sec.tube-quasireg}).
\end{proposition}
\begin{proof}
Note that for every rigid C$^*$-tensor category and every set of objects $\cO \subset \cC$, we can define the $*$-algebra
$$\cA_\cO = \bigoplus_{i,j \in \cO} \;\; \bigoplus_{\al \in \Irr(\cC)} \;\; (i \al , \al j)$$
in exactly the same way as we defined the tube $*$-algebra $\cA$ in the beginning of this section. By construction, we have
$$\cA_\cO = \bigoplus_{i,j \in \Irr(\cC)} \bigl( K_i \ot p_i \cdot \cA \cdot p_j \ot \overline{K_j} \bigr)$$
where $K_i$ is the vector space given as the algebraic direct sum $K_i = \bigoplus_{k \in \cO} (k,i)$. So, when $\cO$ is large enough in the sense that for every $i \in \Irr(\cC)$, there exists a $k \in \cO$ with $(k,i) \neq \{0\}$, we get that $\cA_\cO$ is strongly Morita equivalent with $\cA$.

Returning to the context of Proposition \ref{prop.morita}, we put $\cO = \{\, \overline{\beta} \, \al \, \mid \, \al,\be \in \Irr(\cC) \, \}$. We denote by $\cA_1$ the tube $*$-algebra associated with the quasi-regular inclusion SE-inclusion $T \subset S$ and $\cC_1$. Note that $\Irr(\cC_1) = \{ \al \ot \overline{\be} \mid \al,\be \in \Irr(\cC)\}$. The construction of the SE-inclusion $T \subset S$ comes with canonical intertwiners $\delta_\eta \in (\cS, \eta \ot \overline{\eta})$ for every $\eta \in \Irr(\cC)$, see \cite[Remark 2.7]{PV14}.

Let $\al,\al',\be,\be',\eta,\eta' \in \Irr(\cC)$. Whenever $V \in (\al \eta, \eta' \al')$ and $W \in (\be \eta, \eta' \beta')$, the tensor product of $V$ and $\overline{W}$ defines a morphism $\theta(V,W)$, in the category $\cC_1$, from $(\eta' \ot \overline{\eta'})(\al' \ot \overline{\beta'})$ to $(\al \ot \overline{\beta})(\eta \ot \overline{\eta})$. There is a unique $*$-isomorphism $\Psi : \cA_1 \recht \cA_{\cO}$ given by
\begin{equation}\label{eq.concrete-Morita}
\Psi\bigl( \, (1 \ot \delta_\eta) \theta(V,W) (\delta_{\eta'}^* \ot 1) \, \bigr)
= (1 \ot V) (1^2 \ot s_{\beta'}^* \ot 1) (1 \ot W^* \ot 1^2) (t_\beta \ot 1^3)
\end{equation}
for all $\al,\al',\be,\be',\eta,\eta' \in \Irr(\cC)$, $V \in (\al \eta, \eta' \al')$ and $W \in (\be \eta, \eta' \beta')$. Note that the right hand side belongs to $(\overline{\beta} \al \eta , \eta \overline{\beta'} \al')$ and thus defines an element in $\cA_\cO$.

It is straightforward to check that $\Psi$ is indeed a $*$-isomorphism.
\end{proof}

Still assume that $M$ is a II$_1$ factor and that $\cC$ is a tensor category of finite index $M$-bimodules having equal left and right dimension, with associated SE-inclusion $T \subset S$.
Recall from the first two paragraphs of Section \ref{sec.tube-algebras} the notion of a generalized SE-correspondence. Combining Proposition \ref{prop.morita} and Theorem \ref{thm.tube-vs-corr}, we thus obtain the following result.

\begin{corollary}\label{cor.corr-vs-tube-category}
Let $M$ be a II$_1$ factor and $\cC$ a tensor category of finite index $M$-bimodules having equal left and right dimension. There is a natural bijection between generalized SE-correspondences of the SE-inclusion $T \subset S$ and nondegenerate $*$-representations of the tube $*$-algebra $\cA$ of $\cC$.
\end{corollary}

\subsection[Representations of the tube $*$-algebra and unitary half braidings]{\boldmath Representations of the tube $*$-algebra and unitary half braidings}\label{sec.rep-tube-vs-half-braiding}

Given a II$_1$ factor $M$ and a tensor category $\cC$ of finite index $M$-bimodules having equal left and right dimension, we have seen in Section \ref{sec.tube-tensor} two equivalent ways to express the associated representation theory: as generalized SE-correspondences for the SE-inclusion $T \subset S$ and as representations of the tube $*$-algebra $\cA$ of $\cC$.

In \cite{NY15a}, it was proved that there is a natural bijection between generalized SE-correspon\-dences and \emph{unitary half braidings on ind-objects} for $\cC$. Formally, an ind-object $X \in \indC$ is a possibly infinite direct sum of objects in $\cC$. Then, $\indC$ is again a C$^*$-tensor category and we refer to
\cite[Section 1.2]{NY15a} for a rigorous definition. Following \cite{NY15a}, a unitary half braiding $\sigma$ on an ind-object $X \in \indC$ is a family of unitary morphisms $\sigma_\al \in (X \al,\al X)$ satisfying
\begin{itemize}
\item naturality, meaning that $(1 \ot V) \sigma_\al = \sigma_\be (V \ot 1)$ for all $V \in (\be,\al)$~;
\item $\sigma_\eps = \id$~;
\item multiplicativity, meaning that $\sigma_{\al\be} = (\sigma_\al \ot 1) (1 \ot \sigma_\be)$.
\end{itemize}
Let $\sigma$ be a unitary half braiding on the ind-object $X \in \indC$. Since $\cC$ is a category of finite index $M$-bimodules, we can realize $\indC$ as the category of Hilbert $M$-bimodules $\cH$ that can be written as a direct sum of $M$-bimodules belonging to $\cC$. For every $\al \in \cC$, we have the $M$-bimodular unitary operator
$$\sigma_\al : \cH_\al \ot_M X \recht X \ot_M \cH_\al \; .$$
Since $L^2(S)$ is the direct sum of the $T$-bimodules $\cH_\al \ot \overline{\cH_\al}$, $\al \in \Irr(\cC)$, we find a unitary operator
$$\Sigma : L^2(S) \ot_M X \recht X \ot_M L^2(S)$$
by composing
$$L^2(S) \ot_M X = \bigoplus_{\al \in \Irr(\cC)} (H_\al \ot_M X) \ot \overline{\cH_\al} \;\;\overset{\oplus \sigma_\al}{\longrightarrow}\;\;
\bigoplus_{\al \in \Irr(\cC)} (X \ot_M H_\al) \ot \overline{\cH_\al} = X \ot_M L^2(S) \; .$$
Define $\cH = L^2(S) \ot_M X$ and note that $\cH$ is a left Hilbert $S$-module. Defining
\begin{equation}\label{eq.right-module}
\xi \cdot x = \Sigma^* ( (\Sigma \xi) \cdot x) \;\; ,
\end{equation}
we also have a right Hilbert $S$-module structure on $\cH$. In \cite{NY15a}, it is proved that these left and right actions commute and that $\cH$ is a generalized SE-correspondence. Moreover, it is proved in \cite{NY15a} that all generalized SE-correspondences arise canonically in this way from a unitary half braiding on an ind-object.

In combination with Corollary \ref{cor.corr-vs-tube-category}, there is thus also a natural bijection between nondegenerate Hilbert space representations of the tube $*$-algebra $\cA$ and unitary half braidings on $\indC$. Both the tube $*$-algebra $\cA$ and the notion of a unitary half braiding are defined without referring to the realization of $\cC$ as a category of finite index $M$-bimodules. It is therefore not surprising that we can as follows construct this bijection in an abstract context of rigid C$^*$-tensor categories.

\begin{proposition} \label{prop.braiding-vs-rep}
Let $\cC$ be a rigid C$^*$-tensor category and denote by $\cA$ the associated tube $*$-algebra. The following defines a natural bijection between unitary half braidings on ind-objects for $\cC$ and nondegenerate right Hilbert $\cA$-modules $\cK$.
\begin{itemize}
\item Given a unitary half braiding $\sigma$ on $X \in \indC$, define the Hilbert spaces $\cK_i = (X,i)$, $i \in \Irr(\cC)$ and define $\cK$ as the orthogonal direct sum of all $\cK_i$, $i \in \Irr(\cC)$. The formula
    \begin{equation}\label{eq.from-braiding-to-rep}
    \xi \cdot V = (\Tr_\al \ot \id)(\sigma_\al^* (\xi \ot 1) V) \quad\text{for all}\;\; \xi \in \cK_i , V \in (i\al,\al j), i,j,\al \in \Irr(\cC)
    \end{equation}
    turns $\cK$ into a nondegenerate right Hilbert $\cA$-module satisfying $\cK_i = \cK \cdot p_i$.
\item Given a nondegenerate right Hilbert $\cA$-module $\cK$, write $\cK_i = \cK \cdot p_i$ for all $i \in \Irr(\cC)$ and define the ind-object $X \in \indC$ such that $(X,i) = \cK_i$ for all $i \in \Irr(\cC)$. There is a unique unitary half braiding $\sigma$ on $X$ satisfying
    \begin{equation}\label{eq.from-rep-to-braiding}
    \Tr_{\al j}((1 \ot \eta^*) \sigma_\al^* (\xi \ot 1) V) = \langle \xi \cdot V , \eta \rangle
    \end{equation}
    for all $i,j,\al \in \Irr(\cC)$, $\xi \in \cK_i$, $\eta \in \cK_j$, $S \in (i \al , \al j)$.
\end{itemize}
\end{proposition}
\begin{proof}
Let $\sigma$ be a unitary half braiding on $X \in \indC$. Define the pre-Hilbert space $\cK^0$ as the algebraic direct sum of the Hilbert spaces $\cK_i$, $i \in \Irr(\cC)$. Consider the bilinear map $\cK^0 \times \cA \recht \cK^0$ given by \eqref{eq.from-braiding-to-rep}. The multiplicativity of $\sigma$, i.e.\ $\sigma_{\al \be} = (\sigma_\al \ot 1)(1 \ot \sigma_\be)$, implies that $(\xi \cdot V) \cdot W = \xi \cdot (V \cdot W)$ for all $\xi \in \cK^0$, $V,W \in \cA$.

Since $\sigma_{\albar \al} = (\sigma_{\albar} \ot 1)(1 \ot \sigma_\al)$, we get $1 \ot t_\al = (\sigma_{\albar} \ot 1)(1 \ot \sigma_\al)(t_\al \ot 1)$ and thus,
$$\sigma_\al = (s_\al^* \ot 1 \ot 1)(1 \ot \sigma_{\albar}^* \ot 1)(1 \ot 1 \ot t_\al) \; .$$
It follows that $\langle \xi \cdot V ,\eta \rangle = \langle \xi, \eta \cdot V^\# \rangle$ for all $\xi,\eta \in \cK$, $V \in \cA$.

By Lemma \ref{lem.sum}, this $*$-anti-representation of $\cA$ on $\cK^0$ is necessarily by bounded operators. So, we can pass to the completion and have found the nondegenerate right Hilbert $\cA$-module $\cK$.

Conversely, assume that we are given a nondegenerate right Hilbert $\cA$-module $\cK$. Define the ind-object $X \in \indC$ such that $(X,i) = \cK_i$ for all $i \in \Irr(\cC)$. Define $X_i$ as the sub-object of $X$ given as the direct sum of all sub-objects equivalent with $i$. For all fixed $i,j,\al \in \Irr(\cC)$ and every fixed $V \in (i \al,\al j)$, we have that $(\xi,\eta) \mapsto \langle \xi \cdot V, \eta\rangle$ is a bounded sesquilinear form on $\cK_i \times \cK_j$. So we have uniquely defined bounded morphisms $\sigma_{\al,ij} \in (X_i \al, \al X_j)$ satisfying
$$\Tr_{\al j}((1 \ot \eta^*) \sigma_{\al,ij}^* (\xi \ot 1) S) = \langle \xi \cdot S , \eta \rangle$$
for all $\xi \in \cK_i$, $\eta \in \cK_j$, $S \in (i \al , \al j)$.

For fixed $\al,j \in \Irr(\cC)$, there are only finitely many $i \in \Irr(\cC)$ for which $(i\al,\al j) \neq \{0\}$. So, for fixed $\al, j \in \Irr(\cC)$, there are only finitely many $i \in \Irr(\cC)$ for which $\sigma_{\al,ij} \neq 0$. Similarly, for fixed $\al,i \in \Irr(\cC)$, there are only finitely many $j \in \Irr(\cC)$ for which $\sigma_{\al,ij} \neq 0$. Define $\sigma_\al = (\sigma_{\al,ij})_{ij}$ as an infinite matrix indexed by $\Irr(\cC)$. We uniquely define $\sigma_\al$ for arbitrary objects $\al \in \cC$ such that naturality holds. By the finiteness properties, all these infinite matrices can be multiplied.

The multiplicativity of the right $\cA$-action on $\cK$ translates to $\sigma_{\al \be} = (\sigma_\al \ot 1)(1 \ot \sigma_\be)$. We then also get that $1 \ot t_\al = (\sigma_{\albar} \ot 1)(1 \ot \sigma_\al)(t_\al \ot 1)$ and thus,
$$1 = \sigma_{\albar} \; (1 \ot 1 \ot s_\al^*) (1 \ot \sigma_\al \ot 1) (t_\al \ot 1 \ot 1) \; .$$
The property that $\langle \xi \cdot V,\eta \rangle = \langle \xi, \eta \cdot V^\# \rangle$ translates to
$$\sigma_{\albar}^* = (1 \ot 1 \ot s_\al^*) (1 \ot \sigma_\al \ot 1) (t_\al \ot 1 \ot 1)$$
and we find that $\sigma_{\albar} \sigma_{\albar}^* = 1$.

From the formula $\sigma_{\albar \al} = (\sigma_{\albar} \ot 1)(1 \ot \sigma_\al)$, we also get that $t_\al^* \ot 1 = (1 \ot t_\al^*)(\sigma_{\albar} \ot 1)(1 \ot \sigma_\al)$ and thus,
$$1 = (1 \ot 1 \ot t_\al^*)(1 \ot \sigma_{\albar} \ot 1)(s_\al \ot 1 \ot 1) \; \sigma_\al \; ,$$
meaning that $1 = \sigma_\al^* \sigma_\al$. Altogether, it follows that for every $\al \in \cC$, the infinite matrix $\sigma_\al$ actually defines a unitary morphism $\sigma_\al \in (X \al,\al X)$. So we have found the required unitary half braiding $\sigma$ on $X$.
\end{proof}

\subsection{The tube algebra and the affine category of a planar algebra}\label{sec.affine-cat}

Jones introduced the \emph{affine category} associated to a subfactor inclusion (this notion is related to his annular category \cite{jones:annular}).  Let us briefly recall its construction.
Let $P=(P^{\pm}_k)$ be a planar algebra which can be viewed as the quotient of a universal planar algebra \cite{jones:planar} by a set of relations $R$.  Given a tangle $T$ in the universal planar algebra, one can separate its strings into three groups and draw it on the sphere with two disks labeled ``left'' and ``right'' removed (in the drawing the sphere is identified with the plane to which we add a point at infinity):

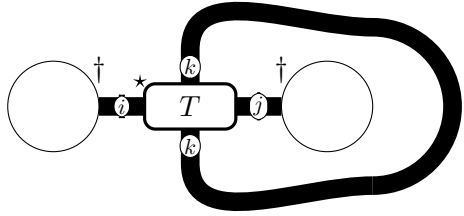
\begin{equation*}
 \begin{tikzpicture}[scale=.60]
 \draw[Box](0,0) rectangle (2,1);
  \draw[verythickline](0,.5) --node[rcount,scale=.75]{$i$}++ (-1,0);
  \draw[verythickline](2,.5) --node[rcount,scale=.75]{$j$}++ (1,0);
  \draw (3,.5) arc(0:360:-1); \node at (3.0,1.3){$\dagger$};
   \draw (-1,.5) arc(0:360:1);  \node at (-1.0,1.3){$\dagger$};
  \draw[verythickline](1,1)  --node[rcount,scale=.75]{$k$}++ (0,.75) to[out=90,in=180] ++(4,0.5) arc(90:-90:1.75);
   \draw[verythickline](1,0)  --node[rcount,scale=.75]{$k$}++ (0,-.75) to[out=-90,in=180] ++(4,-0.5);
 \node at (1,.5) {$T$}; \node[marked] at (-.1,1.05) {};
 \end{tikzpicture}\end{equation*}

Here thick lines stand for the indicated number of parallel strings.
The symbols $\dagger$ mark a preferred interval on each of the two disks, corresponding to region bounding both the preferred portion of the leftmost disk and $\star$ as well as the region where the topmost string of $T$ connects to the rightmost disk.  Such drawings make sense both for shaded planar algebras (the kind coming from subfactor theory) as well as the unshaded planar algebras.  We will mainly concentrate on the shaded case in this paper, although it is worth pointing out that our constructions work unaltered in the unshaded case as well.  In the shaded case, the additional data on the picture is the shading (not shown) so that each string lies at the boundary of a shaded and an unshaded region.  The shading of the picture is completely determined once we specify the shading of one of the regions (e.g., the region marked by $\star$).  In this case the shading of the region containing the left-most $\dagger$ is the same as the shading of the region containing $\star$, while the shading of the rightmost region containing $\dagger$ is either the same or opposite, depending on whether $k$ is even or not.   Alternatively, we can fix the shading of each the two regions containing the symbols $\dagger$ (note that this also fixes the parity of $k$).

Because the drawing is on the sphere, we can equally well draw it as
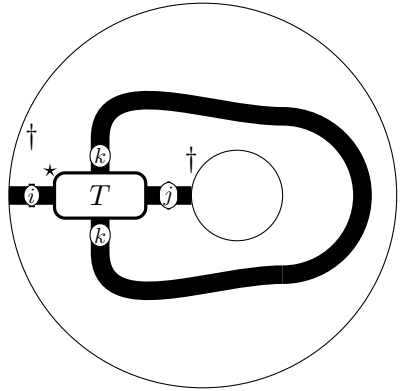
\begin{equation}\label{eq:tangle}
 \begin{array}{c}
 \begin{tikzpicture}[scale=.60]
 \draw[Box](0,0) rectangle (2,1);
  \draw[verythickline](0,.5) --node[rcount,scale=.75]{$i$}++ (-1,0);
  \draw[verythickline](2,.5) --node[rcount,scale=.75]{$j$}++ (1,0);
  \draw (3,.5) arc(0:360:-1);  \node at (3.0,1.3){$\dagger$};
   \draw (-1,.5) arc(0:360:-4.25); \node at (-0.5,1.8){$\dagger$};
  \draw[verythickline](1,1)  --node[rcount,scale=.75]{$k$}++ (0,.75) to[out=90,in=180] ++(4,0.5) arc(90:-90:1.75);
   \draw[verythickline](1,0)  --node[rcount,scale=.75]{$k$}++ (0,-.75) to[out=-90,in=180] ++(4,-0.5);
 \node at (1,.5) {$T$}; \node[marked] at (-.1,1.05) {};
\end{tikzpicture}\end{array}\end{equation} which is more customary (in the latter picture the inner disk is often called the ``input disk'' and the outer disk, the ``output disk'').

One considers the linear span of such diagrams (taken up to isotopy that fixes the boundaries of the annulus) and then takes a quotient by an appropriate subspace which ensures that any relation $R$ still holds when drawn in any open simply connected region inside the annulus.  The resulting quotient is denoted by $\mathscr{A}(P)$ and is called the affine category (or affine algebroid) associated to $P$.  We will sometimes write $\mathscr{A}$ when $P$ is understood.

Note that $\mathscr{A}(P)$ is bi-graded by the numbers of strings going to the left and right disks
as well as the choices of shading of the two regions marked by $\dagger$.

This linear space has a natural  multiplication $x\cdot y$ given by drawing the tangles for $x$ and $y$ as in \eqref{eq:tangle} and then gluing $y$ into the input disk of $x$
in a way that matches the regions marked by $\dagger$ (the multiplication is defined to be zero unless the the output disk of $y$ has the same number of string boundary points as the input disk of $x$ and compatible shading).  There is also an involution $\#$ given by an orientation-reversing diffeomorphism of the sphere that switches the two removed disks.

By \cite[Proposition 5.6]{Cu12} and \cite[Proposition 3.5]{GJ15}, the algebra $\mathscr{A}(P)$ is naturally Morita equivalent to the tube algebra $\mathcal{A}$ of Section \ref{sec.tube-tensor}.

For future purposes we point out that the algebra $\mathscr{A}$ has a natural subalgebra $\mathscr{B}$ consisting of sums of elements of the form

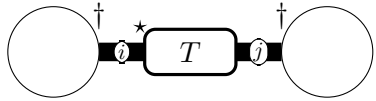
\begin{equation*}
 \begin{tikzpicture}[scale=.60]
 \draw[Box](0,0) rectangle (2,1);
  \draw[verythickline](0,.5) --node[rcount,scale=.75]{$i$}++ (-1,0);
  \draw[verythickline](2,.5) --node[rcount,scale=.75]{$j$}++ (1,0);
  \draw (3,.5) arc(0:360:-1); \node at (3.0,1.3){$\dagger$};
   \draw (-1,.5) arc(0:360:1);  \node at (-1.0,1.3){$\dagger$};
 \node at (1,.5) {$T$}; \node[marked] at (-.1,1.05) {};
 \end{tikzpicture}\end{equation*}

i.e., ones that have no strings looping around the right disk.
This subalgebra is also bi-graded according to the shading of input and output disks and the number of input/output strings.

Also for future purposes we would like to present a graphical picture for the tensor product $\mathscr{A}\otimes_\mathscr{B} \mathscr{A}$ as well as the higher tensor powers $\mathscr{A}\otimes_{\mathscr{B}} \otimes \cdots \otimes_{\mathscr{B}} \mathscr{A}$.  To do so, we consider the space $X_k$ which is the two-sphere $S^2$ with $k$ points $r_1,\ldots,r_k$ as well as two disks removed; these disks are labeled ``left'' and ``right''.  We then consider once again a planar algebra $P$ as a quotient of the universal planar algebra by a set of relations $R$.  This time, we consider the space $\mathscr{A}_k$ given by the linear span of isotopy classes of elements of the universal planar algebra drawn on $X_k$ modulo the linear span of relations from $R$ which are taken to hold true in any open simply connected region in $X_k$.  In the example below,  $T\in P_{\frac{1}{2}(i+j+2r+2s)}$ gives rise to an element in $\mathscr{A}_2$:

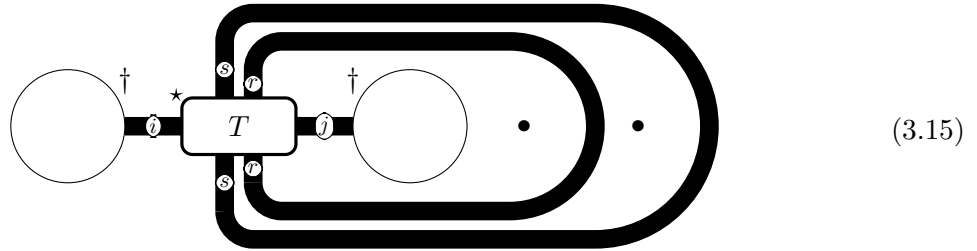
\begin{equation}\label{eq:tangleAlmostStandard}
\begin{array}{c}
 \begin{tikzpicture}[scale=.75]
 \draw[Box](0,0) rectangle (2,1);
  \draw[verythickline](0,.5) --node[rcount,scale=.75]{$i$}++ (-1,0);
  \draw[verythickline](2,.5) --node[rcount,scale=.75]{$j$}++ (1,0);
  \draw (3,.5) arc(0:360:-1);\node at (3.0,1.3){$\dagger$};
   \draw (-1,.5) arc(0:360:1);\node at (-1.0,1.3){$\dagger$};
 \node at (1,.5) {$T$}; \node[marked] at (-.1,1.05) {};
 \node at (6,.5){$\bullet$};
  \node at (8,.5){$\bullet$};
  \draw[verythickline](1.25,1) --node[rcount,scale=.75]{$r$}++ (0,0.5) arc(180:90:0.5) -- ++(4,0) arc(90:-90:1.5) -- ++(-4,0) arc(-90:-180:0.5);
  \draw[verythickline](1.25,0) --node[rcount,scale=.75]{$r$}++ (0,-0.5);
  \draw[verythickline](.75,1) --node[rcount,scale=.75]{$s$}++ (0,1) arc(180:90:0.5) -- ++(6,0) arc(90:-90:2) -- ++(-6,0) arc(-90:-180:0.5);
  \draw[verythickline](.75,0) --node[rcount,scale=.75]{$s$}++ (0,-1);
 \end{tikzpicture}\end{array}\end{equation}

We endow the space $\mathscr{A}_k$ with an $\mathscr{A}$-bimodule structure as follows.  The left multiplication action is given by gluing an element of $x\in \mathscr{A}$ drawn as in \eqref{eq:tangle} into the left disk of an element $\xi\in \mathscr{A}_k$ (with $x\xi=0$ if the number of string boundary points on the outer disk of $x$ is different from the number of boundary points on the left disk of $A$).  The right action (of the opposite algebra $\mathscr{A}^{op}$) is given by gluing an element of $\mathscr{A}$ into the right disk, with a similar requirement of equality of numbers of boundary points.

Note that by isotoping strings as illustrated below
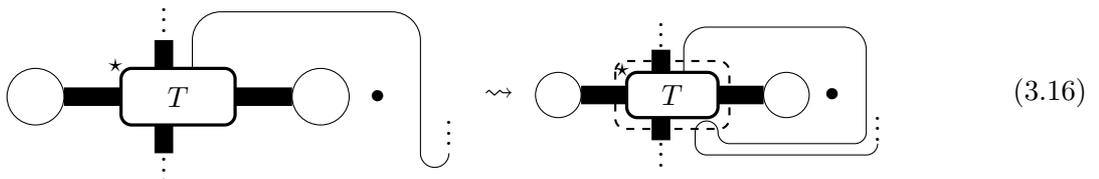
\begin{equation}\label{eq:addcup}
\begin{array}{c}
 \begin{tikzpicture}[scale=.75]
 \draw[Box](0,0) rectangle (2,1);
  \draw[verythickline](0,.5) --++ (-1,0);
  \draw[verythickline](2,.5) --++ (1,0);
  \draw (3,.5) arc(0:360:-0.5);
   \draw (-1,.5) arc(0:360:0.5);
 \node at (1,.5) {$T$}; \node[marked] at (-.1,1.05) {};
 \node at (4.5,.5){$\bullet$};
 \draw[verythickline](0.75,1) --++(0,0.5) ++(0,0.5) node{$\vdots$};
  \draw[verythickline](0.75,0) --++(0,-0.5) ++(0,-.150) node{$\vdots$};
  \draw(1.25,1) --++ (0,0.5) arc(180:90:0.5) -- ++(3,0) arc(90:0:.5) --++(0,-2) arc(-180:0:0.25) ++(0,0.5) node{$\vdots$};
 \end{tikzpicture}
 \end{array} \rightsquigarrow \begin{array}{c}
 \begin{tikzpicture}[scale=.60]
 \draw[Box](0,0) rectangle (2,1);
  \draw[verythickline](0,.5) --++ (-1,0);
  \draw[verythickline](2,.5) --++ (1,0);
  \draw (3,.5) arc(0:360:-0.5);
   \draw (-1,.5) arc(0:360:0.5);
 \node at (1,.5) {$T$}; \node[marked] at (-.1,1.05) {};
 \node at (4.5,.5){$\bullet$};
 \draw[verythickline](0.75,1) --++(0,0.5) ++(0,0.5) node{$\vdots$};
  \draw[verythickline](0.75,0) --++(0,-0.5) ++(0,-.150) node{$\vdots$};
  \draw(1.25,1) --++ (0,0.5) arc(180:90:0.5) -- ++(3,0) arc(90:0:.5) --++(0,-1.825) arc (0:-90:0.25) --++(-2.75,0) arc(-90:-180:0.25) arc(0:180:0.25) --++(0,-0.25) arc(-180:-90:0.25)--++(3.5,0)  arc(-90:0:0.25) ++(0,0.5) node{$\vdots$};
  \draw[thick,rounded corners,dashed] (-0.25,1.25) rectangle (2.25,-0.25);
 \end{tikzpicture}
 \end{array}\end{equation}
and viewing the inside of the dashed region as another planar algebra element, $T'$, obtained by adding $\cap$ to the bottom of $T$, one can always draw an element of $\mathscr{A}_k$ in the form of \eqref{eq:tangleAlmostStandard}.

  Alternatively, using the fact that the picture is drawn on the sphere, any element can be viewed as linear combination of elements of the form
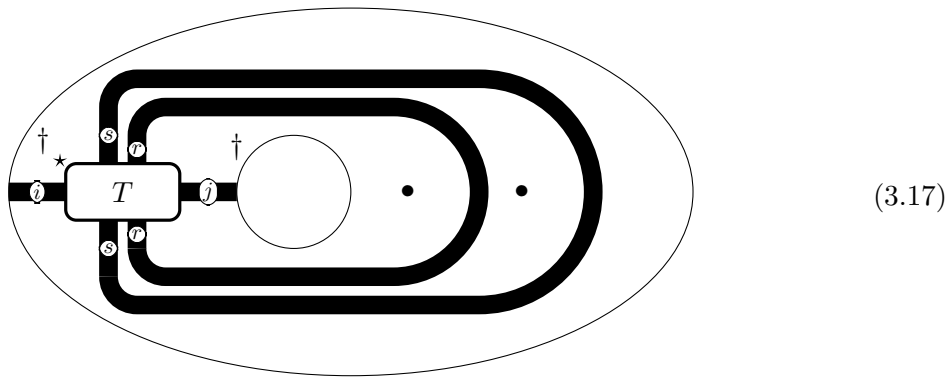
\begin{equation}\label{eq:tangleStandard}
\begin{array}{c}
 \begin{tikzpicture}[scale=.75]
 \draw[Box](0,0) rectangle (2,1);
  \draw[verythickline](0,.5) --node[rcount,scale=.75]{$i$}++ (-1,0);
  \draw[verythickline](2,.5) --node[rcount,scale=.75]{$j$}++ (1,0);
  \draw (3,.5) arc(0:360:-1); \node at (3.0,1.3){$\dagger$};
   \draw (-1,.5) arc(0:360:-6 and -3.25); \node at (-0.4,1.4){$\dagger$};
 \node at (1,.5) {$T$}; \node[marked] at (-.1,1.05) {};
 \node at (6,.5){$\bullet$};
  \node at (8,.5){$\bullet$};
  \draw[verythickline](1.25,1) --node[rcount,scale=.75]{$r$}++ (0,0.5) arc(180:90:0.5) -- ++(4,0) arc(90:-90:1.5) -- ++(-4,0) arc(-90:-180:0.5);
  \draw[verythickline](1.25,0) --node[rcount,scale=.75]{$r$}++ (0,-0.5);
  \draw[verythickline](.75,1) --node[rcount,scale=.75]{$s$}++ (0,1) arc(180:90:0.5) -- ++(6,0) arc(90:-90:2) -- ++(-6,0) arc(-90:-180:0.5);
  \draw[verythickline](.75,0) --node[rcount,scale=.75]{$s$}++ (0,-1);
 \end{tikzpicture}  \end{array}\end{equation}

\begin{lemma}\label{lem.tensor-product-planar}
$\mathscr{A}_k$ is isomorphic to the $k$-fold tensor product  $\mathscr{A}\otimes_{\mathscr{B}} \cdots \otimes_{\mathscr{B}} \mathscr{A}$.
\end{lemma}

\begin{proof}
The proof is by induction on $k$; the case $k=1$ is clear.  Assuming the isomorphism to hold for $k$, we note that there is map from  $\mathscr{A}\otimes_{\mathscr{B}} \mathscr{A}_k$ to $\mathscr{A}_{k+1}$ given by:

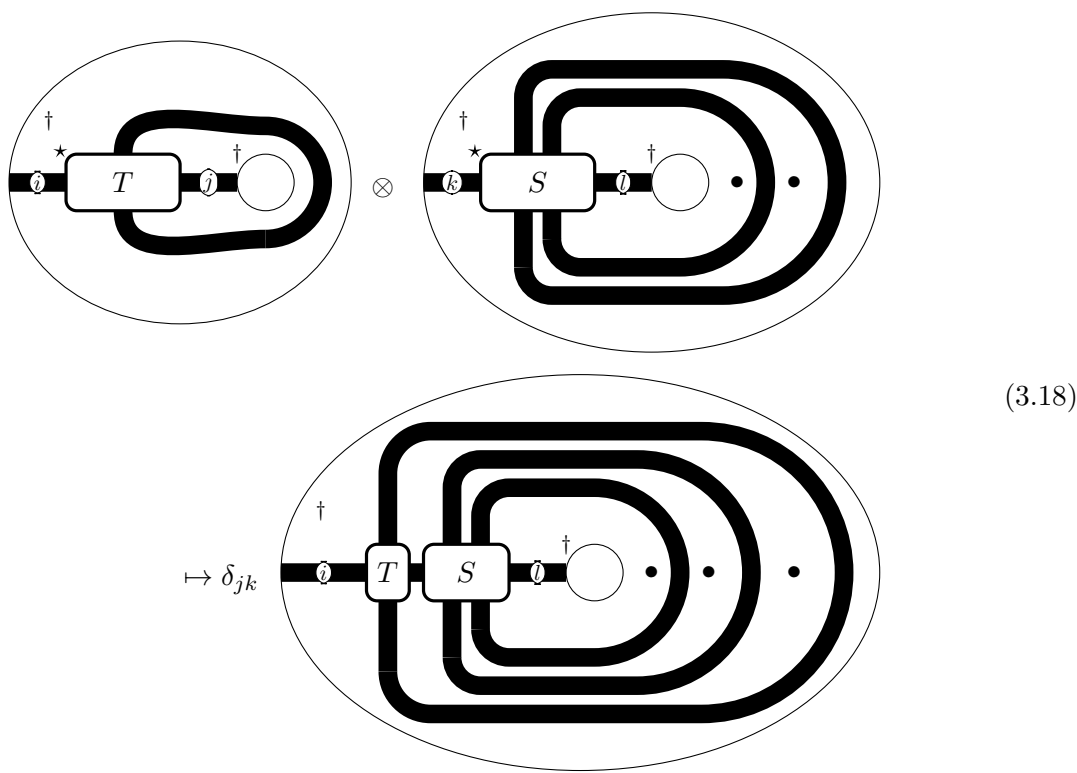
\begin{equation}\label{eq:tangleMap}
\begin{split}
 \begin{array}{c}
 \begin{tikzpicture}[scale=.75]
 \draw[Box](0,0) rectangle (2,1);
  \draw[verythickline](0,.5) --node[rcount,scale=.75]{$i$}++ (-1,0);
  \draw[verythickline](2,.5) --node[rcount,scale=.75]{$j$}++ (1,0);
  \draw (3,.5) arc(0:360:-0.5); \node at (3.0,1){\scriptsize $\dagger$};  \node at (-0.3,1.6){\scriptsize $\dagger$};
   \draw (-1,.5) arc(0:360:-3.0 and -2.5);
  \draw[verythickline](1,1)  --++ (0,0) to[out=90,in=180] ++(2.5,0.5) arc(90:-90:1);
   \draw[verythickline](1,0)  --++ (0,0) to[out=-90,in=180] ++(2.5,-0.5);
 \node at (1,.5) {$T$}; \node[marked] at (-.1,1.05) {};
 \end{tikzpicture}\end{array}
\otimes
\begin{array}{c}
 \begin{tikzpicture}[scale=.75]
 \draw[Box](0,0) rectangle (2,1);
  \draw[verythickline](0,.5) --node[rcount,scale=.75]{$k$}++ (-1,0);
  \draw[verythickline](2,.5) --node[rcount,scale=.75]{$l$}++ (1,0);
  \draw (3,.5) arc(0:360:-0.5);
   \draw (-1,.5) arc(0:360:-4 and -3);  \node at (3.0,1){\scriptsize $\dagger$};  \node at (-0.3,1.6){\scriptsize $\dagger$};
 \node at (1,.5) {$S$}; \node[marked] at (-.1,1.05) {};
 \node at (4.5,.5){$\bullet$};
  \node at (5.5,.5){$\bullet$};
  \draw[verythickline](1.25,1) --++ (0,0.5) arc(180:90:0.5) -- ++(1.75,0) arc(90:-90:1.5) -- ++(-1.75,0) arc(-90:-180:0.5);
  \draw[verythickline](1.25,0) --++ (0,-0.5);
  \draw[verythickline](.75,1) --++ (0,1) arc(180:90:0.5) -- ++(3,0) arc(90:-90:2) -- ++(-3,0) arc(-90:-180:0.5);
  \draw[verythickline](.75,0) --++ (0,-1);
 \end{tikzpicture}  \end{array} \\
 \mapsto \delta_{jk}
\begin{array}{c}
 \begin{tikzpicture}[scale=.75]
 \draw[Box](-.5,0) rectangle (0.25,1);
 \draw[Box](0.5,0) rectangle (2,1);
  \draw[verythickline](-0.5,.5) --node[rcount,scale=.75]{$i$}++ (-1.5,0);
  \draw[verythickline](2,.5) --node[rcount,scale=.75]{$l$}++ (1,0);
  \draw (3,.5) arc(0:360:-0.5);
   \draw (-2,.5) arc(0:360:-5.25 and -3.5);  \node at (3.0,1){\scriptsize $\dagger$};  \node at (-1.3,1.6){\scriptsize $\dagger$};
 \node at (-0.125,.5) {$T$};
 \node at (1.25,.5) {$S$};
 \node at (4.5,.5){$\bullet$};
  \node at (5.5,.5){$\bullet$};
    \node at (7,.5){$\bullet$};
  \draw[verythickline](1.5,1) --++ (0,0.5) arc(180:90:0.5) -- ++(1.5,0) arc(90:-90:1.5) -- ++(-1.5,0) arc(-90:-180:0.5);
  \draw[verythickline](1.5,0) --++ (0,-0.5);
  \draw[verythickline](1,1) --++ (0,1) arc(180:90:0.5) -- ++(2.75,0) arc(90:-90:2) -- ++(-2.75,0) arc(-90:-180:0.5);
    \draw[verythickline](1,0) --++ (0,-1.0);
   \draw[verythickline](-0.125,1) --++ (0,1.25) arc(180:90:0.75) -- ++(4.75,0) arc(90:-90:2.5) -- ++(-4.75,0) arc(-90:-180:0.75);
  \draw[verythickline](-0.125,0) --++ (0,-1.25);
  \draw[verythickline](0.25,0.5) --++ (0.25,0);
 \end{tikzpicture}  \end{array}\end{split}\end{equation}

It is clear that this map is surjective, since every element of the form \eqref{eq:tangleStandard} can be clearly obtained in its image.

We claim that the map is injective.  To see this, consider the construction of the tangle \eqref{eq:tangleStandard}:

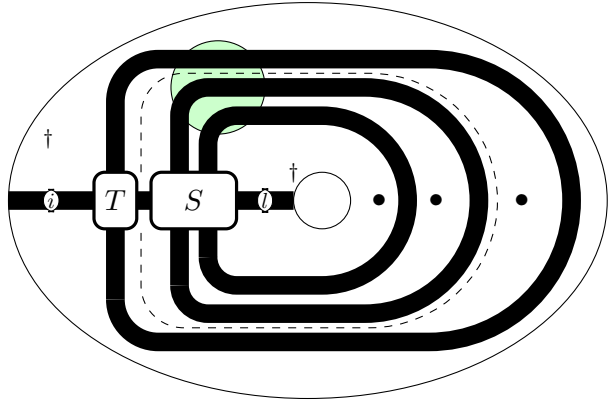
$$
\begin{array}{c}
 \begin{tikzpicture}[scale=.75]
    \filldraw[fill=green!20!white] (2.5,2.5) arc(0:360:0.825);
 \draw[Box](-.5,0) rectangle (0.25,1);
 \draw[Box](0.5,0) rectangle (2,1);
  \draw[verythickline](-0.5,.5) --node[rcount,scale=.75]{$i$}++ (-1.5,0);
  \draw[verythickline](2,.5) --node[rcount,scale=.75]{$l$}++ (1,0);
  \draw (3,.5) arc(0:360:-0.5);
   \draw (-2,.5) arc(0:360:-5.25 and -3.5);\node at (3.0,1){\scriptsize $\dagger$};  \node at (-1.3,1.6){\scriptsize $\dagger$};
 \node at (-0.125,.5) {$T$};
 \node at (1.25,.5) {$S$};
 \node at (4.5,.5){$\bullet$};
  \node at (5.5,.5){$\bullet$};
    \node at (7,.5){$\bullet$};
  \draw[verythickline](1.5,1) --++ (0,0.5) arc(180:90:0.5) -- ++(1.5,0) arc(90:-90:1.5) -- ++(-1.5,0) arc(-90:-180:0.5);
  \draw[verythickline](1.5,0) --++ (0,-0.5);
  \draw[verythickline](1,1) --++ (0,1) arc(180:90:0.5) -- ++(2.75,0) arc(90:-90:2) -- ++(-2.75,0) arc(-90:-180:0.5);
    \draw[verythickline](1,0) --++ (0,-1.0);
   \draw[verythickline](-0.125,1) --++ (0,1.25) arc(180:90:0.75) -- ++(4.75,0) arc(90:-90:2.5) -- ++(-4.75,0) arc(-90:-180:0.75);
  \draw[verythickline](-0.125,0) --++ (0,-1.25);
  \draw[verythickline](0.25,0.5) --++ (0.25,0);
  \draw[dashed](0.325,1) --++ (0,1.0) arc(180:90:0.75) -- ++(3.25,0) arc(90:-90:2.25) -- ++(-3.25,0) arc(-90:-180:0.75) --++(0,1.0);
 \end{tikzpicture}  \end{array}
 $$
and let us analyze the effect of removing the dashed line in the figure above. The removal of this line permits us to apply any isotopy or a relation from $R$ in some open simply connected region that intersects the dashed line (drawn in green on the picture).  We may now move this region along the dashed circle until it is in between $T$ and $S$, deforming the strings going in and out of the green region in the process (in a way similar to \eqref{eq:addcup}).  But this means that by possibly modifying $T$ and $S$ we may assume that the relation precisely amounts to identifying $T a\otimes S$ with $T\otimes aS$ for $a\in \mathscr{B}$.
\end{proof}

\section[Cohomology of quasi-regular inclusions of von Neumann algebras]{Cohomology of quasi-regular inclusions of von Neumann \linebreak algebras}

Throughout this section, we fix a tracial von Neumann algebra $(S,\tau)$ with von Neumann subalgebra $T \subset S$.

\begin{definition}\label{def.cohom-quasireg}
Whenever $\cH$ is a Hilbert $S$-bimodule and $T \subset \cS \subset \QN_S(T)$ is an intermediate $*$-algebra, we define the cohomology $H^n(T \subset \cS,\cH)$ as the $n$-th cohomology of the complex
$$C^0 \overset{\delta}{\longrightarrow} C^1 \overset{\delta}{\longrightarrow} \cdots$$
where $C^0 = \cH^c_T =$ the space of $T$-central vectors in the rank completion $\cH^c$ of $\cH$ as a $\cZ(T)$-bimodule,
\begin{align*}
& C^n = \text{the space of $T$-bimodular maps from}\; \underbrace{\cS \ot_T \cdots \ot_T \cS}_{n \;\text{factors}} \;\text{to}\; \cH^c \quad\text{and}\\
& \delta : C^n \recht C^{n+1} : \delta = \sum_{i=0}^{n+1} (-1)^i \delta_i \quad\text{is given by}\\
& (\delta_0 c)(x_0 \ot \cdots \ot x_n) = x_0 \cdot c(x_1 \ot \cdots \ot x_n) \; ,\\
& (\delta_i c)(x_0 \ot \cdots \ot x_n) = c(x_0 \ot \cdots \ot x_{i-1} x_i \ot \cdots \ot x_n) \quad\text{for}\;\; i=1,\ldots,n \quad\text{and}\\
& (\delta_{n+1} c)(x_0 \ot \cdots \ot x_n) = c(x_0 \ot \cdots \ot x_{n-1}) \cdot x_n \; .
\end{align*}
\end{definition}

\begin{remark}\label{rem.cohom-quasireg}
\begin{enumlist}
\item Note that the maps $\delta_0$ and $\delta_{n+1}$ are well defined because by Lemma \ref{lem.rank-cont}, the rank completion $\cH^c$ is an $\cS$-bimodule. In the definition of $C^n$, we denote by $\ot_T$ the algebraic relative tensor product.

\item When $T$ is a factor, and in particular when $T' \cap S = \C 1$, the rank completion over $\cZ(T)$ in Definition \ref{def.cohom-quasireg} disappears. However, as we will see below, when $T \subset S$ is a Cartan subalgebra, it is crucial to take the rank completion in order to recover the usual cohomology theory of the underlying equivalence relation.

\item We denote $Z^n(T \subset \cS,\cH) = \Ker(\delta : C^n \recht C^{n+1})$ and $B^n(T \subset \cS,\cH) = \Im(\delta : C^{n-1} \recht C^n)$. Note that $Z^1(T \subset \cS,\cH)$ precisely is the space of $T$-bimodular \emph{derivations} from $\cS$ to $\cH^c$, i.e.\ the space of $T$-bimodular maps $c : \cS \recht \cH^c$ satisfying $c(xy) = x c(y) + c(x) y$ for all $x,y \in \cS$.

\item The cohomology $H^n(T \subset \cS,\cH)$ only `sees' the part of $\cH$ that, as a Hilbert $T$-bimodule, is a direct sum of Hilbert $T$-bimodules that appear as subbimodules of some tensor power $L^2(S) \ovt_T \cdots \ovt_T L^2(S)$. Indeed, replacing $\cH$ by this $T$-subbimodule, the cochain spaces $C^n$ do not change.
    
\item In the general context of an algebra $\cS$ with a subalgebra $T$ and a given $\cS$-bimodule, the complex in Definition \ref{def.cohom-quasireg} already appeared in \cite[Section 3]{H56}.
\end{enumlist}
\end{remark}

We define the $L^2$-cohomology of $T \subset \cS$ as the cohomology with values in the following ``universal'' coarse $S$-bimodule (relative to $T$)~:
\begin{equation}\label{eq.coarse}
\begin{split}
\cHreg & = (L^2(S) \ovt_T L^2(S)) \; \oplus \; (L^2(S) \ovt_T L^2(S) \ovt_T L^2(S)) \oplus \cdots \\
& = L^2(S) \ovt_T \cH \ovt_T L^2(S) \quad\text{with}\quad \cH = L^2(T) \oplus L^2(S) \oplus (L^2(S) \ovt_T L^2(S)) \oplus \cdots \; .
\end{split}
\end{equation}
At first sight, it may sound more natural to consider $L^2(S) \ovt_T L^2(S)$ as the coarse $S$-bimodule, but then we do not have a Fell absorption principle and, as seen in Lemma \ref{lem.reg-triv-A}, we miss part of the regular representation from the point of view of the tube algebra.

We define $\cM(T \subset S)$ as the von Neumann algebra $\End_{S-S}(\cHreg)$. We have the following natural normal semifinite faithful weight $\mu$ on $\cM(T \subset S)$~: whenever $\cH_1$ is a bifinite Hilbert $T$-bimodule and $W : \cH_1 \recht \cH$ is a $T$-bimodular isometry with $p = WW^*$, we define the $T$-bimodular isometry $V : \cH_1 \recht \cHreg$ given by $V(\xi) = 1 \ot \xi \ot 1$ and put
\begin{equation}\label{eq.weight-on-End}
\mu((1 \ot p \ot 1) x (1 \ot p \ot 1)) = \Tr_{\cH_1}(V^* x V) \quad\text{for all}\;\; x \in \cM(T \subset S) \; ,
\end{equation}
where $\Tr_{\cH_1}$ is the canonical trace on $\End_{T-T}(\cH_1)$ (see \eqref{eq.canonical-trace}).

Note that $\End_{S-S}(L^2(S) \ovt_T L^2(S))$ is a corner of $\cM(T \subset S)$ and that the restriction of $\mu$ to this corner is the vector state given by the vector $1 \ot 1$.

\begin{definition} \label{def.L2-cohom-quasireg}
Let $T \subset \cS \subset \QN_S(T)$ be an intermediate $*$-algebra. We define the \emph{$L^2$-cohomology} of $T \subset \cS$ as $H^n(T \subset \cS,\cHreg)$.

Note that $H^n(T \subset \cS,\cHreg)$ canonically is an $\cM(T \subset S)$-module. In the unimodular case, i.e.\ when $\mu$ is a trace on $\cM(T \subset S)$, we define
$$\bnl(T \subset \cS) := \dim_{\cM(T \subset S)} H^n(T \subset \cS,\cHreg) \; .$$
Here, we use L\"{u}ck's dimension function for arbitrary modules over a von Neumann algebra with a semifinite trace, see \cite[Section 6.1]{Lu02} and \cite[Section A.4]{KPV13}.
\end{definition}

We are mainly interested in the following two types of quasi-regular von Neumann algebra inclusions $T \subset S$~: Cartan subalgebras and quasi-regular irreducible subfactors. We prove below that for a Cartan subalgebra $A \subset M$ of a tracial von Neumann algebra and $\cS = \lspan \cN_M(A)$, the cohomology theory in Definition \ref{def.cohom-quasireg} amounts to the usual cohomology theory for the underlying equivalence relation $\cR$ (and, in particular, forgets to scalar $2$-cocycle on $\cR$ that is given by $A \subset M$). In that case, unimodularity is automatic, $\cM(A \subset M)$ is an infinite amplification of $L(\cR)$, and $\bns(A \subset \cS)$ equals the $n$-th $L^2$-Betti number of $\cR$ in the sense of Gaboriau, \cite{Ga01}.

When $T \subset S$ is an irreducible quasi-regular subfactor, we interpret the cohomology theory in Definition \ref{def.cohom-quasireg} as a natural Hochschild type cohomology for the associated tube $*$-algebra. We prove that the unimodularity assumption is equivalent to the requirement that every $T$-subbimodule of $L^2(S)$ has equal left and right dimension, i.e.\ that $T \subset S$ is unimodular in the sense of Definition \ref{def.inclusion-unimodular}.

\section{Cohomology of Cartan subalgebras}\label{sec.cohom-cartan}

Fix a tracial von Neumann algebra $(M,\tau)$ with separable predual and Cartan subalgebra $L^\infty(X) \subset M$. Denote by $\cR$ the associated countable probability measure preserving (pmp) equivalence relation on $(X,\mu)$.

By \cite{FM75}, we know that $M$ is canonically isomorphic with $L_\Omega(\cR)$, where $\Omega : \cR^{(2)} \recht \T$ is a scalar $2$-cocycle on $\cR$. Here, $\cR^{(2)} = \{(x,y,z) \in X \times X \times X \mid (x,y) \in \cR \;\text{and}\; (y,z) \in \cR \}$. We similarly define $\cR^{(n)}$, and by convention, $\cR^{(0)} = X$. Recall that all $\cR^{(n)}$ are equipped with a natural $\sigma$-finite measure given by integration over $(X,\mu)$ of the counting measure through the projection $\pi : \cR^{(n)} \recht X : (x_0,\ldots,x_n) \mapsto x_0$.

A \emph{unitary representation} of $\cR$ consists of a measurable field of Hilbert spaces $(\cK_x)_{x \in X}$ and a measurable family of unitary operators $U(x,y) : K_y \recht K_x$ for all $(x,y) \in \cR$ such that $U(x,y) U(y,z) = U(x,z)$ for a.e.\ $(x,y,z) \in \cR^{(2)}$. Put $A = L^\infty(X)$. The integration of the field $(\cK_x)_{x \in X}$ yields the Hilbert $A$-module $\cK$ given by the $L^2$-sections of the field. Denote by $[\cR]$ the full group of $\cR$. For every $\al \in [\cR]$, define the unitary operator $U(\al)$ on $\cK$ given by
$$\bigl( U(\al)\xi \bigr)(x) = U(x,\al^{-1}(x)) \, \xi(\al^{-1}(x)) \quad\text{for all}\;\; \xi \in \cK , x \in X \; .$$
Define the Hilbert space $\cH = L^2(M) \ovt_A \cK$. The formulae
\begin{equation}\label{eq.M-bimod}
x \cdot (\xi \ot \eta) \cdot a v = x \xi a v \ot U(\al_v)^* \eta \quad\text{for all}\;\; x \in M, \xi \in L^2(M), \eta \in \cK, a \in A, v \in \cN_M(A)
\end{equation}
turn $\cH$ into a Hilbert $M$-bimodule. Here, $\al_v$ is the element of $[\cR]$ defined by $v \in \cN_M(A)$.

Note that the vectors $1 \ot \xi \in \cH$ are all $A$-central. Therefore, $\cH$ is generated, as a Hilbert $M$-bimodule, by $A$-central vectors. Conversely, it is easy to check that every Hilbert $M$-bimodule generated by $A$-central vectors arises in this way.

Given a unitary representation $U$ of $\cR$ on $\cK$, the cohomology $H^n(\cR,\cK)$ is defined as follows. Define the field $(\cK^{(n)}_x)_{x \in \cR^{(n)}}$ given by $\cK^{(n)}_x = \cK_{\pi(x)}$. Denote by $C^n$ the set of measurable sections of this field, identifying two sections when they coincide a.e. Then, $H^n(\cR,\cK)$ is defined as the cohomology of the complex
$$C^0 \overset{\delta}{\longrightarrow} C^1 \overset{\delta}{\longrightarrow} \cdots$$
where
\begin{align*}
& \delta : C^n \recht C^{n+1} : \delta = \sum_{i=0}^{n+1} (-1)^i \delta_i \quad\text{and}\\
& (\delta_0 \xi)(x_0,\ldots,x_{n+1}) = U(x_0,x_1) \xi(x_1,\ldots,x_n) \quad , \\
& (\delta_i \xi)(x_0,\ldots,x_{n+1}) = \xi(x_0,\ldots,\widehat{x_i},\ldots,x_{n+1}) \quad\text{for}\;\; i=1,\ldots,n+1 \; .
\end{align*}

The \emph{regular representation} is given by $\cKreg_x = \ell^2(\orbit(x))$ with $U(x,y)$ being the identity identification between $\orbit(x)$ and $\orbit(y)$ when $(x,y) \in \cR$. The $L^2$-cohomology of $\cR$ is thus given by the cohomology of the complex
$$\Creg^0 \overset{\delta}{\longrightarrow} \Creg^1 \overset{\delta}{\longrightarrow} \cdots$$
where $\Creg^n$ consists of the measurable functions $\xi : \cR^{(n+1)} \recht \C$ with the property that for a.e.\ $y \in \cR^{(n)}$, we have
$$\sum_{x \in \orbit(\pi(y))} |\xi(x,y)|^2 < \infty$$
and
$$(\delta \xi)(x,y_0,\ldots,y_{n+1}) = \sum_{j=0}^{n+1} (-1)^j \xi(x,y_0,\ldots,\widehat{y_j},\ldots,y_{n+1}) \; .$$
We can then turn $\Creg^n$ into a left $L(\cR)$-module by defining
$$(a u_\vphi \cdot \xi)(x,y) = a(x) \, \xi(\vphi^{-1}(x),y) \quad\text{for all}\;\; a \in A, \vphi \in [\cR], \xi \in \Creg^n , (x,y) \in \cR^{(n+1)} \; .$$
To define $b \cdot \xi$ for an arbitrary element $b \in L(\cR)$, note that $L^2(\cR^{(n+1)})$ is a Hilbert $L(\cR)$-$L^\infty(\cR^{(n)})$-bimodule. We can view $\Creg^n$ as the rank completion of $L^2(\cR^{(n+1)})$ viewed as a right $L^\infty(\cR^{(n)})$-module. This rank completion canonically stays a left $L(\cR)$-module.

We can then define the $L^2$-Betti numbers of the equivalence relation $\cR$ as
\begin{equation}\label{eq.L2-Betti-cohom-equiv-rel}
\bnl(\cR) := \dim_{L(\cR)} H^n(\cR,\cKreg) \; .
\end{equation}
It follows from \cite[Proposition 3.1]{KPV13} that this definition coincides with Gaboriau's definition of $L^2$-Betti numbers, \cite{Ga01}.

\begin{proposition}
Let $(M,\tau)$ be a tracial von Neumann algebra with separable predual and Cartan subalgebra $A \subset M$. Denote by $\cR$ the associated countable pmp equivalence relation. Denote $\cS = \lspan \cN_M(A)$.

For every unitary representation of $\cR$ on $\cK$, consider the associated $M$-bimodule $\cH$ given by \eqref{eq.M-bimod}. We then have a natural isomorphism
$$H^n(A \subset \cS,\cH) \cong H^n(\cR,\cK) \; .$$
We also have a natural isomorphism between $\cM(A \subset M)$ and an infinite amplification of $L(\cR)$, as well as the equality
$$\bnl(A \subset \cS) = \bnl(\cR) \quad\text{for all}\;\; n \in \N \; .$$
\end{proposition}
\begin{proof}
Denote by $C^n$ the space of measurable sections of the field $(\cK^{(n)}_x)_{x \in \cR^{(n)}}$ as above. The map $\eta \mapsto 1 \ot \eta$ is a unitary operator between $\cK$ and the space of $A$-central vectors in $\cH$. This map extends to a linear bijection between $C^0$ and the space of $A$-central vectors in $\cH^c$. Next, for every $v \in \cN_M(A)$, the map $\theta_v : \eta \mapsto v \ot \eta$ is a unitary operator between $\cK$ and the space of vectors $\xi \in \cH$ satisfying $a \xi = \xi \al_v(a)$ for all $a \in A$, where $\al_v(a) = v a v^*$. Again, $\theta_v$ extends to a linear bijection, still denoted by $\theta_v$, between $C^0$ and the space of vectors $\xi \in \cH^c$ satisfying $a \xi = \xi \al_v(a)$ for all $a \in A$.

We can then define the linear bijection
$$C^n \recht \Mor_{A-A}(\cS \ot_A \cdots \ot_A \cS, \cH^c)$$
given by $\xi \mapsto c_\xi$ where, for all $v_1,\ldots,v_n \in \cN_M(A)$,
$$c_\xi(v_1 \ot \cdots \ot v_n) = \theta_{v_1\cdots v_n}\bigl(x \mapsto U(x,\al_{v_1\cdots v_n}(x)) \xi(\al_{v_1 \cdots v_n}(x),\cdots,\al_{v_n}(x),x)\bigr) \; .$$
These maps define an isomorphism between the bar complexes defining $H^n(\cR,\cK)$ and $H^n(A \subset \cS,\cH)$, so that these cohomology spaces are isomorphic.

Since $A \subset M$ is regular, the coarse $M$-bimodule $\cHreg$ defined in \eqref{eq.coarse} is isomorphic with an infinite amplification of $L^2(M) \ovt_A L^2(M)$. We have the following canonical isomorphism $\Psi : L(\cR) \recht \End_{M-M}(L^2(M) \ovt_A L^2(M))$. For every $\vphi$ in the full group of the equivalence relation $\cR$, we have a unitary element $u_\vphi \in L(\cR)$ and we can choose $v \in \cN_M(A)$ such that $\al_v = \vphi$. We define $(\Psi(u_\vphi))(x \ot y) = xv^* \ot vy$. Note that the definition of $\Psi(u_\vphi)$ is independent of the choice of $v$. We also define $(\Psi(a))(x \ot y) = xa \ot y = x \ot ay$ for all $a \in A$. Together, we have found a trace preserving isomorphism $\Psi : L(\cR) \recht \End_{M-M}(L^2(M) \ovt_A L^2(M))$. So, we also have a canonical trace preserving isomorphism between $\cM(A \subset M)$ and an infinite amplification of $L(\cR)$. In particular, we find that $\bns(A \subset \cS) = \bns(\cR)$ for all $n \geq 0$.
\end{proof}

\section{Homology of irreducible quasi-regular inclusions}

Throughout this section, we fix a II$_1$ factor $S$ with separable predual and an irreducible quasi-regular subfactor $T \subset S$. We put $\cS = \QN_S(T)$. For any $T$-bimodule $\cK$, we denote by $\cK_T$ the subspace of $T$-central vectors.

\begin{definition}\label{def.homology-quasiregular}
For any Hilbert $S$-bimodule $\cH$, we define $H_n(T \subset \cS,\cH)$ as the homology of the complex
$$\cdots \overset{\partial}{\recht} C_2 \overset{\partial}{\recht} C_1 \overset{\partial}{\recht} C_0$$
where $C_0 = \cH_T$, and
\begin{align*}
& C_n = \bigl(\cH \ot_T \underbrace{\cS \ot_T \cdots \ot_T \cS}_{\text{$n$ factors}}\bigr)_T \quad\text{with}\quad \partial : C_n \recht C_{n-1} : \partial = \sum_{i=0}^n (-1)^i \partial_i \quad\text{given by}\\
& \partial_0(\xi \ot x_1 \ot \cdots \ot x_n) = \xi \cdot x_1 \ot x_2 \ot \cdots \ot x_n \;\; ,\\
& \partial_i(\xi \ot x_1 \ot \cdots \ot x_n) = \xi \ot x_1 \ot \cdots x_i x_{i+1} \ot \cdots \ot x_n \quad\text{for}\;\; i = 1,\ldots,n-1 \;\; ,\\
& \partial_n(\xi \ot x_1 \ot \cdots \ot x_n) = p_T(x_n \cdot \xi \ot x_1 \ot \cdots \ot x_{n-1}) \; .
\end{align*}
Here $p_T$ denotes the orthogonal projection onto the subspace of $T$-central vectors. In Remark \ref{rem.homology}, we explain why $\partial$ is well defined and satisfies $\partial^2 = 0$.
\end{definition}

\begin{remark}\label{rem.homology}
\begin{enumerate}
\item The boundary maps $\partial_i$ are well defined for the following reasons. Write $\cS_n = \cS \ot_T \cdots \ot_T \cS$ for the $n$-fold algebraic relative tensor product. For $0 \leq i \leq n-1$, the maps $\partial_i$ are first defined as $T$-bimodular maps from $\cH \ot_T \cS_{n}$ to $\cH \ot_T \cS_{n-1}$ and then restricted to the $T$-central vectors. For $i = n$, it follows from Lemma \ref{lem.well-def-pT} below that for all $\xi \in \cH$ and all $x_1,\ldots,x_n \in \cS$, the element $p_T(x_n \cdot \xi \ot x_1 \ot \cdots \ot x_{n-1})$ belongs to $\cH \ot_T \cS_{n-1}$. Since $p_T(a \cdot \eta - \eta \cdot a) = 0$ for all $a \in T$ and all $\eta \in \cH \ot_T \cS_{n-1}$, we conclude that $\partial_n$ is a well defined map from $\cH \ot_T \cS_{n}$ to $(\cH \ot_T \cS_{n-1})_T$ and we restrict this map to $(\cH \ot_T \cS_{n})_T$.

\item There are two ways to see that $\partial^2 = 0$. First, it would be more natural to consider as chain spaces the \emph{cyclic relative tensor products} $(\cH \ot_T \cS_n)/T$, defined as the quotient of $\cH \ot_T \cS_n$ by the subspace generated by $a \cdot \xi \ot x - \xi \ot x \cdot a$, $a \in T$, $\xi \in \cH$, $x \in \cS_n$. Defined as such, the chain spaces are however too large. Defining $\cH^0 \subset \cH$ as the ``algebraic part'' of $\cH$ consisting of all $T$-bounded vectors $\xi$ for which the closed linear span of $T \xi T$ is a finite index $T$-bimodule, it follows from \cite{Hu98} (based on \cite{FH80,Po97a}) that the inclusion map
    $$(\cH \ot_T \cS_n)_T \recht (\cH^0 \ot_T \cS_n)/T$$
    from the space of $T$-central vectors to the cyclic relative tensor product is indeed an isomorphism. In this way, we find an isomorphism between the complex $(C_n)$ and the natural bar complex $(\cH^0 \ot_T \cS_n)/T$. In particular, $\partial^2 = 0$.

	Secondly, the formula
$$\langle \xi \ot x_1 \ot \cdots \ot x_n , c \rangle = \langle \xi, c(x_n^* \ot \cdots \ot x_1^*) \rangle$$
defines a nondegenerate sesquilinear pairing between $(\cH \ot_T \cS_n)_T$ and $\Mor_{T-T}(\cS_n,\cH)$. Under this duality pairing, the complex $(C_n)$ is dual to the complex $(C^n)$ in Definition \ref{def.cohom-quasireg}. In particular, $\partial^2 = 0$.

Also note that we find in this way a natural sesquilinear pairing between the homology $H_n(T \subset \cS,\cH)$ and the cohomology $H^n(T \subset \cS,\cH)$ defined in \ref{def.cohom-quasireg}. This pairing can be degenerate, see Propositions \ref{prop.0-homology-cohomology} and \ref{prop.characterization-amenable-inclusion}.

\item As in Remark \ref{rem.cohom-quasireg}, note that the homology $H_n(T \subset \cS,\cH)$ only sees the part of $\cH$ that, as a Hilbert $T$-bimodule, is a direct sum of Hilbert $T$-subbimodules of some tensor power $L^2(S) \ovt_T \cdots \ovt_T L^2(S)$.
\end{enumerate}
\end{remark}

\begin{lemma} \label{lem.well-def-pT}
The natural map $V : \cH \ot_T \cS \ot_T \cdots \ot_T \cS \recht \cH \ovt_T L^2(S) \ovt_T \cdots \ovt_T L^2(S)$ is injective and the orthogonal projection onto the $T$-central vectors leaves the image of $V$ globally invariant.
\end{lemma}
\begin{proof}
By Lemma \ref{lem.well-behaved}, we can write $\cS$ as an increasing union of $T$-subbimodules of the form $\cK^0 = \cK \cap S$, where $\cK \subset L^2(S)$ is a finite index $T$-subbimodule. Since for such a $\cK$, the natural map
$$\cH \ot_T \cK^0 \ot_T \cdots \ot_T \cK^0 \recht \cH \ovt_T \cK \ovt_T \cdots \ovt_T \cK$$
is bijective, the lemma follows immediately.
\end{proof}

The $L^2$-Betti numbers of a quasi-regular inclusion were defined in Definition \ref{def.L2-cohom-quasireg} using cohomology with values in the $S$-bimodule $\cHreg$ (see \eqref{eq.coarse}). We now show that, in the case of an irreducible quasi-regular inclusion, they can as well be computed using homology.

\begin{proposition}\label{prop.betti-quasiregular-homology}
Let $S$ be a II$_1$ factor with separable predual and $T \subset S$ an irreducible quasi-regular subfactor. Put $\cS = \QN_S(T)$.
The natural weight $\mu$ on $\cM(T \subset S)$ defined in \eqref{eq.weight-on-End} is a trace if and only if $T \subset S$ is unimodular in the sense of Definition \ref{def.inclusion-unimodular}.
In that case,
$$\bnl(T \subset \cS) = \dim_{\cM(T \subset S)} H_n(T \subset \cS,\cHreg) \; .$$
\end{proposition}
\begin{proof}
Write $\cHreg = L^2(S) \ovt_T \cH \ovt_T L^2(S)$ as in \eqref{eq.coarse}. Denote by $\cC$ the tensor category of finite index $T$-bimodules generated by $L^2(S)$. Let $\cA$ be the associated tube $*$-algebra, with von Neumann algebra $\cA\dpr$ given by Proposition \ref{prop.trace-tube-quasi-reg}. From Lemma \ref{lem.reg-triv-A}, we get that $\cM(T \subset S)$ is isomorphic with
the von Neumann algebra
$$\cM := \bigoplus_{i,j \in \Irr(\cC)} \Bigl( (\cH,i) \ovt p_i \cdot \cA\dpr \cdot p_j \ovt (j,\cH) \Bigr) \; .$$
In Proposition \ref{prop.trace-tube-quasi-reg}, we introduced the weight $\tau$ on $\cA\dpr$. Replacing in the definition of $\tau$, the left trace $\Tr^\ell_i$ by the categorical trace $\Tr_i$, we also have the weight $\tau_1$ on $\cA\dpr$. The isomorphism $\cM(T \subset S)$ sends the weight $\mu$ to the amplification of the weight $\tau_1$ to a weight on $\cM$. Noting that $\tau_1$ is a trace iff $\tau$ is a trace iff $T \subset S$ is unimodular, we conclude that $\mu$ is a trace iff $T \subset S$ is unimodular. In that case, also $\tau_1 = \tau$.

For the rest of the proof, assume that $T \subset S$ is unimodular. Note that all irreducible $T$-bimodules in $\cC$ appear in $\cH$. Choosing one copy of each, we define $\cH_1 = \bigoplus_{i \in \Irr(\cC)} \cH_i$ and put $\cHreg' = L^2(S) \ovt_T \cH_1 \ovt_T L^2(S)$. We have a canonical weight preserving identification between $\End_{S-S}(\cHreg')$ and $\cA\dpr$. Under this identification,
$$\bnl(T \subset \cS) = \dim_{\cA\dpr} H^n(T\subset \cS , \cHreg') \; .$$

Denote by $C^n$, resp.\ $C_n$, the bar complexes defining $H^n(T \subset \cS,\cHreg')$, resp.\ $H_n(T \subset \cS,\cHreg')$. Write $A := \cA\dpr$. Both $C^n$ and $C_n$ are $A$-modules.

For every finite subset $\cF \subset \Irr(\cF)$, define $\cS_\cF := e_\cF(\cS)$ and denote by $\cS_\cF^n$ the $n$-fold relative tensor product $\cS_\cF \ot_T \cdots \ot_T \cS_\cF$.
We also define $C^n_\cF$ as the space of $T$-bimodular maps from $\cS_\cF^n$ to $\cHreg'$, and we define $C^\cF_n$ as the space of $T$-central vectors in $\cHreg' \ot_T \cS^n_\cF$. Note that for every finite set $\cF \subset \Irr(\cF)$, $C_\cF^n$ and $C^\cF_n$ are Hilbert $A$-modules. Choosing an increasing sequence of finite subsets $\cF_k \subset \Irr(\cF)$ with $\bigcup_k \cF_k = \Irr(\cF)$, we can view $C_n$ as the algebraic direct limit of the Hilbert $A$-modules $C_n^{\cF_k}$ and we can view $C^n$ as the inverse limit of the Hilbert $A$-modules $C^n_{\cF_k}$.

For every $n$ and every finite subset $\cF \subset \Irr(\cC)$, consider the finite subset $\overline{\cF^n} \subset \Irr(\cC)$ of all $\al \in \Irr(\cC)$ such that $\albar$ is contained in an $n$-fold tensor product of elements of $\cF$. Because $(\cHreg',i) \cong L^2(\cA \cdot p_i)$, it follows that, as an $A$-module,
$$C_n^\cF \cong \bigoplus_{i \in \overline{\cF^n}} L^2(\cA \cdot p_i) \ot (i \cS_\cF^n ,\eps) \;\; .$$
Since $\tau$ is a trace on $\cA$ and $\tau(p_i) < \infty$, every $L^2(\cA \cdot p_i)$ is an $A$-module of finite $A$-dimension. Every $(i \cS_\cF^n ,\eps)$ is finite dimensional and we conclude that all $C_n^\cF$ have finite $A$-dimension.

The adjoint of $C_n^\cF$ is $C^n_\cF$ and this duality is compatible with the (co)boundary maps (see Remark \ref{rem.homology}). The conclusion
$$\dim_A H_n(T \subset \cS, \cHreg') = \dim_A H^n(T \subset \cS,\cHreg')$$
now follows from the approximation formulae for the $A$-dimensions of direct and inverse limits (see \cite{CG85,Lu02}, and see also \cite[Section A.3]{KPV13} for a self-contained treatment that is directly applicable here).

Amplifying from $\cHreg'$ to $\cHreg$, we get that
$$\bnl(T \subset \cS) = \dim_A H^n(T \subset \cS,\cHreg') = \dim_A H_n(T \subset \cS, \cHreg') = \dim_{\cM(T \subset S)} H_n(T \subset \cS,\cHreg) \; .$$
\end{proof}

We end this section with the following expected result on the $0$'th (co)homology.

\begin{proposition}\label{prop.0-homology-cohomology}
Let $S$ be a II$_1$ factor with separable predual and $T \subset S$ an irreducible quasi-regular subfactor. Put $\cS = \QN_S(T)$. For any Hilbert $S$-bimodule $\cH$, the following holds.
\begin{enumerate}
\item $H^0(T \subset \cS,\cH) = 0$ if and only if $0$ is the only $S$-central vector in $\cH$.
\item $H_0(T \subset \cS,\cH) = 0$ if and only if $\cH_T$ admits no sequence $\xi_n$ of approximately $S$-central unit vectors (meaning that $\lim_n \|x \xi_n - \xi_n x\| = 0$ for all $x \in S$).
\end{enumerate}
\end{proposition}
\begin{proof}
From Definition \ref{def.cohom-quasireg}, we get that $H^0(T \subset \cS,\cH)$ equals the space of $S$-central vectors in $\cH$, so that 1 follows.

Denote by $\cC$ the tensor category generated by the finite index $T$-subbimodules of $L^2(S)$. To prove 2, note that the absence of approximately $S$-central unit vectors in $\cH_T$ is equivalent with the existence of a finite subset $\cG \subset \cS$ satisfying
\begin{equation}\label{eq.my-approx-inv}
\|\xi\| \leq \sum_{x \in \cG} \|x \xi - \xi x\| \quad\text{for all}\;\; \xi \in \cH_T \; .
\end{equation}
For every finite subset $\cF \subset \Irr(\cC)$, we write $\cS_\cF := e_\cF(\cS)$. Note that $\Mor_{T-T}(\cS_\cF,\cH)$ is a Hilbert space and $\delta_\cF : \cH_T \recht \Mor_{T-T}(\cS_\cF,\cH) : (\delta_\cF \xi)(x) = x \xi - \xi x$ is a bounded operator. As in Remark \ref{rem.homology}, the adjoint $\delta_\cF^*$ can be identified with the restriction of the boundary operator $\partial : (\cH \ot_T \cS)_T \recht \cH_T$ to the Hilbert space $(\cH \ot_T \cS_\cF)_T$.

With this notation, the existence of a finite subset $\cG \subset \cS$ satisfying \eqref{eq.my-approx-inv} is equivalent with the existence of a finite subset $\cF \subset \Irr(\cC)$ and an $\eps > 0$ such that $\eps \|\xi\| \leq \|\delta_\cF(\xi)\|$ for all $\xi \in \cH_T$. By the open mapping theorem and the above description of $\delta_\cF^*$, this is equivalent with the existence of a finite subset $\cF \subset \Irr(\cC)$ such that $\partial((\cH \ot_T \cS_\cF)_T) = \cH_T$. By the Baire category theorem, this is equivalent with $\partial((\cH \ot_T \cS)_T) = \cH_T$, i.e.\ with $H_0(T \subset \cS,\cH) = 0$.
\end{proof}

\section{\boldmath A Hochschild type (co)homology of the tube $*$-algebra}\label{sec.hom-cohom-tube}

\subsection{(Co)homology of irreducible quasi-regular inclusions}

Fix an irreducible quasi-regular inclusion of II$_1$ factors $T \subset S$ together with a tensor category $\cC$ of finite index $T$-bimodules containing all finite index $T$-subbimodules of $L^2(S)$. Put $\cS = \QN_S(T)$. Denote by $\cA$ the associated tube $*$-algebra.

In Theorem \ref{thm.tube-vs-corr}, we constructed a bijection between nondegenerate right Hilbert $\cA$-modules $\cK$ and Hilbert $S$-bimodules $\cH$ that, as a $T$-bimodule, are a direct sum of $T$-bimodules contained in $\cC$. In Definitions \ref{def.cohom-quasireg} and \ref{def.homology-quasiregular}, we defined the (co)homology spaces $H^n(T \subset \cS,\cH)$ and $H_n(T \subset \cS,\cH)$. The following is the main result of this section, identifying this (co)homology theory with purely algebraic (co)homology for the tube algebra $\cA$.

For this, we make use of the trivial left and right $\cA$-modules $\cE^r$ and $\cE^\ell$ defined in Lemma \ref{lem.reg-triv-A} and Remark \ref{rem.left-triv}. Whenever $\cK$ is a right Hilbert $\cA$-module, we define $\cK^0$ as the linear span of the Hilbert subspaces $\cK \cdot p_i$, $i \in \Irr(\cC)$.

\begin{theorem} \label{thm.tor-ext}
Let $\cH$ be the Hilbert $S$-bimodule that corresponds to the right Hilbert $\cA$-module $\cK$ through Theorem \ref{thm.tube-vs-corr}. There are natural isomorphisms
$$H_n(T \subset \cS,\cH) \cong \Tor_n^\cA(\cK^0,\cE^\ell) \quad\text{and}\quad H^n(T \subset \cS,\cH) \cong \Ext^n_\cA(\cE^r,\cK^0) \; .$$
\end{theorem}

Theorem \ref{thm.tor-ext} says the following. Whenever
$\cdots \recht L_1 \recht L_0 \recht \cE^\ell \recht 0$
is a resolution of $\cE^\ell$ by projective left $\cA$-modules $L_k$, we can compute $H_n(T \subset \cS,\cH)$ as the homology of
$$\cdots  \recht \cK^0 \ot_\cA L_1  \recht \cK^0 \ot_\cA L_0 \; .$$
Whenever
$\cdots \recht R_1 \recht R_0 \recht \cE^r \recht 0$
is a resolution of $\cE^r$ by projective right $\cA$-modules $R_k$, we can compute $H^n(T \subset \cS,\cH)$ as the cohomology of
$$\Hom_\cA(R_0,\cK^0) \recht \Hom_\cA(R_1,\cK^0) \recht \cdots \; .$$

\begin{proof}[Proof of Theorem \ref{thm.tor-ext}]
We construct a concrete resolution $\cdots \recht \cA^\ell_1 \recht \cA^\ell_0 \recht \cE^\ell \recht 0$ of the left $\cA$-module $\cE^\ell$ and then identify the complex $\cdots  \recht  \cK^0 \ot_\cA \cA^\ell_1  \recht \cK^0 \ot_\cA \cA^\ell_0$ with the bar complex in Definition \ref{def.homology-quasiregular}. Next, we construct a concrete resolution $\cdots \recht \cA^r_1 \recht \cA^r_0 \recht \cE^r \recht 0$ of the right $\cA$-module $\cE^r$ and identify
$\Hom_\cA(\cA^r_0,\cK^0) \recht \Hom_\cA(\cA^r_1,\cK^0) \recht \cdots$ with the bar complex in Definition \ref{def.cohom-quasireg}.

For every $n \geq 0$, define $\cA^\ell_n$ as the algebraic direct sum
$$\cA^\ell_n = \bigoplus_{i \in \Irr(\cC)} (i \cS^{n+1}, \cS)$$
where $\cS^k$ denotes the $k$-fold relative tensor product $\cS \ot_T \cdots \ot_T \cS$. Turn $\cA^\ell_n$ into a left $\cA$-module by
\begin{equation}\label{eq.left-A-module-n}
V \cdot W = (1 \ot m \ot 1^n) (V \ot 1^{n+1}) (1 \ot W) m^*
\end{equation}
for all $V \in (i\cS,\cS j)$ and $W \in (j \cS^{n+1},\cS)$. Note that $\cA^\ell_0 = \cA \cdot p_\eps$.

More generally, we have the isomorphism of left $\cA$-modules
\begin{equation}\label{eq.iso-left-A-modules}
\bigoplus_{i \in \Irr(\cC)} \cA \cdot p_i \ot (i \cS^n,\eps) \recht \cA^\ell_n : V \ot W \mapsto (V \ot 1)(1 \ot W) \; .
\end{equation}
It follows that every $\cA^\ell_n$, $n \geq 0$, is a projective left $\cA$-module.

One checks that the map
\begin{equation}\label{eq.first-A-module-map}
\partial : \cA^\ell_0 \recht \cE^\ell : \partial(V) = (1 \ot a^*)((1 \ot \Delta_\cS^{-1/2})V \ot 1) m^*
\end{equation}
for all $V \in (i \cS, \cS)$ is a left $\cA$-module map. For all $n \geq 1$, we also define the left $\cA$-module maps
\begin{equation}\label{eq.hom-in-resolution}
\begin{split}
& \partial : \cA^\ell_n \recht \cA^\ell_{n-1} : \partial = \sum_{k=0}^n (-1)^k \partial_k \;\; ,\\
\text{where}\;\; & \partial_k(V) = (1^{k+1} \ot m \ot 1^{n-1-k})V \;\;\text{when}\;\; 0 \leq k \leq n-1 \;\; , \\
\text{and}\;\; & \partial_n(V) = (1^{n+1} \ot a^*)((1^{n+1} \ot \Delta_\cS^{-1/2})V \ot 1)m^*
\end{split}
\end{equation}
for all $V \in (i \cS^{n+1},\cS)$.

In this way, we find a complex $\cdots \recht \cA^\ell_1 \recht \cA^\ell_0 \recht \cE^\ell \recht 0$. The maps
\begin{equation}\label{eq.homotopy}
\gamma : \cE^\ell \recht \cA^\ell_0 \quad\text{and}\quad \gamma : \cA^\ell_n \recht \cA^\ell_{n+1}
\end{equation}
given by $\gamma(V) = (1 \ot \delta \ot 1^{n+1})V$ provide a homotopy for this complex, so that we have found a projective resolution of the left $\cA$-module $\cE^\ell$.

Assume now that $\cK$ is a nondegenerate right Hilbert $\cA$-module with corresponding Hilbert $S$-bimodule $\cH$ as in Theorem \ref{thm.tube-vs-corr}. Consider the bar complex $C_n$ defining $H_n(T \subset \cS,\cH)$ as in Definition \ref{def.homology-quasiregular}. By definition, $C_n$ consists of the $T$-central vectors in $\cH \ot_T \cS^n$ and this gives the natural isomorphism
$$C_n \cong \bigoplus_{i \in \Irr(\cC)} (\cH,i) \ot (i\cS^n , \eps) \; .$$
Recall from Theorem \ref{thm.tube-vs-corr} that $(\cH,i) = \cK \cdot p_i$. In combination with \eqref{eq.iso-left-A-modules}, we thus find the isomorphism
$$C_n \cong \bigoplus_{i \in \Irr(\cC)} \cK \cdot p_i \ot (i\cS^n , \eps) \cong \cK^0 \ot_\cA \cA^\ell_n \; .$$
A lengthy, but straightforward computation gives that this is actually an isomorphism between the complexes $(C_n)_{n \geq 0}$ and $(\cK^0 \ot_\cA \cA^\ell_n)_{n \geq 0}$. So, we have found the isomorphism
$$H_n(T \subset \cS,\cH) \cong \Tor_n^\cA(\cK^0,\cE^\ell) \; .$$

Dualizing everything, we find as follows the resolution $\cdots \recht \cA^r_1 \recht \cA^r_0 \recht \cE^r \recht 0$ of the right $\cA$-module $\cE^r$. Define for $n \geq 0$,
$$\cA^r_n = \bigoplus_{i \in \Irr(\cC)} (\cS^{n+1},\cS i)$$
with right $\cA$-module structure
$$V \cdot W = (1^n \ot m)(V \ot 1)(1 \ot W)(m^* \ot 1)$$
for all $V \in (\cS^{n+1},\cS i)$ and $W \in (i \cS, \cS j)$. Note that $\cA^r_0 = p_\eps \cdot \cA$. In general, we have the isomorphism of right $\cA$-modules
\begin{equation}\label{eq.iso-right-A-modules}
\bigoplus_{i \in \Irr(\cC)} (\cS^n , i) \ot p_i \cdot \cA \recht \cA^r_n : V \ot W \mapsto (V \ot 1) W \;\; ,
\end{equation}
so that every $\cA^r_n$, $n \geq 0$, is a projective right $\cA$-module.

The map defined as
$$\partial : \cA_0^r \recht \cE^r : \partial(V) = m(1 \ot \Delta_\cS^{1/2}V)(a \ot 1)$$
for all $V \in (\cS,\cS i)$ is a right $\cA$-module map. Together with the right $\cA$-module maps
\begin{align*}
& \partial : \cA^r_n \recht \cA^r_{n-1} : \partial = \sum_{k=0}^n (-1)^k \partial_k \;\; ,\\
\text{where}\;\; & \partial_0(V) = (a^* \ot 1^n)(1 \ot (\Delta_\cS^{1/2} \ot 1^n)V) (m^* \ot 1) \;\; , \\
\text{and}\;\; & \partial_k(V) = (1^{k-1} \ot m \ot 1^{n-k}) V \;\;\text{for all}\;\; 1 \leq k \leq n \;\; ,
\end{align*}
we find the resolution $\cdots \recht \cA^r_1 \recht \cA^r_0 \recht \cE^r \recht 0$.

Finally, consider the bar complex $C^n$ defining $H^n(T \subset \cS,\cH)$ as in Definition \ref{def.cohom-quasireg}. By definition, $C^n$ consists of all $T$-bimodular linear maps from $\cS^n$ to $\cH$. Using \eqref{eq.iso-right-A-modules}, we identify $\Hom_\cA(\cA^r_n,\cK^0)$ with the direct product
$$\Hom_\cA(\cA^r_n,\cK^0) = \prod_{i \in \Irr(\cC)} \cL\bigl( (\cS^n,i) , \cK \cdot p_i \bigr)$$
of all spaces of linear maps from $(\cS^n,i)$ to $\cK \cdot p_i$. Using the identification $\cK \cdot p_i = (\cH,i)$, we then find the isomorphism
$$\Psi : C^n \recht \Hom_\cA(\cA^r_n,\cK^0) \;\; ,$$
where for $c \in C^n$, $\Psi(c)$ is defined as the collection of linear maps from $(\cS^n,i)$ to $\cK \cdot p_i$ given by $\Psi(c)(V) = c \circ V$, which indeed makes sense because $c$ is a $T$-bimodular map from $\cS^n$ to $\cH$ and thus, $c \circ V$ is an intertwiner from $i \in \Irr(\cC)$ to $\cH$.

It is again straightforward, though a bit tedious, to check that $\Psi$ is an isomorphism of the complexes $(C^n)_{n \geq 0}$ and $(\Hom_\cA(\cA^r_n,\cK^0))_{n \geq 0}$. The conclusion
$$H^n(T \subset \cS,\cH) \cong \Ext^n_\cA(\cE^r,\cK^0)$$
follows.
\end{proof}

Using Proposition \ref{prop.morita}, we obtain as a special case, the following isomorphisms for the (co)ho\-mo\-logy of an SE-inclusion.

\begin{corollary}\label{cor.tor-ext-SE-inclusion}
Let $M$ be a II$_1$ factor and $\cC$ a tensor category of finite index $M$-bimodules having equal left and right dimension. Denote by $T \subset S$ the associated SE-inclusion and let $\cA$ be the tube $*$-algebra of the tensor category $\cC$, together with its co-unit $\counit : \cA \recht \C$.

Whenever $\cH$ is the generalized SE-correspondence associated with the nondegenerate right Hilbert $\cA$-module $\cK$ through Corollary \ref{cor.corr-vs-tube-category}, we have the natural isomorphisms
$$H_n(T \subset \cS,\cH) \cong \Tor_n^\cA(\cK^0,\C) \quad\text{and}\quad H^n(T \subset \cS,\cH) \cong \Ext^n_\cA(\C,\cK^0) \;\; ,$$
where we view $\C$ as a left or right $\cA$-module using the co-unit $\counit$.
\end{corollary}

\subsection[(Co)homology and $L^2$-Betti numbers for rigid C$^*$-tensor categories]{\boldmath (Co)homology and $L^2$-Betti numbers for rigid C$^*$-tensor categories}\label{sec.L2-Betti-category}

The representation theory of a rigid C$^*$-tensor category $\cC$ can be equivalently expressed by unitary half braidings or Hilbert space representations of the tube $*$-algebra $\cA$, see Section \ref{sec.rep-tube-vs-half-braiding}.

By Corollary \ref{cor.tor-ext-SE-inclusion}, the natural (co)homology theory for $\cC$ is precisely the Hochschild (co)homo\-logy of $\cA$ w.r.t.\ the augmentation $\counit : \cA \recht \C$. Given a right Hilbert $\cA$-module $\cK$, these are given by $\Tor_n^\cA(\cK^0,\C)$ and $\Ext^n_\cA(\C,\cK^0)$.
We thus define
$$\bnl(\cC) = \dim_{\cA\dpr} \Tor_n^\cA(L^2(\cA)^0,\C) = \dim_{\cA\dpr} \Ext^n_\cA(\C,L^2(\cA)^0) \; ,$$
where $L^2(\cA)^0$ is the linear span of all $L^2(\cA \cdot p_i)$, $i \in \Irr(\cC)$.

Defining the subalgebra $\cB \subset \cA$ given by $\cB = \lspan \{p_i \mid i \in \Irr(\cC)\}$, the bar resolution for the $\cA$-module $\C$ is
\begin{equation}\label{eq.bar-resolution-tube-category}
\cdots \recht C_2 \recht C_1 \recht C_0 \recht \C \quad\text{with}\;\; C_n = \underbrace{\cA \ot_\cB \cdots \ot_\cB \cA}_{n \;\;\text{factors}} \; \ot_\cB \; \cA \cdot p_\eps
\end{equation}
and $\partial : C_n \recht C_{n-1}$ given by $\partial = \sum_{k=0}^n (-1)^k \partial_k$ where
\begin{align*}
& \partial_k(V_0 \ot \cdots \ot V_n) = V_0 \ot \cdots \ot V_k \cdot V_{k+1} \ot \cdots \ot V_n \quad\text{for}\;\; 0 \leq k \leq n-1 \quad\text{and}\\
&\partial_n(V_0 \ot \cdots \ot V_n) = V_0 \ot \cdots \ot V_{n-1} \cdot p_\eps \; \counit(V_n) \; .
\end{align*}

Also, when $\cC$ is realized as a category of finite index $M$-bimodules having equal left and right dimension, we put $T = M \ovt M\op$ and we consider the SE-inclusion $T \subset S$ as in Section \ref{subsec.SEinclusion}. Combining Proposition \ref{prop.morita} and Theorem \ref{thm.tor-ext}, we get natural isomorphisms between the (co)homology of $T \subset S$ and the (co)homology of $\cC$. In particular, we get that
$$\bnl(\cC) = \bnl(T \subset \cS) \; .$$

To a finite index subfactor $N \subset M$ with Jones tower $N \subset M \subset M_1 \subset \cdots$, we associate the rigid C$^*$-tensor category $\cC_{M-M}$ of finite index $M$-bimodules generated by the $M$-$M$-bimodule $L^2(M_1)$. We similarly have the category $\cC_{N-N}$ of finite index $N$-$N$-bimodules generated by the $N$-$N$-bimodule $L^2(M)$. In Proposition \ref{prop.L2-Betti-Morita}, we prove that $\cC_{N-N}$ and $\cC_{M-M}$ have the same $L^2$-Betti numbers. The reason for this is that $\cC_{M-M}$ and $\cC_{N-N}$ are Morita equivalent and that this Morita equivalence induces a strong Morita equivalence between their tube algebras.

\begin{definition}[{see \cite[Section 4]{Mu01}}]\label{def.Morita-equivalence}
Two rigid C$^*$-tensor categories $\cC_1$ and $\cC_2$ are called \emph{Morita equivalent\footnote{In \cite[Section 4]{Mu01}, the terminology {\it weakly Morita equivalent} is used.}} if there exist nonzero $C^*$-categories $\cC_{12}$ and $\cC_{21}$ with finite dimensional morphism spaces and with tensor functors $\cC_1 \ot \cC_{12} \recht \cC_{12}$, $\cC_{12} \ot \cC_{21} \recht \cC_{11}$, etc., and with a duality functor $\cC_{12} \recht \cC_{21}$ satisfying exactly the same properties as a rigid C$^*$-tensor category.
\end{definition}

For a finite index subfactor $N \subset M$, the tensor categories $\cC_1 = \cC_{N-N}$ and $\cC_2 = \cC_{M-M}$ are Morita equivalent by considering the categories $\cC_{12} = \cC_{N-M}$ and $\cC_{21} = \cC_{M-N}$ of $N$-$M$-bimodules, resp.\ $M$-$N$-bimodules, that are direct sums of subbimodules of some $L^2(M_n)$.

Given a Morita equivalence between $\cC_1$ and $\cC_2$, there is a strong Morita equivalence between the tube $*$-algebras $\cA_1$ and $\cA_2$. This result was obtained in \cite[Section 3]{NY15b} using the notion of $Q$-systems in a tensor category. We provide the following more direct approach.  For all $i \in \Irr(\cC_1)$ and $j \in \Irr(\cC_2)$, define the vector spaces
$$p^1_i \cdot \cA_{12} \cdot p^2_j = \bigoplus_{\al \in \Irr(\cC_{12})} (i \al, \al j) \quad\text{and}\quad p^2_j \cdot \cA_{21} \cdot p^1_i = \bigoplus_{\al \in \Irr(\cC_{21})} (j \al, \al i) \; .$$
The obvious product and adjoint operations are defined in the same way as for the tube $*$-algebra of a rigid C$^*$-tensor category. In this way, we obtain the $*$-algebra
\begin{equation}\label{eq.tube-Morita-entire-A}
\cA = \begin{pmatrix} \cA_1 & \cA_{12} \\ \cA_{21} & \cA_2 \end{pmatrix} \; .
\end{equation}

Similar formulas as in Lemma \ref{lem.sum} still hold~: for every $\al \in \Irr(\cC_{12})$, $i \in \Irr(\cC_1)$ and $j \in \Irr(\cC_2)$, we have
\begin{align*}
& \sum_{k \in \Irr(\cC_2)} \sum_{W \in \onb(i \al, \al k)} \rd(k) \; W \cdot W^\# = \rd(\al) p^1_i \;\; , \\
& \sum_{k \in \Irr(\cC_1)} \sum_{W \in \onb(k \al, \al j)} \rd(k) \; W^\# \cdot W = \rd(\al) p^2_j \;\; .
\end{align*}

It follows that the $\cA_1$-$\cA_2$-bimodule $\cA_{12}$ is a strong Morita equivalence, in the sense that the product maps (inside $\cA$) are isomorphisms
$$\cA_{12} \ot_{\cA_2} \cA_{21} \cong \cA_1 \quad\text{and}\quad \cA_{21} \ot_{\cA_1} \cA_{12} \cong \cA_2 \; .$$

\begin{proposition} \label{prop.L2-Betti-Morita}
If the rigid C$^*$-tensor categories $\cC_1$ and $\cC_2$ are Morita equivalent, then $\bns(\cC_1) = \bns(\cC_2)$ for all $n \in \N$.
\end{proposition}
\begin{proof}
Write $\cM_k = \cA_k\dpr$. The $*$-algebra $\cA$ in \eqref{eq.tube-Morita-entire-A} has a natural semifinite trace $\tau$ and we find the imprimitivity $\cM_2$-$\cM_1$-bimodule $L^2(\cA_{21})$~: the left $\cM_2$ action and the right $\cM_1$ action on $L^2(\cA_{21})$ are each other's commutant.

Given a projective resolution $(L_n)$ of the trivial $\cA_1$-module $\C$, the $L^2$-Betti numbers $\bns(\cC_1)$ are computed as the $\cM_1$-dimension of the homology of the complex $(L^2(\cA_1)^0 \ot_{\cA_1} L_n)_{n \geq 0}$, and thus also as the $\cM_2$-dimension of the homology of the complex $(L^2(\cA_{21})^0 \ot_{\cA_1} L_n)_{n \geq 0}$.

Since the left $\cA_2$-modules $\cA_{21} \ot_{\cA_1} L_n$ form a projective resolution of the $\cA_2$-module $\cA_{21} \ot_{\cA_1} \C$, and since the latter is isomorphic with the trivial $\cA_2$-module, the $L^2$-Betti numbers $\bns(\cC_2)$ can be computed as the $\cM_2$-dimension of the homology of the complex
$$(L^2(\cA_2)^0 \ot_{\cA_2} \cA_{21} \ot_{\cA_1} L_n)_{n \geq 0} \; .$$
Since $L^2(\cA_2)^0 \ot_{\cA_2} \cA_{21} \cong L^2(\cA_{21})^0$, it follows that $\bns(\cC_1) = \bns(\cC_2)$ for all $n \in \N$.
\end{proof}

\subsection{A graphical interpretation of the bar complex associated to the affine category $\mathscr{A}$.}

In this section, we give a diagrammatic description of the homology of the tensor category $\cC$ generated by a finite index subfactor $N \subset M$. Denote by $\mathscr{P}$ the associated standard invariant interpreted as a Jones planar algebra. As we explained above, the resulting homology theory depends only on Ocneanu's tube algebra, which itself has a diagrammatic description purely in terms of the planar algebra $\mathscr{P}$, see Section \ref{sec.affine-cat}. Thus given a planar algebra we can right away associate to it a homology theory, which we now describe explicitly.

Let $\mathscr{P}$ be a planar algebra, which we take to be represented as a quotient of the universal planar algebra $\mathscr{U}$ modulo a set of relations $\mathscr{R}$. For each $k=0,1,2,\dots$ we will denote by $\mathscr{A}_{k}=\mathscr{A}_{k}(\mathscr{P})$ the quotient space
\[
\mathscr{U}_{k}/\mathscr{R}_{k}
\]
where $\mathscr{U}_{k}$ is the linear span of all elements of $\mathscr{U}$ (labeled planar networks) drawn on a two-sphere $S^{2}$ with the ordered collection of $k+2$ disks $r_{0},\dots,r_{k+1}$ removed. We require that the disks $r_{1},\dots,r_{k+1}$ are not connected to any strings of the diagram while $r_{0}$ may be connected to some number of strings of the diagram. Here $\mathscr{R}_{k}$ is the subspace of relations generated by all isotopies as well as those relations obtained by insisting that each relation from $\mathscr{R}$ holds in any contractible disk in $S^{2}\setminus\{r_{0},\dots,r_{k+1}\}$.
For shaded planar algebras (as considered in this paper), we require that the diagram be shaded so that each string is at the boundary between a shaded and an unshaded region.  Note that this shading is specified once we make a choice of shading of two regions: the region marked by $\star$ in the figure below (i.e., the shading of the distinguished interval of $x$) as well as the shading of the region surrounding the point $r_k$ (this shading is actually determined by the parity of the total number of strings of $x$).
We shall denote by $\mathscr{A}_{k}^{(p)}\subset\mathscr{A}_{k}$ the subspace spanned by diagrams having exactly $p$ strings connected to the interior disk $r_{0}$ (note that $p$ has to be even).

In what follows, the disk $r_{0}$ plays a different role than $r_{j}$ for $j\geq1$. To facilitate drawing pictures, we will always identify $S^{2}\setminus\{r_{0},\dots,r_{k+1}\}$ with $\mathbb{R}^{2}\setminus\{r_{0},\dots,r_{k}\}$ by shrinking $r_{1}$ to the point at infinity and shrinking $r_{2},\cdots,r_{k+1}$ to points. We draw an example of an element of $\mathscr{A}_{k}$:

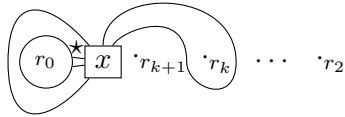
$$
\begin{tikzpicture}
[manyStrings/.style={line width=2.5pt}]
\node [circle,draw] (p0) at (0.25,1) {${}_{r_0}$};
\node [rectangle,draw] (x) at (1,1) {$x$};
\node  at (0.65,1.20){$\star$};
\node [circle] (p1) at (1.75,1) {$\cdot_{r_{k+1}}$};
\node [circle] (p2) at (2.5,1) {$\cdot_{r_{k}}$};
\node [circle] (dots) at (3.25,1) {$\cdots$};
\node [circle] (pkPlus1) at (4,1) {$\cdot_{r_{2}}$};
\draw (x) .. controls +(130:1) and +(90:1) .. (-0.25,1) .. controls +(-90:1) and +(-130:1) ..  (x);
\draw (p0.10) -- (x.170);
\draw (p0.-10) -- (x.-170);
\draw (x) .. controls +(90:1) and +(90:1.0) .. (2.75,1)
	.. controls +(-90:0.6) and +(-90:0.35) .. (2.1,1)
    .. controls +(90:0.75) and +(60:0.5) .. (x);
\end{tikzpicture}$$ where, once again $r_1$ is the point at infinity.  The particular placement of the points $r_1, \ldots, r_{k+1}$ is in principle irrelevant since the whole picture is drawn up to isotopy; however this particular ordering will be useful later in identifying a certain differential complex with a tensor product.

Our convention is that the upper-left corner of $x$ is always the marked boundary segment of $x$. For ease of drawing we denote by thick lines zero or more parallel strings. We will
 frequently omit the labels for the points $r_{1},\dots,r_{k+1}$.

It is not hard to see that,
using the same isotopy as in \eqref{eq:addcup} we can redraw any element in the spanning set for $\mathscr{A}_{k}$ to have the form

$$
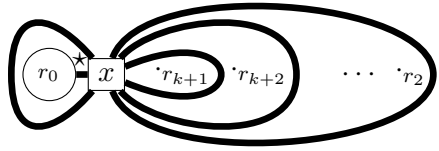
\begin{tikzpicture}
[manyStrings/.style={line width=2.5pt}]
\node [circle,draw] (p0) at (0.25,1) {$_{r_0}$};
\node [rectangle,draw] (x) at (1,1) {$x$};
\node  at (0.65,1.20){$\star$};
\node [circle] (p1) at (2,1) {$\cdot_{r_{k+1}}$};
\node [circle] (p2) at (3,1) {$\cdot_{r_{k+2}}$};
\node [circle] (dots) at (4.35,1) {$\cdots$};
\node [circle] (pkPlus1) at (5,1) {$\cdot_{r_{2}}$};
\draw [manyStrings] (x) .. controls +(130:1) and +(90:1) .. (-0.25,1) .. controls +(-90:1) and +(-130:1) ..  (x);
\draw [manyStrings] (x.west) -- (p0.east);
\draw [manyStrings] (x) .. controls +(23:2.1) and +(157:-2.1) ..  (x);
\draw [manyStrings] (x) .. controls +(50:1) and +(90:1) .. (3.5,1) .. controls +(-90:1) and +(-50:1) .. (x);
\draw [manyStrings] (x) .. controls +(70:1.4) and +(90:1) .. (5.3,1)
.. controls +(-90:1) and +(-70:1.4) .. (x);
\end{tikzpicture}$$

We note that by Lemma \ref{lem.tensor-product-planar}, $\mathscr{A}_{k}$ is exactly the tensor product $\underbrace{\mathscr{A} \ot_\mathscr{B} \cdots \ot_\mathscr{B} \mathscr{A}}_{k} \ot_{\mathscr{B}} \mathscr{A}_{\cdot,0}$ (the last tensor factor accounts for the fact that we do not permit any strings from $x$ to the point $r_{k+1}$). We denote by $d_{j}$ the map from $\mathscr{A}_{k}\to\mathscr{A}_{k-1}$ defined on the spanning set by associating to a diagram in $\mathscr{A}_{k}$ drawn on $S^{2}\setminus\{r_{0},\dots,r_{k+1}\}$ the same diagram but drawn on $S^{2}\setminus\{r_{0},\dots,r_{j},r_{j+2},\dots,r_{k+1}\}$ (these $k+1$ points are ordered as written).
In particular, $d_{0}$ is given by
$$d_0 \left(\begin{array}{c}
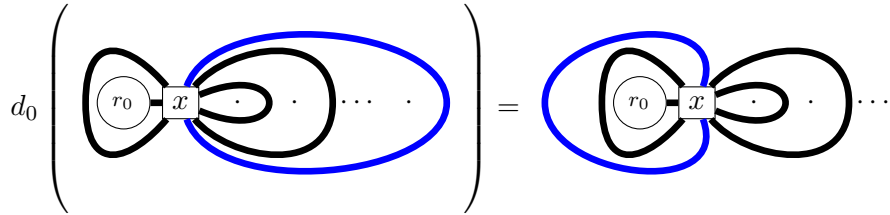
\begin{tikzpicture}
[manyStrings/.style={line width=2.5pt}]
\node [circle,draw] (p0) at (0.25,1) {$_{r_0}$};
\draw [manyStrings] (x) .. controls +(130:1) and +(90:1) .. (-0.25,1) .. controls +(-90:1) and +(-130:1) ..  (x);
\draw [manyStrings] (x.west) -- (p0.east);
\node [rectangle,draw] (x) at (1,1) {$x$};
\node [circle] (p1) at (1.75,1) {$\cdot$};
\node [circle] (p2) at (2.5,1) {$\cdot$};
\node [circle] (dots) at (3.35,1) {$\cdots$};
\node [circle] (pkPlus1) at (4,1) {$\cdot$};
\draw [manyStrings] (x) .. controls +(23:1.6) and +(157:-1.6) ..  (x);
\draw [manyStrings] (x) .. controls +(50:1) and +(90:1) .. (3.0,1) .. controls +(-90:1) and +(-50:1) .. (x);
\draw [manyStrings,color=blue] (x) .. controls +(70:1.4) and +(90:1) .. (4.5,1) .. controls +(-90:1) and +(-70:1.4) .. (x);
\end{tikzpicture}
\end{array}\right)
=
\begin{array}{c}\begin{tikzpicture}
[manyStrings/.style={line width=2.5pt}]
\node [circle,draw] (p0) at (0.25,1) {$_{r_0}$};
\draw [manyStrings] (x) .. controls +(130:1) and +(90:1) .. (-0.25,1) .. controls +(-90:1) and +(-130:1) ..  (x);
\draw [manyStrings] (x.west) -- (p0.east);
\node [rectangle,draw] (x) at (1,1) {$x$};
\node [circle] (p1) at (1.75,1) {$\cdot$};
\node [circle] (p2) at (2.5,1) {$\cdot$};
\node [circle] (dots) at (3.35,1) {$\cdots$};
\draw [manyStrings] (x) .. controls +(23:1.6) and +(157:-1.6) ..  (x);
\draw [manyStrings] (x) .. controls +(50:1) and +(90:1) .. (3.0,1) .. controls +(-90:1) and +(-50:1) .. (x);
\draw [manyStrings,color=blue] (x) .. controls +(70:1.4) and +(90:1) .. (-1,1) .. controls +(-90:1) and +(-70:1.4) .. (x);
\end{tikzpicture}
\end{array}
$$
where we have drawn the strings of $x$ that pass between the point $r_2$ and the point at infinity $r_{1}$ in blue for emphasis.

It is not hard to see that $\sum_{j=0}^{n} (-1)^{j} d_{j}$ corresponds precisely to the differential on the bar complex for the tube algebra as in \eqref{eq.bar-resolution-tube-category}.

\section{\boldmath Vanishing of $L^2$-Betti numbers for amenable quasi-regular inclusions}

Given a tracial von Neumann algebra $(S,\tau)$ with von Neumann subalgebra $T \subset S$, there are several notions of amenability, which for a crossed product inclusion $T \subset T \rtimes \Gamma$ all coincide with the amenability of the group $\Gamma$. In \cite[Definition 3.2.1]{Po86}, the amenability of $S$ relative to $T$ was defined as the trivial $S$-bimodule $L^2(S)$ being weakly contained in the relative coarse $S$-bimodule $L^2(S) \ovt_T L^2(S)$, meaning that there exists a sequence of vectors $\xi_n \in L^2(S) \ovt_T L^2(S)$ such that $\lim_n \|x \xi_n - \xi_n x\| = 0$ and $\lim_n \langle x \xi_n,\xi_n\rangle = \tau(x)$ for all $x \in S$.

When $T \subset S$ is an irreducible quasi-regular subfactor, the above weak containment does not exactly correspond to weak containment of tube algebra representations, where the natural requirement is that the vectors $\xi_n$ can be chosen $T$-central. So for our purposes, the following relative amenability notion is more natural, and we prove in Proposition \ref{prop.characterization-amenable-inclusion} that it indeed has the expected properties.

\begin{definition}\label{def.amenable-quasi-reg-inclusion}
Let $S$ be a II$_1$ factor with irreducible quasi-regular subfactor $T \subset S$. The inclusion $T \subset S$ is called \emph{amenable} if there exists a net of unital, trace preserving, completely positive, $T$-bimodular maps $\vphi_i : S \recht S$ such that $\lim_i \|\vphi_i(x) - x\|_2 = 0$ for all $x \in S$ and such that $\vphi_i$ has finite rank for every fixed $i$, in the sense that the closure of $\vphi_i(S)$ is a finite index $T$-subbimodule of $L^2(S)$.
\end{definition}

\begin{proposition}\label{prop.characterization-amenable-inclusion}
Let $T \subset S$ be an irreducible quasi-regular inclusion of II$_1$ factors. Let $\cC$ be the tensor category generated by the finite index $T$-subbimodules of $L^2(S)$. Denote by $\cA$ the associated tube $*$-algebra. Also denote $\cS = \QN_S(T)$.

The following statements are equivalent.
\begin{enumerate}
\item The inclusion $T \subset S$ is amenable in the sense of Definition \ref{def.amenable-quasi-reg-inclusion}.
\item There exists a net of $T$-central, approximately $S$-central unit vectors in $L^2(S) \ovt_T L^2(S)$.
\item The trivial representation of $\cA$ on $\cE^r$ is weakly contained in the regular representation of $\cA$ on $L^2(p_\eps \cdot \cA)$.
\item There exists a net $\xi_i \in (\cS,\cS) = p_\eps \cdot \cA \cdot p_\eps$ satisfying
$$\|\xi_i\|_{2,\tau} = 1 \;\;\text{for all $i$, and}\quad \lim_i \| V \cdot \xi_i - \Tr(V) \xi_i \|_{2,\tau} = 0 \;\;\text{for all $V \in p_\eps \cdot \cA \cdot p_\eps$}$$
where we use the notation $\|V\|_{2,\tau} := \sqrt{\tau(V^\# \cdot V)}$.
\end{enumerate}

When $S$ has separable predual, these statements are moreover equivalent with the non vanishing of $H_0(T \subset \cS, L^2(S) \ovt_T L^2(S))$.
\end{proposition}
\begin{proof}
The proposition follows immediately from Corollary \ref{cor.cp-maps-vs-states-tube-algebra} and Remark \ref{rem.reg-triv-as-cp-maps-and-states}, and by taking the adjoint to prove the equivalence of 3 and 4. The final statement then follows from Proposition \ref{prop.0-homology-cohomology}.
\end{proof}

The goal of this section is to prove that $\bns(T \subset S) = 0$ for all $n \geq 1$ whenever $T \subset S$ is amenable. We can however only prove this under a possibly stronger, but natural amenability condition on the inclusion $T \subset S$, formulated as a F{\o}lner condition. As we prove in Lemma \ref{lem.strongly-amenable-vs-amenable} at the end of this section, this F{\o}lner property is equivalent with amenability as in Definition~\ref{def.amenable-quasi-reg-inclusion} for several families of quasi-regular inclusions, including all SE-inclusions of extremal subfactors, all crossed product inclusions and all inclusions of the form $N \rtimes \Lambda \subset N \rtimes \Gamma$ where $\Lambda < \Gamma$ is an almost normal subgroup.

Before defining the F{\o}lner property of an arbitrary irreducible quasi-regular inclusion, consider the SE-inclusion $T \subset S$ of an extremal subfactor $N \subset M$ with standard invariant $\cG_{N,M}$. In \cite{Po93}, the standard invariant $\cG_{N,M}$ is called amenable if the weighted principal graph $(\Gamma_{N,M},\vec{v})$ satisfies a F{\o}lner condition as a weighted graph. In \cite{Po94a} and \cite[Theorem 5.3]{Po99}, it is proved that this F{\o}lner condition is equivalent with the amenability of $S$ relative to $T$, and also with the Kesten type condition $\|\Gamma_{N,M}\|^2 = [M:N]$. Note that this last property is also used to define amenability of an abstract rigid C$^*$-tensor category. Reformulating the F{\o}lner property for the weighted principal graph directly in terms of the SE-inclusion $T \subset S$, we define as follows the F{\o}lner property for an arbitrary irreducible quasi-regular inclusion.

So, fix an irreducible quasi-regular inclusion of II$_1$ factors $T \subset S$. Denote by $\cC$ the tensor category generated by the finite index $T$-subbimodules of $L^2(S)$, and write $\cS = \QN_S(T)$. For every $\al \in \Irr(\cC)$, we denote by $e_\al \in (\cS,\cS)$ the orthogonal projection of $L^2(S)$ onto the span of the $T$-subbimodules of $L^2(S)$ that are isomorphic with $\al$. We write $\cS_\al := e_\al(\cS)$. Given a finite symmetric subset $\cG \subset \Irr(\cC)$, we turn $\Irr(\cC)$ into a locally finite graph by putting an edge between $\al,\be \in \Irr(\cC)$ if there exists a $\gamma \in \cG$ such that $e_\beta(\cS_\gamma \cS_\al)$ is nonzero\footnote{Equivalently, we put an edge between $\al$ and $\beta$ iff $\tau(\cS_{\overline{\beta}} \cS_\cG \cS_\al) \neq \{0\}$. Taking the complex conjugate, the latter is equivalent with $\tau(\cS_{\albar} \cS_{\overline{\cG}} \cS_\be) \neq \{0\}$. So, for a symmetric subset $\cG \subset \Irr(\cC)$, we obtain a symmetric condition in $\al,\be$.}. For every finite subset $\cF \subset \Irr(\cC)$, we then denote by $\partial_\cG(\cF)$ the boundary of $\cF$ in this graph, which we define as the union of the inner and outer boundary of $\cF$. By definition, $\partial_\cG(\cF)$ consists of all $\al \in \cF$ that are connected by an edge to some $\beta \not\in \cF$, and of all $\al \not\in \cF$ that are connected to some $\beta \in \cF$. We define the measure $\mu$ on $\Irr(\cC)$ by $\mu(\{\al\}) = \Tr^r(e_\al)$ for every $\al \in \Irr(\cC)$.

\begin{definition}\label{def.strongly-amenable-quasi-reg-inclusion}
An irreducible quasi-regular inclusion of II$_1$ factors $T \subset S$ is said to have the \emph{F{\o}lner property} if the following holds: for every finite subset $\cG \subset \Irr(\cC)$ and every $\eps > 0$, there exists a finite subset $\cF \subset \Irr(\cC)$ such that
$$\mu(\partial_\cG(\cF)) < \eps \, \mu(\cF) \; .$$
\end{definition}

In Lemma \ref{lem.strongly-amenable-vs-amenable}, we prove that the F{\o}lner property implies amenability in the sense of Definition~\ref{def.amenable-quasi-reg-inclusion} and that the converse holds for large classes of quasi-regular inclusions. We do not know whether the converse holds in general.

We now turn to $L^2$-Betti numbers. So, we assume that $S$ has separable predual and that the inclusion $T \subset S$ is unimodular, i.e.\ that all $T$-subbimodules of $L^2(S)$ have equal left and right dimension, so that the $L^2$-Betti numbers $\bns(T \subset S)$ are well defined.

\begin{theorem}\label{thm.vanish-amenable}
If $T \subset S$ satisfies the F{\o}lner property, then $\bns(T \subset S) = 0$ for all $n \geq 1$.
\end{theorem}

In combination with Lemma \ref{lem.strongly-amenable-vs-amenable} and Remark \ref{rem.amenable-SE-inclusion} below, we then get the following.

\begin{corollary}\label{cor.vanish-amenable-tensor-category}
For every amenable rigid C$^*$-tensor category $\cC$, we have that $\bns(\cC) = 0$ for all $n \geq 1$.
\end{corollary}

Before proving Theorem \ref{thm.vanish-amenable}, we introduce the following notation and prove a general vanishing lemma for $L^2$-Betti numbers.

\begin{definition}\label{def.L2-Betti-map}
Let $(\cM,\tau)$ be a von Neumann algebra with a normal semifinite faithful trace $\tau$ and let $\cA \subset \cM$ be a dense $*$-subalgebra contained in the domain of $\tau$.

For every $V \in M_{m,n}(\C) \ot \cA$, viewed as an operator from $L^2(\cM)^{\oplus n}$ to $L^2(\cM)^{\oplus m}$ given by left multiplication, we define
$$\bel(V) = \dim_\cM\bigl(\Ker V \cap (\Ker V \cap \cA^{\oplus n})^\perp\bigr) \; .$$
\end{definition}

Note that $\bes(V) = 0$ iff $\Ker V \cap \cA^{\oplus n}$ is dense in $\Ker V$.

The proof of the following lemma is basically identical to the end of the proof of \cite[Theorem 6.37]{Lu02}.

\begin{lemma}\label{lem.general-vanishing}
Let $T \subset S$ be an irreducible quasi-regular inclusion that is unimodular. Let $\cA$ be the tube $*$-algebra as above. If $\bes(V) = 0$ for every $i \in \Irr(\cC)$ and every $V \in M_{m,k}(\C) \ot p_i \cdot \cA \cdot p_i$, then $\bns(T \subset S) = 0$ for all $n \geq 1$.
\end{lemma}
\begin{proof}
Write $\cM := \cA\dpr$ and $\cM_i = p_i \cdot \cM \cdot p_i$ for every $i \in \Irr(\cC)$. We first prove that $\bes(V) = 0$ for all $V \in M_{m,k}(\C) \ot \cA$. For this, it suffices to prove that for all $i \in \Irr(\cC)$, we have that $\Ker V \cap (\cA \cdot p_i)^{\oplus k}$ is dense in $\Ker V \cap L^2(\cA \cdot p_i)^{\oplus k}$.

Take $\xi \in L^2(\cA \cdot p_i)^{\oplus k}$ with $V \xi = 0$. Since the image of the multiplication map
$$\cA \cdot p_i \underset{p_i \cdot \cA \cdot p_i}{\ot} L^2(p_i \cdot \cA \cdot p_i) \recht L^2(\cA \cdot p_i)$$
is a dense right $\cM_i$-submodule, we can take a projection $q \in \cM_i$ that is arbitrarily close to $1$ such that
$$\xi q = W \cdot \eta \quad\text{with}\quad W \in M_{k,l}(\C) \ot \cA \cdot p_i \;\; , \;\; \eta \in L^2(\cM_i)^{\oplus l} \; .$$
Since $V \xi q = 0$, it follows that $\eta$ belongs to the kernel of
$$U := W^\# \cdot V^\# \cdot V \cdot W \in M_{l,l}(\C) \ot p_i \cdot \cA \cdot p_i \; .$$
Because we assumed that $\bes(U) = 0$, we can take a sequence $\eta_r \in \Ker U \cap (p_i \cdot \cA \cdot p_i)^{\oplus l}$ such that $\|\eta - \eta_r\|_2 \recht 0$. Since $\Ker U = \Ker (V \cdot W)$, it follows that $W \cdot \eta_r$ is a sequence in $\Ker V \cap (\cA \cdot p_i)^{\oplus k}$ that converges to $\xi q$. Since $q$ is arbitrarily close to $1$, we have proved that $\Ker V \cap (\cA \cdot p_i)^{\oplus k}$ is dense in $\Ker V \cap L^2(\cA \cdot p_i)^{\oplus k}$.

We now prove that $\bns(T \subset S) = 0$ for all $n \geq 1$. Up to taking adjoints, it follows from Theorem \ref{thm.tor-ext} that we can choose an exact sequence $\cdots \recht L_1 \recht L_0 \recht \cE^r \recht 0$ of right $\cA$-modules in which every $L_n$ is isomorphic with a direct sum of right $\cA$-modules of the form $p_i \cdot \cA$, $i \in \Irr(\cC)$ and such that $\bns(T \subset S)$ is computed as the $\cM$-dimension of the homology of
$$\cdots \recht L_1 \ot_\cA {}_0 L^2(\cA) \recht L_0 \ot_\cA {}_0 L^2(\cA) \; ,$$
where ${}_0 L^2(\cA)$ is the linear span of all $L^2(p_i \cdot \cA)$, $i \in \Irr(\cC)$.

To prove the lemma, it thus suffices to prove that whenever
$$L_2 \overset{f}{\recht} L_1 \overset{g}{\recht} L_0$$
is a sequence of right $\cA$-modules such that $\Ker g = \Im f$ and such that both $L_1$ and $L_0$ are isomorphic with a direct sum of $p_i \cdot \cA$, $i \in \Irr(\cC)$, then the induced sequence of right $\cM$-modules given by
$$\Ltil_2 \overset{\ftil}{\recht} \Ltil_1 \overset{\gtil}{\recht} \Ltil_0 \quad\text{where}\;\; \Ltil_n = L_n \ot_\cA {}_0 L^2(\cA) \; ,$$
satisfies
$$\dim_\cM \frac{\Ker \gtil}{\Im \ftil} = 0 \; .$$
Write $L_1$ as the union of an increasing sequence of $\cA$-submodules $R_k \subset L_1$ such that each $R_k$ is the direct sum of finitely many $p_i \cdot \cA$. Write $\Rtil_k = R_k \ot_\cA {}_0 L^2(\cA)$. Then $\Ker \gtil / \Im \ftil$ is the union of the increasing sequence of $\cM$-submodules
\begin{equation}\label{eq.quotient-module}
\frac{\Ker \gtil \cap \Rtil_k}{\Im \ftil \cap \Rtil_k} \; .
\end{equation}
It thus suffices to prove that for every $k$, the $\cM$-module in \eqref{eq.quotient-module} has $\cM$-dimension zero.

Fix $k$ and write $R_k = \bigoplus_{j=1}^n p_{i_j} \cdot \cA$. Then $\Rtil_k = \bigoplus_{j=1}^n L^2(p_{i_j} \cdot \cA)$. The restriction of $g$ to $R_k$ can be viewed as left multiplication by some $V \in M_{m,n}(\C) \ot \cA$. Then also the restriction of $\gtil$ to $\Rtil_k$ is given by left multiplication with the same element $V$. Since $\Ker g = \Im f$, we have $\Ker g \cap R_k \subset \Im \ftil \cap \Rtil_k$. Since $\bes(V) = 0$, the $\cM$-module
$$\frac{\Ker \gtil \cap \Rtil_k}{\Ker g \cap R_k}$$
has $\cM$-dimension zero. The $\cM$-module in \eqref{eq.quotient-module} is a quotient of this and hence also has $\cM$-dimension zero.
\end{proof}

We now prove Theorem \ref{thm.vanish-amenable} by showing that the assumptions of Lemma \ref{lem.general-vanishing} hold for inclusions with the F{\o}lner property. The proof follows the same lines as the proof of the same result for discrete groups, see \cite{CG85} and \cite[Theorem 6.37]{Lu02}.

\begin{theorem}\label{thm.density}
Let $S$ be a II$_1$ factor with separable predual and $T \subset S$ an irreducible quasi-regular subfactor. Assume that the inclusion $T \subset S$ is unimodular and satisfies the F{\o}lner property. Let $\cA$ be the tube $*$-algebra as above. For every $V \in M_{m,n}(\C) \ot \cA$, we have $\bes(V) = 0$.
\end{theorem}

To prove Theorem \ref{thm.density}, we need some notation and a lemma. For every finite subset $\cF \subset \Irr(\cC)$, we denote by $P_\cF$ the orthogonal projection of $L^2(\cA)$ onto the closed linear span of all subspaces $(1 \ot e_\cF)(i \cS, \cS j)(e_\cF \ot 1)$, $i,j \in \Irr(\cC)$. We write $\cS_\cF = e_\cF(\cS)$ and abbreviate $(1 \ot e_\cF)(i \cS, \cS j)(e_\cF \ot 1) = (i \cS_\cF, \cS_\cF j)$. For every finite subset $\cI \subset \Irr(\cC)$, we also denote $p_\cI := \sum_{i \in \cI} p_i$, which is a projection in $\cA$. We let $\cA$ act by left multiplication operators on $L^2(\cA)$. Then, the projections $p_\cI$ and $P_\cF$ commute and their product $p_\cI P_\cF$ is a finite rank projection. Finally, denote by $D$ the (possibly unbounded) positive self-adjoint operator on $L^2(\cA)$ given by multiplication with $\rdl(j)$ on $(i \cS, \cS j)$.

\begin{lemma}\label{lem.trace-formula}
Assume that $T \subset S$ is unimodular. For every finite subset $\cF \subset \Irr(\cC)$ and for every $V \in \cA$, acting by left multiplication on $L^2(\cA)$, we have
\begin{equation}\label{eq.trace-formula}
\Tr(V P_\cF D) = \mu(\cF)^2 \tau(V) \; ,
\end{equation}
where $\Tr$ denotes the operator trace on $B(L^2(\cA))$.
\end{lemma}

Note that for every $V \in \cA$, there exists a finite set $\cI \subset \Irr(\cC)$ such that $V = V \cdot p_\cI$. Therefore, $V P_\cF D$ is a finite rank operator and its trace is well defined.
Denoting by $\cA\dpr$ the von Neumann algebra generated by $\cA$ acting by left multiplication on $L^2(\cA)$, we get by continuity that \eqref{eq.trace-formula} holds (and is meaningful) for all $V \in p_\cI \cdot \cA\dpr \cdot p_\cI$ and all finite subsets $\cI \subset \Irr(\cC)$.

\begin{proof}[Proof of Lemma \ref{lem.trace-formula}]
Since $T \subset S$ is unimodular, we have that $\Delta_\cS = 1$ and $\tau$ is a trace on $\cA$. Fix a finite subset $\cI \subset \Irr(\cC)$ such that $V = V \cdot p_\cI$. Since the left scalar product on $(i \cS, \cS j)$ coincides with the scalar product on $(i \cS, \cS j)$ given by viewing it as a subspace of $L^2(\cA)$, we have
$$\Tr(V P_\cF D) = \sum_{i \in \cI, j \in \Irr(\cC)} \sum_{W \in \onb_\ell(i \cS_\cF, \cS_\cF j)} \;\; \rdl(j) \; \langle V \cdot W, W \rangle \; .$$
Using that $\tau$ is a trace and using Lemma \ref{lem.other-sum}, we get that
\begin{align*}
\Tr(V P_\cF D) &= \sum_{i \in \cI, j \in \Irr(\cC)} \sum_{W \in \onb_\ell(i \cS_\cF, \cS_\cF j)} \;\; \rdl(j) \; \tau(V \cdot W \cdot W^\#) \\ & = \mu(\cF)^2 \sum_{i \in \cI} \tau(V \cdot p_i) = \mu(\cF)^2 \tau(V) \; .
\end{align*}
\end{proof}

We can now prove Theorem \ref{thm.density}.

\begin{proof}[Proof of Theorem \ref{thm.density}]
Take a finite subset $\cI \subset \Irr(\cC)$ such that $V \in M_{m,n}(\C) \ot \cA \cdot p_\cI$. Then, $\Ker(V)$ is the direct sum of $((1-p_\cI) \cdot L^2(\cA))^{\oplus n}$ and the kernel of the restriction of $V$ to $L^2(p_\cI \cdot \cA)^{\oplus n}$. It thus suffices to prove that $\Ker(V) \cap (p_\cI \cdot \cA)^{\oplus n}$ is dense in $\Ker(V) \cap L^2(p_\cI \cdot \cA)^{\oplus n}$.

Define $q$, resp.\ $p$, as the orthogonal projection of $L^2(p_\cI \cdot \cA)^{\oplus n}$ onto $\Ker(V) \cap L^2(p_\cI \cdot \cA)^{\oplus n}$, resp.\ onto
the closure of $\Ker(V) \cap (p_\cI \cdot \cA)^{\oplus n}$. We have $p \leq q$ and we must prove that $p = q$. Note that $p$ and $q$ are projections in $M_n(\C) \ot p_\cI \cdot \cA\dpr \cdot p_\cI$ acting by left multiplication. We prove that $(\Tr \ot \tau)(q-p) = 0$.

Take a large enough finite subset $\cG \subset \Irr(\cC)$ such that all matrix entries of $V \in M_{m,n}(\cC) \ot \cA$ belong to the linear span of $(i \cS_{\cG}, \cS_{\cG} j)$ with $i \in \Irr(\cC)$, $j \in \cI$. Choose $\cG$ symmetrically, i.e.\ $\overline{\cG} = \cG$. Choose $\eps > 0$. Because $T \subset S$ has the F{\o}lner property, we can take a non empty finite subset $\cF \subset \Irr(\cC)$ such that $\mu(\partial_\cG(\cF)) < \eps \mu(\cF)$.

Write $\cF' = \partial_\cG(\cF)$. Using the same notations $P_{\cF}$, $p_\cI$ and $D$ to denote their $n$-fold direct sum as operators on $L^2(\cA)^{\oplus n}$, we claim that if $\xi \in L^2(p_\cI \cdot \cA)^{\oplus n}$ belongs to $\Ker(V)$ and satisfies $P_{\cF'}(\xi) = 0$, then $P_{\cF}(\xi)$ belongs to $\Ker(V) \cap (p_\cI \cdot \cA)^{\oplus n}$. To prove this claim, we first show that $P_{\cF}(\xi) \in (p_\cI \cdot \cA)^{\oplus n}$. This follows because $P_{\cF}(\xi) = P_\cF p_\cI (\xi)$ and because $P_\cF p_\cI$ is a finite rank projection with image in $(p_\cI \cdot \cA)^{\oplus n}$.

The definition of $\cF' = \partial_\cG(\cF)$ implies that
$$e_\cF \, m \, (e_\cG \ot 1) = e_\cF \, m \, (e_\cG \ot e_{\cF \cup \cF'}) \quad\text{and}\quad m \, (e_\cG \ot e_{\cF \setminus \cF'}) = e_\cF \, m \, (e_\cG \ot e_{\cF \setminus \cF'}) \; .$$
Since for every $W \in (i \cS,\cS j) \subset L^2(\cA)$, we have that $P_\cF(W) = (1 \ot e_\cF) W (e_\cF \ot 1)$, it follows that whenever $\xi \in L^2(\cA)^{\oplus n}$ and $P_{\cF'}(\xi) = 0$, we have that
$$P_\cF(V \cdot \xi) = P_\cF(V \cdot P_{\cF \cup \cF'}(\xi)) = P_\cF(V \cdot P_{\cF \setminus \cF'}(\xi)) = V \cdot P_{\cF \setminus \cF'}(\xi) = V \cdot P_{\cF}(\xi) \; .$$
So, if moreover $\xi \in \Ker(V)$, then also $P_{\cF}(\xi)$ belongs to $\Ker(V)$ and the claim is proved.

The claim means that the range projection of $P_{\cF} (q \wedge (1-P_{\cF'}))$ is smaller than $p$. Therefore, $(q-p) P_{\cF} (q \wedge (1-P_{\cF'}) = 0$ and, in particular,
\begin{equation}\label{eq.st-1}
\Tr\bigl(D (q-p) P_{\cF} (q \wedge (1-P_{\cF'})) \bigr) = 0 \; .
\end{equation}
Denote by $w$ the polar part of $q P_{\cF'}$ and note that $w w^* = q - (q \wedge (1-P_{\cF'}))$. It thus follows from \eqref{eq.st-1} that
$$\Tr(D (q-p) P_{\cF} q) = \Tr(D (q-p) P_{\cF} w w^*) \; .$$
Since both $q$ and $P_{\cF'}$ commute with $D$, the same holds for $w$ and we get that
$$\Tr(D (q-p) P_{\cF} q) = \Tr(D w^*(q-p) P_\cF w) \; .$$
Since $w^* w \leq p_\cI P_{\cF'}$, it follows that
\begin{equation}\label{eq.st1}
|\Tr(D (q-p) P_{\cF} q)| \leq \|D p_\cI P_{\cF'}\|_{1,\Tr} \; \|w^*(q-p) P_\cF w\| \leq \|D p_\cI P_{\cF'}\|_{1,\Tr} \; .
\end{equation}
Taking into account that all our operators act on the $n$-fold direct sum $L^2(\cA)^{\oplus n}$ and that by unimodularity, $\rdl(i) = \rdr(i) = \rd(i)$ for all $i \in \cC$, we have
\begin{align*}
\|D p_\cI P_{\cF'}\|_{1,\Tr} &= n \sum_{i \in \cI, j \in \Irr(\cC)} \rd(j) \; \dim (i \cS_{\cF'} , \cS_{\cF'} j) \\
& = n \sum_{i \in \cI, j \in \Irr(\cC)} \rd(j) \; \mult(j, \cS_{\overline{\cF'}} i \cS_{\cF'}) \\
& = n \sum_{i \in \cI} \rd(i) \; \rd(\cS_{\cF'})^2 = n \, \rd(\cI) \, \mu(\cF')^2 \leq n \, \rd(\cI) \, \eps^2 \, \mu(\cF)^2 \; ,
\end{align*}
where $\rd(\cI) := \sum_{i \in \cI} \rd(i)$.

In combination with \eqref{eq.st1} and the observation that $D$ and $q$ commute, we get that
$$|\Tr((q-p) P_\cF D)| \leq n \, \rd(\cI) \, \eps^2 \, \mu(\cF)^2 \; .$$
Using Lemma \ref{lem.trace-formula}, we conclude that
$$(\Tr \ot \tau)(q-p) \leq n \, \rd(\cI) \, \eps^2 \; .$$
Since $\eps > 0$ is arbitrary, it follows that $q-p=0$.
\end{proof}

We finally prove that in many interesting cases, amenability and strong amenability are actually equivalent conditions.

\begin{lemma} \label{lem.strongly-amenable-vs-amenable}
Let $T \subset S$ be an irreducible, quasi-regular inclusion of II$_1$ factors. If $T \subset S$ has the F{\o}lner property, then $T \subset S$ is amenable in the sense of Definition \ref{def.amenable-quasi-reg-inclusion}.

The converse holds under the extra assumptions that every irreducible $T$-subbimodule of $L^2(S)$ appears with multiplicity one and that there exists a $\delta > 0$ such that for all $\al,\be,\gamma \in \Irr(\cC)$,
$$\text{either}\;\; e_\be \; m(e_\gamma \ot e_\al)m^* = 0 \quad\text{or}\quad  e_\be \; m(e_\gamma \ot e_\al)m^* \geq \delta \, e_\beta \; .$$
Here, $e_\al$ denotes the projection of $L^2(S)$ onto the $T$-subbimodule equivalent with $\al$.

The above extra assumptions are satisfied for all SE-inclusions, as well as all inclusions of the form $N \rtimes \Lambda \subset N \rtimes \Gamma$ given by an almost normal subgroup $\Lambda < \Gamma$ and an outer action of $\Gamma$ on a II$_1$ factor $N$.
\end{lemma}

\begin{proof}
First assume that $T \subset S$ has the F{\o}lner property. Take a net of finite subsets $\cF_i \subset \Irr(\cC)$ such that $\mu(\partial_\cG(\cF_i)) / \mu(\cF_i)$ tends to zero for every finite subset $\cG \subset \Irr(\cC)$. For every finite subset $\cF \subset \Irr(\cC)$, define as before the element $e_\cF := \sum_{\al \in \cF} e_\al$ in $(\cS,\cS) = p_\eps \cdot \cA \cdot p_\eps$. Define $\xi_i := \mu(\cF_i)^{-1/2} e_{\cF_i} \Delta_\cS^{1/2}$. By construction, $\|\xi_i\|_{2,\tau} = 1$ for all $i$. We claim that for every $V \in (\cS,\cS)$, we have
$$\lim_i \langle V \cdot \xi_i , \xi_i \rangle = \Tr(V) \; .$$
Once this claim is proved, the amenability of $T \subset S$ follows from Proposition \ref{prop.characterization-amenable-inclusion}. Fix $V \in (\cS,\cS)$. Take a finite subset $\cG \subset \Irr(\cC)$ such that $V = e_\cG V$. Write $\cF'_i = \cF_i \setminus \partial_\cG(\cF_i)$. Since
$$\|\xi_i - \mu(\cF_i)^{-1/2} e_{\cF'_i} \Delta_\cS^{1/2}\|_{2,\tau}^2 \leq \frac{\mu(\partial_\cG(\cF_i))}{\mu(\cF_i)} \recht 0 \; ,$$
it suffices to prove that
$$\lim_i \frac{1}{\mu(\cF_i)} \; \langle V \cdot (e_{\cF'_i}\Delta_\cS^{1/2}), e_{\cF_i}\Delta_\cS^{1/2} \rangle = \Tr(V) \; .$$
Since $m(e_\cG \ot e_{\cF'_i}) = e_{\cF_i} m (e_\cG \ot e_{\cF'_i})$ and since the co-unit is a character on $p_\eps \cdot \cA \cdot p_\eps$, we get for every $i$ that
\begin{align*}
\langle V \cdot (e_{\cF'_i}\Delta_\cS^{1/2}), e_{\cF_i}\Delta_\cS^{1/2} \rangle &= \Tr^\ell(e_{\cF_i}\Delta_\cS^{1/2} \; m(V \ot e_{\cF'_i}\Delta_\cS^{1/2})m^*) \\
& = \Tr^\ell(m(V\Delta_\cS^{1/2} \ot e_{\cF'_i}\Delta_\cS )m^*) \\
& = \Tr^\ell(V\Delta_\cS^{1/2}) \, \Tr^r(e_{\cF'_i}) = \Tr(V) \, \mu(\cF'_i) \; .
\end{align*}
Dividing by $\mu(\cF_i)$ and taking the limit over $i$, the claim follows. So, $T \subset S$ is amenable.

Conversely, assume that $T \subset S$ is an amenable inclusion and that the extra conditions in the lemma are satisfied. Fix a finite, symmetric subset $\cG \subset \Irr(\cC)$. We prove that there exists a sequence of finite subsets $\cF_n \subset \Irr(\cC)$ such that $\mu(\partial_\cG(\cF_n))/\mu(\cF_n) \recht 0$.

By Proposition \ref{prop.characterization-amenable-inclusion}, we find a net of vectors $\xi_i \in p_\eps \cdot \cA \cdot p_\eps$ such that $\|\xi_i\|_{2,\tau} = 1$ for all $i$ and
\begin{equation}\label{eq.asymptotic-inv-xi-i}
\lim_i \|V \cdot \xi_i - \Tr(V) \xi_i \|_\tau = 0 \quad\text{for all $V \in p_\eps \cdot \cA \cdot p_\eps$.}
\end{equation}
Since the irreducible $T$-subbimodules of $L^2(S)$ appear with multiplicity one, $(\cS,\cS)$ is the linear span of the elements $e_\al$. Define the finitely supported functions $\eta_i : \Irr(\cC) \recht \C$ such that
$$\xi_i = \sum_{\al \in \Irr(\cC)} \frac{\eta_i(\al)}{\sqrt{\Tr^\ell(e_\al)}} \, e_\al \; .$$
Since $\langle e_\gamma \cdot e_\al,e_\beta \rangle = \Tr^\ell(e_\be m(e_\gamma \ot e_\al)m^*)$ and $e_\gamma^\# = e_{\overline{\gamma}}$, the infinite matrix
$$T^\cG_{\be,\al} = \sum_{\gamma \in \cG} \frac{\Tr^\ell(e_\be m(e_\gamma \ot e_\al)m^*)}{\sqrt{\Tr^\ell(e_\be) \; \Tr^\ell(e_\al)}}$$
is symmetric. By \eqref{eq.asymptotic-inv-xi-i}, we have that $\lim_i \|T^\cG(\eta_i) - \Tr(e_{\cG}) \eta_i\|_2 = 0$, where $\|\,\cdot\,\|_2$ is computed w.r.t.\ the counting measure on $\Irr(\cC)$. Writing $v_\al = \sqrt{\mu(\al)} = \sqrt{\Tr^r(e_\al)}$, we get that
\begin{align*}
\sum_{\al \in \Irr(\cC)} T^\cG_{\be,\al} \, v_\al &= \sum_{\al \in \Irr(\cC)} \rdr(\al)^{1/2} \; \frac{\Tr^\ell(e_\al m(e_\cG \ot e_\be) m^*)}{\rdl(\al)^{1/2} \, \rdl(\be)^{1/2}} \\
&= \sum_{\al \in \Irr(\cC)} \Delta_\al^{1/2} \; \rdl(\be)^{-1/2} \; \Tr^\ell(e_\al m(e_\cG \ot e_\be) m^*) \\
&= \sum_{\al \in \Irr(\cC)} \rdl(\be)^{-1/2} \; \Delta_\be^{1/2} \; \Tr^\ell(e_\al m(\Delta_\cS^{1/2} e_\cG \ot e_\be) m^*) \\
&= \rdl(\be)^{-1/2} \; \Delta_\be^{1/2} \Tr^\ell(m(\Delta_\cS^{1/2} e_\cG \ot e_\be) m^*) \\
&= \rdl(\be)^{-1/2} \; \Delta_\be^{1/2} \; \Tr^\ell(\Delta_\cS^{1/2} e_\cG) \; \Tr^\ell(e_\be) = \Tr(e_\cG) \; v_\be \; .
\end{align*}
So, the formal equality $T^\cG(v) = \Tr(e_\cG) v$ holds.

Whenever $\al,\be \in \Irr(\cC)$ and $T^\cG_{\be,\al} \neq 0$, it follows from our assumptions that
$$T^\cG_{\be,\al} \geq \delta \frac{\sqrt{\rdl(\be)}}{\sqrt{\rdl(\al)}} \; .$$
In that case, we find in particular a $\gamma \in \cG$ such that the bimodule $\be$ is contained in $\gamma \ot \al$. Then also $\al$ is contained in $\overline{\gamma} \ot \be$ and we conclude that $\rdl(\al) \leq \Tr^\ell(e_\cG) \, \rdl(\be)$. We conclude that all non zero entries of $T^\cG_{\be,\al}$ are bounded from below by $\delta / \Tr^\ell(e_\cG)$. Also note that $\partial_\cG(\cF_n)$ is the boundary of $\cF_n$ in the graph structure on $\Irr(\cC)$ in which $\al,\be$ are connected by an edge if and only if $T^\cG_{\be,\al} > 0$. So, it follows from \cite[Corollary 2.1]{Po97b} that there exists a sequence of non empty finite subsets $\cF_n \subset \Irr(\cC)$ such that $\mu(\partial_\cG(\cF_n))/\mu(\cF_n) \recht 0$. So, $T \subset S$ has the F{\o}lner property.

Next consider the case of SE-inclusions. So we are given a II$_1$ factor $M$ and a tensor category $\cC_1$ of finite index $M$-bimodules having equal left and right dimension. We write $T = M \ovt M\op$ and we have the SE-inclusion $T \subset S$. By construction, for all $\al \in \Irr(\cC_1)$, we have a $T$-bimodular map
$$\delta_\al : \cH_\al \ot \overline{\cH_\al} \recht L^2(S)$$
satisfying $\delta_\al^* \delta_\al = \rd(\al)^{-1} \, 1$ and $\delta_\al^* \delta_\be = 0$ if $\al \neq \be$. Also,
$$m \circ (\delta_\gamma \ot \delta_\al) = \sum_{\be \in \Irr(\cC_1)} \;\; \sum_{V \in \onb(\be,\gamma \al)} \; \rd(\be) \; \delta_\be \circ (V \ot \overline{V}) \; .$$
Note that also the $T$-bimodules contained in $L^2(S)$ have equal left and right dimension. We denote by $e_\al \in (\cS,\cS)$ the minimal projection corresponding to the irreducible $T$-bimodule $\cH_\al \ot \overline{\cH_\al}$. So, $e_\al = \rd(\al) \delta_\al \delta_\al^*$. A direct computation then gives
$$e_\be m(e_\gamma \ot e_\al)m^* = \frac{\mult(\be,\gamma \ot \al) \; \rd(\gamma) \; \rd(\al)}{\rd(\be)} \; e_\be \; .$$
This expression is non zero if and only if $\be$ is contained in $\gamma \ot \al$. In that case, we have $\rd(\be) \leq \rd(\gamma) \, \rd(\al)$ and it follows that $e_\be m(e_\gamma \ot e_\al)m^* \geq e_\be$.

Finally consider an almost normal subgroup $\Lambda < \Gamma$ and an outer action of $\Gamma$ on a II$_1$ factor $N$. Put $T = N \rtimes \Lambda$ and $S = N \rtimes \Gamma$.
For every double coset $\gamma \in \Lambda \backslash \Gamma / \Lambda$, denote by $\cH(\gamma)$ the $\|\,\cdot\,\|_2$-closed linear span of $\{x u_g \mid x \in N , g \in \gamma\}$. Each $\cH(\gamma)$ is an irreducible $T$-subbimodule of $L^2(S)$ and these $T$-subbimodules are mutually inequivalent. Fix $\al,\be,\gamma \in \Lambda \backslash \Gamma / \Lambda$. Take $a_1,\ldots,a_k \in \al$ such that $\al$ is the disjoint union of the cosets $\Lambda a_i$. Then, the map
$$\cU : \cH(\gamma) \ot \C^k \recht \cH(\gamma) \ot_T \cH(\al) : \cU(\xi \ot e_i) = \xi \ot u_{a_i}$$
is unitary. Write $\cW = m \circ \cU$ and note that $\cW(\xi \ot e_i) = \xi u_{a_i}$. For all $x \in N$ and $g \in \Gamma$, we have that
$$\cW^*(x u_g) = \sum_{i, g a_i^{-1} \in \gamma} x u_g u_{a_i}^* \ot e_i \; .$$
Thus, writing $\be = \Lambda b \Lambda$ for some $b \in \Gamma$, we get that
$$e_\be \; m (e_\gamma \ot e_\al) m^* = e_\be \cW \cW^* = \#\{i \mid b a_i^{-1} \in \gamma \} \; e_\be \; ,$$
which is either $0$ or at least $e_\be$.
\end{proof}

\begin{remark}\label{rem.amenable-SE-inclusion}
\begin{enumerate}
\item When $T \subset S$ is the SE-inclusion of a tensor category $\cC_1$ of finite index $M$-bimodules having equal left and right dimension, then the amenability of the inclusion $T \subset S$ is equivalent with the amenability of $\cC_1$ as a rigid C$^*$-tensor category. This follows immediately from Proposition \ref{prop.characterization-amenable-inclusion} and the identification between $p_\eps \cdot \cA \cdot p_\eps$ and the fusion algebra of $\cC_1$.

\item When $\Lambda < \Gamma$ is an almost normal subgroup and $\Gamma \actson N$ is an outer action on the II$_1$ factor $N$, then the amenability of the inclusion of $T = N \rtimes \Lambda$ inside $S = N \rtimes \Gamma$ is equivalent with the amenability of the Schlichting completion $G$, which is the locally compact group defined as the closure of $\Gamma$ inside the permutation group of $\Gamma / \Lambda$ equipped with the topology of pointwise convergence. Indeed, the closure of $\Lambda$ inside $G$ is a compact open subgroup of $G$ and there is a natural identification of $K \backslash G/K$ with $\Lambda \backslash \Gamma/\Lambda$. Condition~4 in Proposition \ref{prop.characterization-amenable-inclusion} then becomes the existence of a net of unit vectors $\xi_i \in L^2(K \backslash G/K)$ such that viewing $\xi_i$ as vectors in $L^2(G)$, we have $\lim_i \langle \lambda_g \xi_i,\xi_i \rangle = 1$ for every $g \in G$. This last condition is equivalent with the amenability of $G$.
\end{enumerate}
\end{remark}

\section{Computations and properties}

\subsection[The $0$'th $L^2$-Betti number]{\boldmath The $0$'th $L^2$-Betti number}

\begin{proposition}\label{prop.0-L2-Betti}
Let $T \subset S$ be an irreducible, quasi-regular, unimodular inclusion of II$_1$ factors. Then, $\bes_0(T \subset S) = [S:T]^{-1}$.
\end{proposition}

\begin{proof}
Let $\cC$ be the tensor category of finite index $T$-bimodules generated by $L^2(S)$. Denote by $\cA$ the tube $*$-algebra and write $\cM = \cA\dpr$. Using the resolution in the proof of Theorem \ref{thm.tor-ext}, we get that $\bes_0(T \subset S)$ equals the $\cM$-dimension of the left $\cM$-module $\cK_0$, where $\cK_0$ is defined as the orthogonal complement in $L^2(\cA \cdot p_\eps)$ of the image of the map
$$\bigoplus_{i \in \Irr(\cC)} (i \cS^2,\cS) \recht L^2(\cA \cdot p_\eps) : V \mapsto (1 \ot m) V - (1 \ot 1 \ot a^*)(V \ot 1)m^* \; .$$
Define $q_\eps \in \cZ(\cM)$ as the central support of $p_\eps$. Then, $L^2(\cA \cdot p_\eps) = q_\eps \cdot L^2(\cA \cdot p_\eps)$ and thus, $\cK_0 = q_\eps \cdot \cK_0$. So, writing $\cM_\eps := p_\eps \cdot \cM \cdot p_\eps$, it follows from \cite[Lemma A.15]{KPV13} that $\dim_\cM \cK_0 = \dim_{\cM_\eps} (p_\eps \cdot \cK_0)$. Note that $p_\eps \cdot \cK_0$ equals the orthogonal complement in $L^2(\cM_\eps)$ of the image of the map
$$(\cS^2,\cS) \recht L^2(\cM_\eps) : V \mapsto m V - (1 \ot a^*) (V \ot 1) m^* \; .$$

For every $W \in (\cS,\cS)$ and $V \in (\cS^2,\cS)$, we have that
\begin{align*}
\langle m V - (1 \ot a^*)(V \ot 1)m^*,W \rangle &= \Tr(W^* m V) - \Tr(m(W^* \ot 1) V) \\ &= \Tr\bigl((W^* m - m(W^* \ot 1))V\bigr) \; .
\end{align*}
Note that $L^2(\cM_\eps)$ is the completion of $(\cS,\cS)$ with respect to the scalar product $\langle V,W \rangle = \Tr(V W^*)$ and thus, $L^2(\cM_\eps)$ can be viewed as the space of bounded $T$-bimodular operators $V : L^2(S) \recht L^2(S)$ with the property that $\Tr(VV^*) < \infty$. Then $W \in p_\eps \cdot \cK_0$ if and only if we have
$$W^* m(1 \ot e_\cF) = m(1 \ot e_\cF) (W^* \ot 1)$$
for every finite subset $\cF \subset \Irr(\cC)$, and where the equality holds as bounded $T$-bimodular operators from $L^2(S) \ovt_T L^2(S)$ to $L^2(S)$. Composing with $\delta \ot 1$, where $\delta : L^2(T) \recht L^2(S)$ is the inclusion map as before, we find that
$$W^* e_\cF = m(1 \ot e_\cF) (W^* \delta \ot 1) \; .$$
Since $W^* \delta$ is a $T$-bimodular map from $L^2(T)$ to $L^2(S)$, it must be a multiple of $\delta$. We conclude that
$$W^* e_\cF = \tau(W^*) \, e_\cF$$
for all finite subsets $\cF \subset \Irr(\cC)$. This means that $p_\eps \cdot \cK_0$ consists of the multiples of the identity operator on $L^2(S)$. If $[S:T] = \infty$, also $\Tr(1) = \infty$ and it follows that $p_\eps \cdot \cK_0 = \{0\}$. Then also $\bes_0(T \subset S) = 0$. If $[S:T] < \infty$, we write $z_\eps = [S:T]^{-1} 1$ and we get that $z_\eps$ is a minimal central projection in $\cM_\eps$ projecting onto $p_\eps \cdot \cK_0$. So in that case, $\bes_0(T \subset S) = \tau(z_\eps) = [S:T]^{-1}$.
\end{proof}

\begin{corollary}\label{cor.zero-betti-category}
If $\cC$ is a rigid C$^*$-tensor category, we have $\bes_0(\cC) = \Bigl(\sum_{\al \in \Irr(\cC)} \rd(\al)^2 \Bigr)^{-1}$.
\end{corollary}

\begin{corollary}
Let $T \subset S$ be an irreducible, unimodular inclusion of II$_1$ factors with finite index. Then,
$$\bel_0(T \subset S) = [S:T]^{-1} \quad\text{and}\quad \bnl(T \subset S) = 0 \quad\text{for all}\;\; n \geq 1 \; .$$
\end{corollary}
\begin{proof}
By Theorem \ref{thm.tor-ext}, we compute $\bns(T \subset S)$ by tensoring an exact sequence of $\cA$-modules with $L^2(p_\eps \cdot \cA)^0 \ot_\cA \cdot\;$. In the finite index case, $L^2(p_\eps \cdot \cA)^0 = p_\eps \cdot \cA$ and the sequence stays exact. So, $\bns(T \subset S) = 0$ for all $n \geq 1$, while $\bes_0(T \subset S)$ was computed in Proposition \ref{prop.0-L2-Betti}.
\end{proof}

\subsection[The $L^2$-Betti numbers of free products]{\boldmath The $L^2$-Betti numbers of free products}

\begin{proposition}\label{prop.free-product}
Let $T \subset S_1$ and $T \subset S_2$ be nontrivial, irreducible, quasi-regular, unimodular inclusions of II$_1$ factors. Define $S$ as the amalgamated free product $S = S_1 *_T S_2$ w.r.t.\ the trace preserving conditional expectations. Assume that $T \subset S$ is still irreducible. Then,
\begin{align*}
\bel_0(T \subset S) &= 0 \;\; , \\
\bel_1(T \subset S) &= \bel_1(T \subset S_1) + \bel_1(T \subset S_2) + 1 - (\bel_0(T \subset S_1) + \bel_0(T \subset S_2)) \quad\text{and} \\
\bel_n(T \subset S) &= \bel_n(T \subset S_1) + \bel_n(T \subset S_2) \quad\text{for all}\;\; n \geq 2 \; .
\end{align*}
\end{proposition}

Denote by $\cC_i$ the tensor category of finite index $T$-bimodules generated by $T \subset S_i$. If $\cC_1$ and $\cC_2$ are free, in the sense that every alternating tensor product of $T$-bimodules in $\Irr(\cC_1) \setminus \{\eps\}$ and $\Irr(\cC_2) \setminus \{\eps\}$ stays irreducible, then $T \subset S_1 *_T S_2$ is automatically irreducible.

\begin{corollary}
If a rigid C$^*$-tensor category $\cC$ is the free product of non trivial full tensor subcategories $\cC_1$ and $\cC_2$, then
\begin{align*}
\bel_0(\cC) &= 0 \;\; , \\
\bel_1(\cC) &= \bel_1(\cC_1) + \bel_1(\cC_2) + 1 - (\bel_0(\cC_1) + \bel_0(\cC_2)) \quad\text{and} \\
\bel_n(\cC) &= \bel_n(\cC_1) + \bel_n(\cC_2) \quad\text{for all}\;\; n \geq 2 \; .
\end{align*}
\end{corollary}

\begin{proof}[Proof of Proposition \ref{prop.free-product}]
Let $\cC$ be the tensor category of finite index $T$-bimodules generated by $T \subset S$. Write $\cS = \QN_S(T)$ and $\cS_k = \QN_{S_k}(T)$ for $k=1,2$. Denote by $\cA$ the associated tube $*$-algebra. Consider the left $\cA$-module $\cE^\ell$ as in Remark \ref{rem.left-triv}. To compute $\bns(T \subset S)$, we will construct a specific resolution of $\cE^\ell$. Define the $\cA$-module map
$\partial : \cA \cdot p_\eps \recht \cE^\ell$ given by \eqref{eq.first-A-module-map}. For $k=1,2$ and $n \geq 1$, consider the $n$-fold relative tensor product $\cS_k^n = \cS_k \ot_T \cdots \ot_T \cS_k$, define
$$\cA^k_n = \bigoplus_{i \in \Irr(\cC)} (i \cS \cS_k^n , \cS)$$
and turn $\cA^k_n$ into a left $\cA$-module as in \eqref{eq.left-A-module-n}. We have the left $\cA$-module isomorphisms
$$\bigoplus_{i \in \Irr(\cC_k)} \cA \cdot p_i \ot (i \cS_k^n, \eps) \recht \cA^k_n : V \ot W \recht (V \ot 1)(1 \ot W)$$
so that every $\cA^k_n$ is a projective left $\cA$-module. The same formulas as in \eqref{eq.hom-in-resolution} yield $\cA$-module maps
$$\partial : \cA^k_1 \recht \cA \cdot p_\eps \quad\text{and}\quad \partial : \cA^k_n \recht \cA^k_{n-1} \quad\text{for all}\;\; n \geq 2 \; .$$
Taking direct sums, we find the complex
\begin{equation}\label{eq.resolution-free-product}
\cdots \recht \cA^1_3 \oplus \cA^2_3 \recht \cA^1_2 \oplus \cA^2_2 \recht \cA^1_1 \oplus \cA^2_1 \recht \cA \cdot p_\eps \recht \cE \recht 0 \; .
\end{equation}
We claim that the complex in \eqref{eq.resolution-free-product} is exact. The exactness at the position $\cA^1_n \oplus \cA^2_n$ for $n \geq 2$ follows by using the same homotopy as in \eqref{eq.homotopy}. The exactness at the position $\cA^1_1 \oplus \cA^2_1$ follows in the same way, once we prove that $\partial(V) = \partial(W)$ for $V \in \cA^1_1$ and $W \in \cA^2_1$ implies that $\partial(V) = \partial(W) = 0$. To prove this statement, define $\cS'_1$ as the linear span of $T$ and all alternating products of $\cS_1 \ominus T$ and $\cS_2 \ominus T$ that end with $\cS_2 \ominus T$. Note that the multiplication map defines a unitary $T$-bimodular operator $L^2(\cS'_1) \ovt_T L^2(\cS_1) \recht L^2(\cS)$. We similarly define $\cS'_2$. In this way, we identify $(i\cS,\cS)$ with $(i \cS'_1 \cS_1,\cS)$ and with $(i \cS'_2 \cS_2,\cS)$. Viewing $(i \cS'_1,\cS)$ as a subspace of $(i \cS,\cS)$ by the inclusion $\cS'_1 \subset \cS$ and using the multiplication maps $m_k : \cS \ot_T \cS_k \recht \cS$, we define the linear maps
\begin{align*}
& D_1 : \cA \cdot p_\eps \recht \cA \cdot p_\eps : V \mapsto (1^2 \ot a_1^*)(V \ot 1)m_1^* \quad\text{for all}\;\; V \in (i \cS'_1 \cS_1,\cS) \;\; , \\
& D_2 : \cA \cdot p_\eps \recht \cA \cdot p_\eps : V \mapsto (1^2 \ot a_2^*)(V \ot 1)m_2^* \quad\text{for all}\;\; V \in (i \cS'_2 \cS_2,\cS) \;\; .
\end{align*}
Note that for every $V \in (i\cS, \cS)$, we have $D_1(V) \in (i \cS'_1, \cS)$. We also have $D_1(V) = V$ for all $V \in (i \cS'_1,\cS)$. Analogous statements hold for $D_2$.

Using the embedding $(i \cS'_k,\cS) \subset (i \cS,\cS)$ as the homotopy, we get that
$$\cA^k_1 \overset{\partial}{\recht} \cA \cdot p_\eps \overset{D_k}{\recht} \cA \cdot p_\eps$$
is exact.

Writing $\cS_k^\circ = \cS_k \ominus T$, we have that $\cS$ is the linear span of $T$ and all alternating products in $\cS_1^\circ$ and $\cS_2^\circ$. When e.g.\ $V \in (i \cS_1^\circ \cS_2^\circ \cS_1^\circ , \cS)$, then $D_1(V)$ belongs to $(i \cS_1^\circ \cS_2^\circ, \cS)$, and $D_2(D_1(V))$ belongs to $(i \cS_1^\circ, \cS)$, so that $D_1(D_2(D_1(V)))$ belongs to $(i,\cS)$ and equals $\partial(V)$, where we viewed $(i, \cS) \subset (i \cS,\cS)$ through the identification of $W$ and $(1 \ot \delta)W$. All further $(D_2 D_1)^n(V)$ with $n \geq 2$ equal $\partial(V)$. In general, for all $V \in \cA \cdot p_\eps$, the sequences $(D_1 D_2)^n(V)$ and $(D_2 D_1)^n(V)$ become constantly equal to $\partial(V)$ for $n$ large enough.

So, defining for all $n \geq 1$, the maps
\begin{align*}
& S_n : \cA \cdot p_\eps \recht \cA \cdot p_\eps : S_n(V) = D_1(V) - D_2(D_1(V)) + \cdots + D_1((D_2 D_1)^{n-1}(V)) - (D_2 D_1)^n(V)\\
& T_n : \cA \cdot p_\eps \recht \cA \cdot p_\eps : T_n(V) = D_2(V) - D_1(D_2(V)) + \cdots + D_2((D_1 D_2)^{n-1}(V)) - (D_1 D_2)^n(V)
\end{align*}
also the sequences $S_n(V)$ and $T_n(V)$ become constant for $n$ large enough, and we denote this `limit' as $S(V)$, resp.\ $T(V)$. When $V \in (i \cS_1' \cS_1^\circ,\cS)$, we have $T_n(V) = V - S_{n-1}(V) - (D_1 D_2)^n(V)$, so that $S(V) + T(V) = V - \partial(V)$. The same formula holds when $V \in (i \cS_2' \cS_2^\circ,\cS)$ and when $V \in (i,\cS)$. So, we get that
$$S(V) + T(V) = V - \partial(V) \quad\text{for all}\;\; V \in \cA \cdot p_\eps \; .$$

We are now ready to prove the exactness of \eqref{eq.resolution-free-product} at the position $\cA^1_1 \oplus \cA^2_1$. Assume that $\partial(V) = \partial(W)$ for $V \in \cA^1_1$ and $W \in \cA^2_1$. Since $\cA_2^k \recht \cA_1^k \recht \cA \cdot p_\eps$ is exact, it suffices to prove that $\partial(V) = 0$. We have that $D_1(\partial(V)) = 0$. But also $D_2(\partial(W)) = 0$ and thus, $D_2(\partial(V)) = 0$. Both together imply that $S(\partial(V)) = 0 = T(\partial(V))$, so that $\partial(V) = \partial(\partial(V)) = 0$.

Finally, we have to prove that \eqref{eq.resolution-free-product} is exact at the position $\cA \cdot p_\eps$. Take $V \in \cA \cdot p_\eps$ with $\partial(V) = 0$. Then, $V = S(V) + T(V)$. It suffices to prove that $S(V) \in \partial(\cA^2_1)$ and that $T(V) \in \partial(\cA^1_1)$. For this, it suffices to prove that $D_2(S(V)) = 0$ and $D_1(T(V)) = 0$. Since $D_2(W - D_2(W)) = 0$ for all $W \in \cA \cdot p_\eps$, the definition of $S$ immediately implies that $D_2(S(V))= 0$. Similarly, we get that $D_1(T(V)) = 0$.

So, we have proved that \eqref{eq.resolution-free-product} is a resolution of $\cE$ by projective left $\cA$-modules. Write $\cM = \cA\dpr$. By Theorem \ref{thm.tor-ext}, the $L^2$-Betti numbers of $T \subset S$ can thus be computed as the $\cM$-dimension of the homology of the complex
$$\cdots \recht \cB^1_3 \oplus \cB^2_3 \recht \cB^1_2 \oplus \cB^2_2 \recht \cB^1_1 \oplus \cB^2_1 \recht L^2(\cA \cdot p_\eps) \; ,$$
where
$$\cB^k_n = \bigoplus_{i \in \Irr(\cC_k)} L^2(\cA \cdot p_i) \ot (i \cS_k^n, \eps)$$
and the boundary maps are the natural extensions of the boundary maps in \eqref{eq.resolution-free-product}. Denote by $\cA_k$ the tube $*$-algebra of $T \subset S_k$ and $\cC_k$. Write $\cM_k = \cA_k\dpr$. By Theorem \ref{thm.tor-ext}, the $L^2$-Betti numbers of $T \subset S_k$ are computed as the $\cM_k$-dimension of the homology of the complex
$$\cdots \recht \cL^k_3 \recht \cL^k_2 \recht \cL^k_1 \recht L^2(p_\eps \cdot \cA_k \cdot p_\eps) \; ,$$
where
$$\cL^k_n = \bigoplus_{i \in \Irr(\cC_k)} L^2(\cA_k \cdot p_i) \ot (i \cS_k^n,\eps)$$
and the boundary maps are as above.

For $k=1,2$, define the projection $q_k \in \cM$ given by $q_k = \sum_{i \in \Irr(\cC_k)} p_i$. For any chain complex $(L_n)_{n \geq 0}$ of $\cM_k$-modules, the $\cM_k$-dimension of the homology of $(L_n)_{n \geq 0}$ equals the $\cM$-dimension of the homology of the complex $(\cM \cdot q_k \ot_{\cM_k} L_n)_{n \geq 0}$. Since for every $i \in \Irr(\cC_k)$, the multiplication map $\cM \cdot q_k \ot_{\cM_k} L^2(\cA_k \cdot p_i) \recht L^2(\cA \cdot p_i)$ is a dimension isomorphism, it follows that $\bns(T \subset S_k)$ can be computed as the $\cM$-dimension of the homology of $(\cB^k_n)_{n \geq 0}$.

We then immediately get that
$$\bnl(T \subset S) = \bnl(T \subset S_1) + \bnl(T \subset S_2) \quad\text{for all}\;\; n \geq 2 \; .$$
We also get that
\begin{equation}\label{eq.beta-one}
\bel_1(T \subset S) = \bel_1(T \subset S_1) + \bel_1(T \subset S_2) + \dim_\cM(\partial(\cB_1^1) \cap \partial(\cB_1^2)) \; .
\end{equation}
Since both $S_k \neq T$, we get that all alternating products of $\cS_1^\circ$ and $\cS_2^\circ$ define nonzero orthogonal $T$-subbimodules of $L^2(S)$. Therefore, $T \subset S$ has infinite index and $\bes_0(T \subset S)= 0$.

For $k=1,2$, define the projection $z_k \in p_\eps \cdot \cA \cdot p_\eps$ given by $z_k = 0$ if $[S_k:T]=\infty$ and otherwise given as $[S_k:T]^{-1}$ times the projection of $L^2(S)$ onto $L^2(S_k)$ viewed as an element in $p_\eps \cdot \cA \cdot p_\eps = (\cS,\cS)$.
Write $\cM_\eps = p_\eps \cdot \cM \cdot p_\eps$. Exactly as in the proof of Proposition \ref{prop.0-L2-Betti}, we get that
\begin{align*}
\dim_\cM(\partial(\cB_1^1) \cap \partial(\cB_1^2)) & = \dim_{\cM_\eps}(p_\eps \cdot \partial(\cB_1^1) \cap p_\eps \cdot \partial(\cB_1^2)) \\
&= 1 - \dim_{\cM_\eps}(p_\eps \cdot \partial(\cA^1_1)^\perp + p_\eps \cdot \partial(\cA^2_1)^\perp) = 1 - \tau(z_1 \vee z_2) \\
& = 1-(\tau(z_1) + \tau(z_2) - \tau(z_1 \wedge z_2)) \; .
\end{align*}
Since $T \subset S$ has infinite index, we have $z_1 \wedge z_2 = 0$ and we conclude that
$$\dim_\cM(\partial(\cB_1^1) \cap \partial(\cB_1^2)) = 1 - [S_1:T]^{-1} - [S_2,T]^{-1} = 1 - \bel_0(T \subset S_1) - \bel_0(T \subset S_2) \; .$$
Together with \eqref{eq.beta-one}, we have found the required formula for $\bes_1(T \subset S)$.
\end{proof}

\subsection[The $L^2$-Betti numbers of tensor products]{\boldmath The $L^2$-Betti numbers of tensor products}

\begin{proposition}
If $T_1 \subset S_1$ and $T_2 \subset S_2$ are irreducible, quasi-regular, unimodular inclusions of II$_1$ factors, then
$$\bel_n(T_1 \ovt T_2 \subset S_1 \ovt S_2) = \sum_{k=0}^n \bel_k(T_1 \subset S_1) \, \bel_{n-k}(T_2 \subset S_2) \; .$$
If $\cC_1$ and $\cC_2$ are rigid C$^*$-tensor categories, we have a similar formula for $\bns(\cC_1 \times \cC_2)$.
\end{proposition}

\begin{proof}
The tube algebra $\cA$ of $T \subset S$ is canonically isomorphic with the algebraic tensor product $\cA_1 \ot \cA_2$ of the tube algebras $\cA_k$ of $T_k \subset S_k$. Also the trivial left $\cA$-module $\cE$ is the tensor product $\cE_1 \ot \cE_2$ of the trivial left $\cA_k$-modules $\cE_k$. Given resolutions $(L^k_n)$ of $\cE_k$ by projective left $\cA_k$-modules, we build the bicomplex of $\cA$-modules $(L^1_n \ot L^2_m)_{n,m}$. The total complex
$$L_n = \bigoplus_{k=0}^n (L^1_k \ot L^2_{n-k})$$
is a resolution of $\cE$ by projective left $\cA$-modules. The computation of $\bns(T \subset S)$ can then be done exactly as in the proof of \cite[Theorem 2.1]{Ky09}.
\end{proof}

\subsection[The $L^2$-Betti numbers of the Temperley-Lieb-Jones subfactors]{\boldmath The $L^2$-Betti numbers of the Temperley-Lieb-Jones subfactors}

\begin{definition}
For every extremal finite index subfactor $N \subset M$, we define $\bes_{\sub,n}(N \subset M) := \bns(T \subset S)$ where $T \subset S$ is the SE-inclusion of $N \subset M$.
\end{definition}

The following proposition implies in particular that $\bes_{\sub,n}(N \subset M)$ only depends on the standard invariant of the subfactor $N \subset M$.

\begin{proposition} \label{prop.L2-Betti-subfactors-generalities}
Let $N \subset M$ be an extremal finite index subfactor with tunnel/tower $(M_k)_{k \in \Z}$. Let $\cC_M$ be the category of finite index $M$-bimodules generated by $N \subset M$.
\begin{enumerate}
\item We have $\bes_{\sub,n}(N \subset M) = \bns(\cC_M)$. More generally, whenever $M_k \subset P \subset M_m$ for some $k \leq m$, we have $\bes_{\sub,n}(N \subset M) = \bns(\cC_P)$ where $\cC_P$ is the category of finite index $P$-bimodules generated by $L^2(M_n)$, $n \geq m$.
\item We have $\bes_{\sub,n}(N \subset M) = \bes_{\sub,n}(P \subset Q)$ whenever $M_a \subset P \subset M_k \subset M_m \subset Q \subset M_b$ with $a \leq k < m \leq b$.
\item The $0$'th $L^2$-Betti number is given by the inverse of the \emph{global index} of $N \subset M$ meaning that $\bes_{\sub,0}(N \subset M) = \Bigl(\sum_{\al \in \Irr(\cC_M)} \rd(\al)^2\Bigr)^{-1}$.
\end{enumerate}
\end{proposition}
\begin{proof}
These are immediate consequences of the discussion in Section \ref{sec.L2-Betti-category} and the stability of $L^2$-Betti numbers under Morita equivalence in Proposition \ref{prop.L2-Betti-Morita}. The last point follows from Corollary \ref{cor.zero-betti-category}.
\end{proof}

Recall that a subfactor $N \subset M$ is called Temperley-Lieb-Jones (TLJ) if the relative commutants $M_i' \cap M_j$ in the Jones tower $N \subset M \subset M_1 \subset \cdots$ are as small as possible, i.e.\ generated by the Jones projections $e_k$, $i < k < j$.

A TLJ subfactor $N \subset M$ is said to be of type $A_n$ with $n \in \{2,3,\ldots\} \cup \{\infty\}$ if the principal graph of $N \subset M$ is the Dynkin graph $A_n$. Equivalently, a TLJ subfactor $N \subset M$ is of type $A_n$ with $2 \leq n < \infty$ if $[M:N] = 4 \cos^2(\pi/(n+1))$ and it is of type $A_\infty$ if $[M:N] \geq 4$. The $A_2$ case corresponds to the trivial subfactor $N=M$.

\begin{theorem} \label{thm.L2-Betti-TLJ-Fuss-Catalan}
\begin{enumerate}
\item Let $N \subset M$ be a TLJ subfactor of type $A_n$ with $n \in \{2,3,\ldots\}\cup\{\infty\}$. Then, $\bes_{\sub,k}(N \subset M) = 0$ for all $k \geq 1$ and
$$\bes_{\sub,0}(N \subset M) = \frac{4 \sin^2\bigl(\frac{\pi}{n+1}\bigr)}{n+1} \; ,$$
where the right hand side is interpreted as $0$ when $n=\infty$.

\item Let $N \subset M$ be a Fuss-Catalan subfactor in the sense of \cite{BJ95}, given as the free composition of a TLJ subfactor of type $A_n$ and a TLJ subfactor of type $A_m$ with $n,m \in \{3,4,\ldots\}\cup\{\infty\}$. Then $\bes_{\sub,k}(N \subset M) = 0$ for all $k \neq 1$ and
    $$\bes_{\sub,1}(N \subset M) = 1 - \frac{4 \sin^2\bigl(\frac{\pi}{n+1}\bigr)}{n+1} - \frac{4 \sin^2\bigl(\frac{\pi}{m+1}\bigr)}{m+1} \; .$$
    Note that $\bes_{\sub,1}(N \subset M) > 0$, except in the amenable case $n=m=3$.
\end{enumerate}
\end{theorem}

\begin{proof}
1.\ Let $\cC$ be the tensor category of finite index $M$-bimodules generated by $N \subset M$. Denote by $\cA$ its tube $*$-algebra, with corresponding von Neumann algebra $\cM = \cA\dpr$. In \cite[Section 5.2]{GJ15}, it is proved that for every $i \in \Irr(\cC)$, the von Neumann algebra $p_i \cdot \cM \cdot p_i$ is diffuse abelian and the subalgebra $p_i \cdot \cA \cdot p_i$ is essentially a polynomial algebra. In particular, it follows from \cite[Section 5.2]{GJ15} that every nonzero element of $p_i \cdot \cA \cdot p_i$ defines an injective operator in $p_i \cdot \cM \cdot p_i$. Combining Lemma \ref{lem.polynomials} below and Lemma \ref{lem.general-vanishing}, we conclude that $\bes_k(\cC) = 0$ for all $k \geq 1$. The formula for $\bes_{\sub,0}(N \subset M)$ follows from Proposition \ref{prop.L2-Betti-subfactors-generalities} and the following computation of the global index of a TLJ subfactor of type $A_n$.

For $0 \leq k \leq n-1$, define
$$d_k = \frac{\sin\bigl(\frac{k+1}{n+1} \pi \bigr)}{\sin \bigl(\frac{\pi}{n+1}\bigr)} \; .$$
By \cite[Table 1.4.8]{GHJ89}, the dimensions of the irreducible $M$-bimodules generated by $N \subset M$ are given by $d_k$, $k$ even. A direct computation then gives that the global index of $N \subset M$ equals
$$\frac{n+1}{4 \sin^2\bigl(\frac{\pi}{n+1}\bigr)} \; .$$

2.\ Let $N \subset M$ be a Fuss-Catalan subfactor given as the free composition of TLJ subfactors $N \subset P$ and $P \subset M$. Let $N \subset M \subset M_1 \subset \cdots$ be the Jones tower. Define $\cC_P$ as the category of finite index $P$-bimodules generated by $L^2(M_n)$, $n \geq 0$. By definition, $\cC_P$ is the free product of the categories of $P$-bimodules $\cC_1$ and $\cC_2$ generated by resp.\ $N \subset P$ and $P \subset M$. By Proposition~\ref{prop.L2-Betti-subfactors-generalities}, we have $\bes_{\sub,k}(N \subset M) = \bes_k(\cC_P)$ for all $k \geq 0$. Since $\cC_P$ is the free product of $\cC_1$ and $\cC_2$, the conclusion of the theorem follows from~1 and Proposition~\ref{prop.free-product}.
\end{proof}

In the proof of Theorem \ref{thm.L2-Betti-TLJ-Fuss-Catalan}, we needed the following lemma, using the notation of Definition~\ref{def.L2-Betti-map}.

\begin{lemma} \label{lem.polynomials}
Let $(X,\mu)$ be a standard probability space and $\cD \subset L^\infty(X,\mu)$ a dense $*$-subalgebra with the property that every $a \in \cD \setminus \{0\}$ satisfies $a(x) \neq 0$ for a.e.\ $x \in X$. Then for every $V \in M_{m,n}(\C) \ot \cD$, we have that $\bes(V) = 0$.
\end{lemma}

\begin{proof}
View $V$ as a measurable function $X \recht M_{m,n}(\C)$ with the property that the components $x \mapsto V(x)_{ij}$ belong to $\cD$ for all $i,j$. Denote by $K$ the closure of $\Ker V \cap \cD^{\oplus n}$ inside $L^2(X,\mu)^{\oplus n}$. We have to prove that $K = \Ker V$.

For all subsets $I \subset \{1,\ldots,m\}$ and $J \subset \{1,\ldots,n\}$ with $|I|=|J|$, we denote by $V(x)_{I,J}$ the $I \times J$ minor of the matrix $V(x)$, i.e.\ the determinant of the matrix given by the $I$-rows and $J$-columns of $V(x)$. Define $k \in \{0,\ldots,n\}$ as the largest integer for which there exist such subsets $I$ and $J$ with $|I|=|J|=k$ and with $x \mapsto V(x)_{I,J}$ being nonzero on a non negligible set of $x \in X$. Since $x \mapsto V(x)_{I,J}$ belongs to $\cD$, we then get that $V(x)_{I,J} \neq 0$ for a.e.\ $x \in X$, while $V(x)_{I',J'} = 0$ for a.e.\ $x \in X$ and all subsets $I',J'$ with $|I'|=|J'| > k$. After removing from $X$ a set of measure zero and after reordering the indices, we may assume that with $I=J= \{1,\ldots,k\}$, we have $V(x)_{I,J} \neq 0$ for all $x \in X$, and $V(x)_{I',J'} = 0$ for all $x \in X$ and all subsets $I',J'$ with $|I'|=|J'| > k$.

We define for all $r = k+1,\ldots,n$, the elements $\xi_r \in \cD^{\oplus n}$ given by
$$\xi_r(x)_j = \begin{cases} (-1)^j V(x)_{I,(J \setminus \{j\}) \cup \{r\}} &\;\; \text{if}\;\; 1 \leq j \leq k \; ,\\
(-1)^{k+1} V(x)_{I,J} &\;\; \text{if}\;\; j = r \; , \\
0 &\;\; \text{if}\;\; j \in \{k+1,\ldots,n\} \setminus \{r\} \; .\end{cases}$$
For every $x \in X$, the matrix $V(x)$ has rank $r$ and the vectors $\xi_r(x) \in \C^n$, $r = k+1,\ldots,n$, form a basis for $\Ker V(x)$.

Fix $\eta \in \Ker V$. Then, for a.e.\ $x \in X$, we have that
\begin{equation}\label{eq.my-eta}
\eta(x) = \sum_{r=k+1}^n V(x)_{I,J}^{-1} \eta(x)_r \; \xi_r(x) \; .
\end{equation}
Fix $\eps > 0$. Take a measurable subset $X_0 \subset X$ such that $\mu(X \setminus X_0) < \eps$ and such that both $x \mapsto V(x)_{I,J}^{-1}$ and $x \mapsto \eta(x)_r$ are bounded on $X_0$. Denote by $1_{X_0}$ the projection in $L^\infty(X,\mu)$ that corresponds to $X_0$. Then \eqref{eq.my-eta} implies that $\eta \cdot 1_{X_0}$ belongs to the linear span of $(\Ker V \cap \cD^{\oplus n}) \cdot L^\infty(X)$. Thus, $\eta \cdot 1_{X_0} \in K$. Since $\eps > 0$ is arbitrary, we conclude that $\eta \in K$. So we have proved that $K = \Ker V$.
\end{proof}

\subsection{Homology with trivial coefficients}

The following result generalizes the statement that homology of finite groups with trivial coefficients vanishes.

\begin{proposition}\label{prop:FiniteIndexVanishes}
Let $T\subset S$ be a finite index inclusion. Then $H_{n}(T\subset S,L^2(S))=0$ for all $n \geq 1$, while $H_0(T \subset S,L^2(S)) = \C$.
\end{proposition}
\begin{proof}
By Remark \ref{rem.homology}, the differential complex in Definition \ref{def.homology-quasiregular} computing $H_{n}(T\subset S,L^2(S))$ consists of cyclic tensor products, which
are exactly the higher relative commutants $T'\cap S_{k}$ associated to the Jones tower $T\subset S\subset S_{1}\subset\cdots$. The differential of this complex is then precisely the one considered in \cite[Section 6]{Jo98} and it follows from \cite{Jo98} that the complex is acyclic.
\end{proof}

\begin{remark}
Let $T \subset S$ be an irreducible quasi-regular inclusion of II$_1$ factors. Write $\cS = \QN_S(T)$. By Remark \ref{rem.homology}.2, the homology $H_n(T \subset S,L^2(S))$ can be computed by the bar complex $(C_n)_{n \geq 0}$ given by the $(n+1)$-fold cyclic tensor products $C_n = \cS^{n+1}/T$ of $\cS$ relative to $T$. Defining the shift $\tau : C_n \recht C_n : \tau(x_0 \ot \cdots \ot x_n) = x_n \ot x_0 \ot \cdots x_{n-1}$, one can also define the \emph{cyclic chain complex} $(C_n^\lambda)_{n \geq 0}$ given by $C_n^\lambda = C_n / \{ \xi - (-1)^n \tau(\xi) \mid \xi \in C_n \}$. The \emph{cyclic homology} of $T \subset S$ can then be defined as the homology of $(C_n^\lambda)_{n \geq 0}$. Similarly, one defines the \emph{cyclic cohomology} of $T \subset S$. Again, when $T \subset S$ has finite index, cyclic homology trivializes (see \cite[Corollary 6.3]{Jo98}), but for other quasi-regular inclusions like the SE-inclusion of a subfactor of infinite depth, one obtains a potentially interesting cyclic (co)homology theory.

Using the methods of Theorem \ref{thm.tor-ext}, the cyclic homology of $T \subset S$ can be identified with a cyclic homology theory for the tube $*$-algebra $\cA$ associated with $T \subset S$ and the tensor category $\cC$ generated by the finite index $T$-subbimodules of $L^2(S)$. In particular, one can define in this way a cyclic homology theory for rigid C$^*$-tensor categories.
\end{remark}

Let $\mathscr{P}$ be a planar algebra. The tube algebra associated to $\mathscr{P}$ has a canonical trivial module, which corresponds under Morita equivalence to the trivial module $\mathcal{E}^r$ discussed in Lemma~\ref{lem.reg-triv-A} and Remark~\ref{rem.reg-triv-as-cp-maps-and-states} (as a graded vector space, this module is the planar algebra $\mathscr{P}$ itself, with the action defined by gluing elements of $\mathscr{P}$ into the input disk of a tube algebra element).  The homology with coefficients in this module is then computed by the differential complex described as follows. The space $\mathscr{C}_{k}$ is the linear span of diagrams drawn on the sphere with points $r_{k+1},\dots,r_{1}=\infty$ removed:

$$
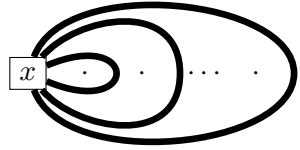
\begin{tikzpicture}
[manyStrings/.style={line width=2.5pt}]
\node [rectangle,draw] (x) at (1,1) {$x$};
\node [circle] (p1) at (1.75,1) {$\cdot$};
\node [circle] (p2) at (2.5,1) {$\cdot$};
\node [circle] (dots) at (3.35,1) {$\cdots$};
\node [circle] (pkPlus1) at (4,1) {$\cdot$};
\draw [manyStrings] (x) .. controls +(23:1.6) and +(157:-1.6) ..  (x);
\draw [manyStrings] (x) .. controls +(50:1) and +(90:1) .. (3.0,1) .. controls +(-90:1) and +(-50:1) .. (x);
\draw [manyStrings] (x) .. controls +(70:1.4) and +(90:1) .. (4.5,1) .. controls +(-90:1) and +(-70:1.4) .. (x);
\end{tikzpicture}$$

The differential $\partial_{k}:\mathscr{C}_{k}\to\mathscr{C}_{k-1}$ is again given by $\sum_{j=0}^{k}(-1)^{j}d_{j}$ where $d_{j}$ sends a diagram drawn on $S^{2}\setminus\{r_{1},\dots,r_{k+1}\}$ to the diagram drawn on $S^{2}\setminus\{r_{1},\dots,r_{j},r_{j+2},\dots,r_{k+1}\}$. In particular, we have that $d_{0}$ is given by

$$d_0 \left(\begin{array}{c}
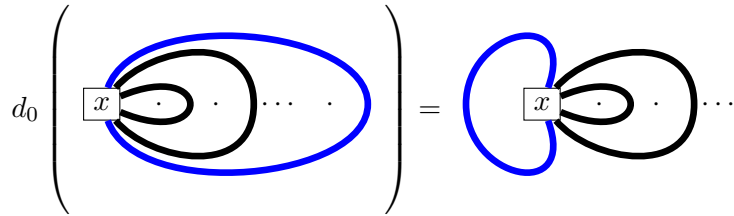
\begin{tikzpicture}
[manyStrings/.style={line width=2.5pt}]
\node [rectangle,draw] (x) at (1,1) {$x$};
\node [circle] (p1) at (1.75,1) {$\cdot$};
\node [circle] (p2) at (2.5,1) {$\cdot$};
\node [circle] (dots) at (3.35,1) {$\cdots$};
\node [circle] (pkPlus1) at (4,1) {$\cdot$};
\draw [manyStrings] (x) .. controls +(23:1.6) and +(157:-1.6) ..  (x);
\draw [manyStrings] (x) .. controls +(50:1) and +(90:1) .. (3.0,1) .. controls +(-90:1) and +(-50:1) .. (x);
\draw [manyStrings,color=blue] (x) .. controls +(70:1.4) and +(90:1) .. (4.5,1) .. controls +(-90:1) and +(-70:1.4) .. (x);
\end{tikzpicture}
\end{array}\right)
=
\begin{array}{c}\begin{tikzpicture}
[manyStrings/.style={line width=2.5pt}]
\node [rectangle,draw] (x) at (1,1) {$x$};
\node [circle] (p1) at (1.75,1) {$\cdot$};
\node [circle] (p2) at (2.5,1) {$\cdot$};
\node [circle] (dots) at (3.35,1) {$\cdots$};
\draw [manyStrings] (x) .. controls +(23:1.6) and +(157:-1.6) ..  (x);
\draw [manyStrings] (x) .. controls +(50:1) and +(90:1) .. (3.0,1) .. controls +(-90:1) and +(-50:1) .. (x);
\draw [manyStrings,color=blue] (x) .. controls +(70:1.4) and +(90:1) .. (0,1) .. controls +(-90:1) and +(-70:1.4) .. (x);
\end{tikzpicture}
\end{array}
$$
where we have colored the strings of $x$ that pass between the point $r_2$ and the point at infinity $r_{1}$ in blue for emphasis.

In the case of the TLJ planar algebra, the space $\mathscr{C}_{k}$ is linearly spanned by all possible topological arrangements of non-intersecting circles surrounding $k$ points in the plane (which is identified with the sphere with a point at infinity removed).
Furthermore, the interiors of the circles are shaded in an alternating fashion, so that each circle lies at the boundary between a shaded an unshaded region.  The shading of the entire picture is completely determined by whether the region near the point at infinity is shaded or not, and we refer to this as the shading of the picture.
Thus, for example
\[
\mathscr{C}_{0}=\mathbb{C},\qquad\mathscr{C}_{1}=\operatorname{span}\{\sigma_{k}:k\geq0\},\qquad\mathscr{C}_{2}=\operatorname{span}\{\sigma_{a,b}^{c}:a,b,c\geq0\}
\]
where

\begin{equation}\label{eq.sigmas}
\sigma_k=
\begin{array}{c}
\begin{tikzpicture}
[manyStrings/.style={line width=2.5pt}]
\node[circle,draw,line width=2.5pt]  (x) at (0,0) {$\cdot$};
\node[token](p0) at (x.east) {\scriptsize $k$};
\end{tikzpicture}
\end{array}
,\qquad \sigma_{a,b}^c =
\begin{array}{c}
\begin{tikzpicture}
[manyStrings/.style={line width=2.5pt}]
\node[circle,draw,line width=2.5pt]  (c) at (0,0) {$\qquad\qquad\ $};
\node [circle,draw, line width=2.5pt] (a) at (-0.5,0) {$\cdot$};
\node [circle,draw, line width=2.5pt] (b) at (0.5,0) {$\cdot$};
\node[token](p0) at (c.east) {\scriptsize $c$};
\node[token](p1) at (a.east) {\scriptsize $a$};
\node[token](p2) at (b.west) {\scriptsize $b$};
\end{tikzpicture}
\end{array}
\end{equation}
(the letters indicate numbers of parallel strings).  Here we are abusing notation and are using the same symbol and picture not specifying the shading at infinity.  These elements are linearly independent in the case that the parameter $\delta$ is generic (i.e., $\delta\geq2$). For $\delta<2$ there are relations between these elements. In particular, in that case $\mathscr{C}_{1}$ is the linear span of $\sigma_{k}$ for $0\leq k\leq d$ for some fixed $d$ (depending on $\delta<2$).

\begin{proposition}\label{prop.low-homology-TLJ-trivial-coeff}
We have $H_{0}(\TLJ(\delta))=\mathbb{C}$, while $H_{1}(\TLJ(\delta)) = H_{2}(\TLJ(\delta)) =0$.
\end{proposition}

Before proving Proposition \ref{prop.low-homology-TLJ-trivial-coeff}, it is worth noting that if $\delta\geq2$, the fusion algebra associated to the TLJ planar algebra is isomorphic to the algebra of single-variable polynomials $R=\mathbb{C}[t]$ with the augmentation given by $\counit : p\mapsto p(\delta)$. One can easily check that the map $q\mapsto q'(\delta)$ is a nontrivial linear function on the space of Hochschild $1$-cycles for $(R,\alpha)$ and descends to a nonzero functional on $\HH_{1}(R,\mathbb{C})$. This homology group is therefore nonzero (it is in fact equal to $\mathbb{C}$).

So, the TLJ planar algebra provides an example where the homology of the tube algebra is different from the homology of the associated fusion algebra.

\begin{proof}
Using the notation \eqref{eq.sigmas}, we have
\[
\partial_{1}\sigma_{k}=0 \quad\text{and}\quad\partial_{2}\sigma_{a,b}^{c}=\delta^{a}\sigma_{b+c}-\delta^{b}\sigma_{a+c}+\delta^{c}\sigma_{a+b},
\] where the shading of all of the terms on the right hand side of the equation is the same as that of the element on the left, except that the shading of the last term is reversed if $c$ is odd.
So, $H_{0}(\TLJ(\delta))=\mathbb{C}$. We also get that $\partial_{2}\sigma_{a,1}^{0}=\delta^{a}\sigma_{1}-\delta\sigma_{a}+\sigma_{a+1}$ so that $\sigma_{a+1}$  (with either shading) is homologous to a linear combination of $\sigma_{a}$ and $\sigma_{1}$. Applying this inductively shows that any $\sigma_{k}$ (with either shading) is homologous to an element of the linear span of $\sigma_{1}$ and $\sigma_{0}$ (both with the opposite shading). On the other hand, $\partial_{2}\sigma_{0,0}^{1}=\sigma_{1}-\sigma_{1}+\delta\sigma_{0}$ which shows that $\sigma_{0}$ (with either shading) is homologous to zero. Finally, $\partial_{2}\sigma_{0,1}^{0}=\sigma_{1}-\delta\sigma_{0}+\sigma_{1}$ which shows that $2\sigma_{1}$ (with either shading) is homologous to $\delta\sigma_{0}$ and thus to zero. So, we have proved that $H_{1}(\TLJ(\delta))=0$.

We further compute

$$
\partial_3\left(
\begin{array}{c}
\begin{tikzpicture}
[manyStrings/.style={line width=2.5pt}]
\node[circle,draw,line width=2.5pt]  (a) at (0,0) {$\cdot$};
\node[token](p0) at (a.east) {\scriptsize $a$};
\node[circle,draw,line width=2.5pt]  (b) at (1,0) {$\cdot$};
\node[token](p0) at (b.east) {\scriptsize $b$};
\node[circle,draw,line width=2.5pt]  (c) at (2,0) {$\cdot$};
\node[token](p0) at (c.east) {\scriptsize $c$};
\end{tikzpicture}
\end{array}
\right) = \delta^a \sigma_{b,c}^0 - \delta^b \sigma_{a,c}^0 + \delta^c \sigma_{a,b}^0 -\sigma_{a,b}^c
$$
where the shading of all terms on the right is the same as that of the term on the left, except that the shading of $\sigma_{a,b}^c$ is reversed when $c$ is odd.

We will use the notation $x\sim y$ to indicate that $x-y\in\operatorname{image}\partial_{3}$.
Thus:
\begin{equation}
\sigma_{a,b}^{c}\sim\delta^{a}\sigma_{b,c}^{0}-\delta^{b}\sigma_{a,c}^{0}+\delta^{c}\sigma_{a,b}^{0}.\label{eq:removeouter}
\end{equation}
(with same or reversed shading depending on the parity of $c$).

Next, consider

\begin{eqnarray*}
\partial_3\left(
\begin{array}{c}
\begin{tikzpicture}
[manyStrings/.style={line width=2.5pt}]
\node[circle,line width=2.5pt]  (a) at (0,0) {$\cdot$};
\node[circle,draw,line width=2.5pt]  (b) at (1,0) {$\cdot$};
\node[token](p0) at (b.west) {\scriptsize $b$};
\node[circle,draw,line width=2.5pt]  (c) at (2,0) {$\cdot$};
\node[token](p0) at (c.east) {\scriptsize $c$};
\node[circle,draw,line width=2.5pt] (l) at (0.5,0) {$\qquad\quad\ \ \ $};
\node[token](l0) at (l.west) {\scriptsize $l$};
\end{tikzpicture}
\end{array}
\right) & = & \sigma_{b+l,c}^0 - \delta^b \sigma_{l,c}^0 + \delta^c \sigma_{0,b}^l - \sigma_{0,b}^{l+c}\\
&\sim& \sigma_{b+l,c}^0 - \delta^b \sigma_{l,c}^0 + \delta^c \left[ \sigma_{b,l}^0 - \delta^b \sigma_{0,l}^0
+\delta^l \sigma_{0,b}^0\right] \\ && - \left[ \sigma_{b,l+c}^0 - \delta^b \sigma_{0,l+c} + \delta^{l+c} \sigma_{0,b}^0\right].
\end{eqnarray*}
where the shading of all of the terms is the same as that of the term on the left hand side, except that the shading of the first occurrence of $\sigma_{0,b}^0$ is reversed according to the parity of $l$ and the shading of its second occurrence is reversed according to the parity of $l+c$.

Thus
\[
\sigma_{b+l,c}^{0}-\delta^{b}\sigma_{l,c}^{0}\sim\sigma_{b,l+c}^{0}-\delta^{c}\sigma_{b,l}^{0}+\operatorname{span}\{\sigma_{a,0}^{0},\sigma_{0,b}^{0}:a,b\geq0\}.
\] (with all possible shadings of the right hand side).
Taking $l=1$ we get:
\begin{equation}
\sigma_{b+1,c}^{0}\sim\delta^{b}\sigma_{1,c}^{0}+\sigma_{b,c+1}^{0}-\delta^{c}\sigma_{b,1}+\operatorname{span}\{\sigma_{a,0}^{0},\sigma_{0,b}^{0}:a,b\geq0\}.\label{eq:sigmabplus1}
\end{equation}
Applying this recursively shows that
\begin{equation}
\sigma_{a,b}^{c}\sim\operatorname{span}\{\sigma_{1,a}^{0},\sigma_{b,0}^{0},\sigma_{0,c}^{0}:a,b,c\geq0\}\label{eq:firstSpan}
\end{equation}
(with all possible shadings).

Setting $c=0$ in (\ref{eq:sigmabplus1}) gives
\[
\sigma_{b+1,0}^{0}\sim\delta^{b}\sigma_{1,c}^{0}+\operatorname{span}\{\sigma_{a,0}^{0},\sigma_{0,b}^{0}:a,b\geq0\}.
\]
 Thus
\begin{equation}
\sigma_{1,c}^{0}\sim\sigma_{b+1,0}^{0}+\operatorname{span}\{\sigma_{a,0}^{0},\sigma_{0,b}^{0}:a,b\geq0\}.\label{eq:sigma1c}
\end{equation}

Using (\ref{eq:removeouter}) with $a=c=0$ we get that
\begin{equation}
\sigma_{0,b}^{0}\sim\sigma_{b,0}^{0}-\delta^{b}\sigma_{0,0}^{0}+\sigma_{0,b}^{0}\label{eq:sigma0b}
\end{equation}
so that $\sigma_{b,0}^{0}\sim\delta^{b}\sigma_{0,0}^{0}$. Using (\ref{eq:sigma0b})
and (\ref{eq:sigma1c}) we deduce that
\[
\sigma_{1,c}^{0}\sim\operatorname{span}\{\sigma_{a,0}^{0},\sigma_{0,b}^{0}:a,b\geq0\},
\]
which together with (\ref{eq:sigma0b}) implies that
\[
\operatorname{span}\{\sigma_{1,a}^{0},\sigma_{b,0}^{0},\sigma_{0,c}^{0}:a,b,c\geq0\}\sim\operatorname{span}\{\sigma_{a,0}^{0},\sigma_{0,b}^{0}:a,b\geq0\}\sim\operatorname{span}\{\sigma_{0,a}^{0}:a\geq0\}.
\]
Combining this with (\ref{eq:firstSpan}) we obtain that any $\sigma_{a,b}^{c}$
is equivalent modulo the image of $\partial_{3}$ to an element of
$\operatorname{span}\{\sigma_{0,a}^{0}:a\geq0\}$ (with all possible shadings).

Assume now that that $z\in\ker\partial_{2}$. Then we may assume that
$z$ (up to the image of $\partial_{3}$) is of the form $\sum\alpha_{a}\sigma_{0,a}^{0}$ (with various shadings).

If $\delta\geq2$ then $\{\sigma_{0,a}^{0}:a\geq0\}$ and $\{\sigma_{a}:a\geq0\}$
are both linearly independent sets (with either shading). Using
\begin{eqnarray*}
\partial_{2}\sigma_{0,a}^{0} & = & 2\sigma_{a}-\delta^{a}\sigma_{0}
\end{eqnarray*} (with same shadings on both sides)
we deduce
\[
2\sum\alpha_{a}\sigma_{a}-(\sum\alpha_{a}\delta^{a})\sigma_{0}=\partial_{2}z=0
\]
which implies that $\alpha_{a}=0$ for all $a$ and so $z\sim0$. We have proved that $H_2(\TLJ(\delta))=0$.

If $\delta<2$ we already know that $H_{2}(\TLJ(\delta))$ vanishes because $\TLJ(\delta)$ is finite-depth; however, there is a short independent argument. Indeed, there exists an integer $k$ so that $\{\sigma_{0,a}^{0}:0\leq a\leq k\}$ and $\{\sigma_{a}:0\leq a\leq k\}$ are both linearly independent sets and moreover $\operatorname{span}\{\sigma_{0,a}^{0}:0\leq a\leq k\}=\operatorname{span}\{\sigma_{0,a}^{0}:a\geq0\}$. Thus we may assume that $z=\sum_{a=0}^{k}\alpha_{a}\sigma_{0,a}^{0}$ and using the formula for $\partial_{2}z$ we conclude again that
$\alpha_{a}=0$ for $0\leq a\leq k$ and that $z\sim0$.
\end{proof}

We do not know if $H_{n}(\TLJ(\delta))=0$ for all values of $\delta$ but suspect that this is the case. In general, it would be very interesting to construct a resolution (of finite length?) for $\TLJ(\delta)$ that allows to prove at the same time that $H_n(\TLJ(\delta)) = 0$ for all $n \geq 1$ and $\bns(\TLJ(\delta))=0$ for all $n \geq 0$.

\subsection{One-cohomology characterizations of property (T), the Haagerup property and amenability}

We recall the following definitions from \cite{Po86,Po01}.

\begin{definition}\label{def.prop-T-Haagerup}
Let $S$ be a II$_1$ factor and $T \subset S$ a quasi-regular irreducible subfactor.
\begin{enumerate}
\item \cite[Definition 4.1.3]{Po86} $S$ has property~(T) relative to $T$ if the following holds: whenever $\vphi_i : S \recht S$ is a net of normal $T$-bimodular completely positive maps satisfying $\lim_i \|\vphi_i(x) - x\|_2 = 0$ for every $x \in S$, then $\lim_i \bigl(\sup_{x, \|x\|\leq 1} \|\vphi_i(x) - x\|_2 \bigr) = 0$.
\item \cite[Definition 2.1]{Po01} $S$ has the Haagerup property relative to $T$ if there exists a net of normal $T$-bimodular completely positive maps $\vphi_i : S \recht S$ such that $\lim_i \|\vphi_i(x) - x\|_2 = 0$ for every $x \in S$ and such that for every $i$, the map $\vphi_i : S \recht S$ belongs to the compact ideal space $\cJ(\langle S,e_T \rangle)$ (i.e.\ the norm closed linear span of all finite projections in the semifinite factor $\langle S,e_T \rangle$, see \cite[Section 1.3.3]{Po01}).
\end{enumerate}
\end{definition}

Whenever $T \subset S$ is a quasi-regular irreducible subfactor, we denote by $\cC$ the tensor category of finite index $T$-bimodules generated by $L^2(S)$. As before, for every subset $\cF \subset \Irr(\cC)$, we denote by $e_\cF$ the orthogonal projection of $L^2(S)$ onto the closed linear span of all $T$-subbimodules of $L^2(S)$ that are isomorphic with a $T$-bimodule contained in $\cF$.

\begin{definition}\label{def.bounded-proper-cocycle}
Let $S$ be a II$_1$ factor and $T \subset S$ a quasi-regular irreducible subfactor. Denote $\cS = \QN_S(T)$. A \emph{$1$-cocycle} for $T \subset S$ is a $T$-bimodular derivation $c : \cS \recht \cH$ from $\cS$ to a Hilbert $S$-bimodule $\cH$. Such a $1$-cocycle is said to be
\begin{enumerate}
\item \emph{inner} if there exists a $T$-central vector $\xi \in \cH$ such that $c(x) = x \xi - \xi x$ for all $x \in \cS$~;
\item \emph{approximately inner} if there exists a net of $T$-central vectors $\xi_i \in \cH$ such that $\lim_i \|c(x) - (x \xi_i - \xi_i x)\| = 0$ for all $x \in \cS$~;
\item \emph{bounded} if $c$ extends to a bounded operator from $L^2(S)$ to $\cH$~;
\item \emph{proper} if for every $\kappa > 0$, there exists a finite subset $\cF \subset \Irr(\cC)$ such that $\|c(x)\| \geq \kappa \|x\|_2$ for all $x \in (1-e_\cF)(\cS)$.
\end{enumerate}
\end{definition}

The following is the main result of this section and provides a one-cohomology characterization of property~(T), the Haagerup property and amenability. These characterizations are well known in the group case~: the first is analogous to the Delorme-Guichardet theorem (see e.g.\ \cite[Theorem 2.12.4]{BHV08})~; for the second one, see \cite[Theorem 2.1.1]{CCJJV01}~; for the last one, see \cite[Chapter III, Corollary 2.4]{Gu80}.

\begin{theorem}\label{thm.one-cohom-characterizations}
Let $S$ be a II$_1$ factor with separable predual and $T \subset S$ a quasi-regular irreducible subfactor. Denote $\cS = \QN_S(T)$.
\begin{enumerate}
\item $S$ has property~(T) relative to $T$ if and only if for every Hilbert $S$-bimodule $\cH$, every $1$-cocycle $c : \cS \recht \cH$ is inner.
\item $S$ has the Haagerup property relative to $T$ if and only if there exists a proper $1$-cocycle $c : \cS \recht \cH$ into some Hilbert $S$-bimodule $\cH$.
\item $S$ is amenable relative to $T$ (see Definition \ref{def.amenable-quasi-reg-inclusion}) and $[S:T] = \infty$ if and only if there exists an approximately inner, but non inner $1$-cocycle $c : \cS \recht L^2(S) \ovt_T L^2(S)$.
\end{enumerate}
\end{theorem}

The following is an immediate consequence of Theorem \ref{thm.one-cohom-characterizations}.1.

\begin{corollary}\label{cor.vanish-first-L2-Betti-T}
Let $S$ be a II$_1$ factor and $T \subset S$ a unimodular quasi-regular irreducible subfactor. If $S$ has property~(T) relative to $T$, then $\bes_1(T \subset S) = 0$.
\end{corollary}

Before proving Theorem \ref{thm.one-cohom-characterizations}, we need a few technical lemmas.

\begin{lemma}\label{lem.bounded-inner}
Let $S$ be a II$_1$ factor and $T \subset S$ a quasi-regular irreducible subfactor. Denote $\cS = \QN_S(T)$. A $1$-cocycle $c: \cS \recht \cH$ is bounded if and only if it is inner.
\end{lemma}
\begin{proof}
When $\xi \in \cH_T$, the normal functional $S \recht \C : x \mapsto \langle x \xi,\xi\rangle$ is $T$-central and hence a multiple of the trace $\tau$. Therefore, $\|x\xi\| = \|x\|_2 \, \|\xi\| = \|\xi x \|$ for all $\xi \in \cH_T$ and $x \in S$. It follows in particular that every inner $1$-cocycle is bounded.

Conversely, if $c : \cS \recht \cH$ is a bounded $1$-cocycle, which we extend to $c : L^2(S) \recht \cH$, we define $\xi$ as the center of the closed convex hull $K$ of $\{u^* c(u) \mid u \in \cU(S)\}$. Since $v^* K v = K$ for all $v \in \cU(T)$, it follows that $v^* \xi v = \xi$ for all $v \in \cU(T)$, so that $\xi$ is $T$-central. When $v \in \cU(S)$, the map $\eta \mapsto v^* \eta v + v^* c(v)$ is an isometry that globally preserves $K$. Therefore $v^* \xi v + v^* c(v) = \xi$ for all $v \in \cU(S)$, so that $c(x) = x \xi - \xi x$ for all $x \in S$.
\end{proof}

\begin{lemma}\label{lem.uniformity}
Let $S$ be a II$_1$ factor and $T \subset S$ a quasi-regular irreducible subfactor. Let $\vphi_i : S \recht S$ be a net of normal $T$-bimodular completely positive maps. If $\vphi_i \recht \id$ in $\|\,\cdot\,\|_2$ uniformly on $\{x \in S \mid \|x\| \leq 1\}$, then $\vphi_i \recht \id$ in $\|\,\cdot\,\|_2$ uniformly on $\{x \in S \mid \|x\|_2 \leq 1\}$.
\end{lemma}
\begin{proof}
It suffices to prove the following statement: if $\eps > 0$ and $\vphi : S \recht S$ is a normal unital $T$-bimodular completely positive map satisfying $\|\vphi(u)-u\|_2 \leq \eps^2/8$ for all $u \in \cU(S)$, then $\|\vphi(x)-x\|_2 \leq \eps \|x\|_2$ for all $x \in S$. To prove this statement, construct the Hilbert $S$-bimodule $\cH$ with $T$-central unit vector $\xi \in \cH_T$ satisfying $\langle x \xi y , \xi \rangle = \tau(x \vphi(y))$ for all $x,y \in S$. By our assumption, $\|u^* \xi u - \xi\| \leq \eps/2$ for all $u \in \cU(S)$. Averaging, it follows that $\|P_S(\xi) - \xi\| \leq \eps/2$, where $P_S$ denotes the orthogonal projection of $\cH$ onto the $S$-central vectors in $\cH$.

As in the proof of Lemma \ref{lem.bounded-inner}, $\|x\eta\| = \|x\|_2 \, \|\eta\| = \|\eta x \|$ for all $\eta \in \cH_T$ and $x \in S$. Therefore,
$$\|x \xi - \xi x\| = \|x (\xi - P_S(\xi)) - (\xi - P_S(\xi)) x\| \leq \eps \|x\|_2$$
for all $x \in S$. But then we get, for all $x,y \in S$ that
$$|\tau(y^*(\vphi(x)-x))| = |\langle \xi x - x \xi , y \xi \rangle| \leq \eps \, \|x\|_2 \, \|y\|_2 \; .$$
Therefore, $\|\vphi(x)-x\|_2 \leq \eps \|x\|_2$ for all $x \in S$.
\end{proof}

\begin{lemma}\label{lem.compact-ideal-space}
Let $S$ be a II$_1$ factor and $T \subset S$ a quasi-regular irreducible subfactor. Let $\vphi : S \recht S$ be a normal completely positive $T$-bimodular map. Then $\vphi$ belongs to the compact ideal space $\cJ(\langle S,e_T \rangle)$ if and only if for every $\eps > 0$, there exists a finite subset $\cF \subset \Irr(\cC)$ such that $\|\vphi(x)\|_2 < \eps \|x\|_2$ for all $x \in (1-e_\cF)L^2(S)$.
\end{lemma}
\begin{proof}
We denote by $R_\vphi$ the bounded operator on $L^2(S)$ defined by $R_\vphi(x) = \vphi(x)$ for all $x \in S$. Note that $R_\vphi \in T' \cap \langle S,e_T \rangle$. First assume that $R_\vphi \in \cJ(\langle S,e_T \rangle)$ and choose $\eps > 0$. Define the spectral projection $q_\eps := 1_{[\eps,+\infty)}(|R_\vphi|)$. Denoting by $\Tr^r$ the canonical semifinite trace on $\langle S,e_T \rangle$, we have $\Tr^r(q_\eps) < \infty$. Since $q_\eps \in T' \cap \langle S,e_T \rangle$, it follows that the range of $q_\eps$ is a $T$-subbimodule of $L^2(S)$ of finite right dimension. So we can take a finite subset $\cF \subset \Irr(\cC)$ such that $q_\eps \leq e_\cF$. Whenever $x \in (1-e_\cF)L^2(S)$, we get $q_\eps(x) = 0$ and thus, $\|\vphi(x)\|_2 < \eps \|x\|_2$.

To prove the converse, assume that $\eps > 0$ and that $\cF \subset \Irr(\cC)$ is a finite subset such that $\|\vphi(x)\|_2 < \eps \|x\|_2$ for all $x \in (1-e_\cF)L^2(S)$. Then $\|R_\vphi - R_\vphi e_\cF\| \leq \eps$. By Lemma \ref{lem.well-behaved}.5, we have that $e_\cF \in \cJ(\langle S,e_T \rangle)$, so that $R_\vphi$ lies at distance less than $\eps$ from $\cJ(\langle S,e_T \rangle)$.
\end{proof}

\begin{lemma}\label{lem.construct-cocycle}
Let $S$ be a II$_1$ factor with separable predual and $T \subset S$ a quasi-regular irreducible subfactor. Denote $\cS = \QN_S(T)$. Let $\vphi_n : S \recht S$ be a sequence of unital normal $T$-bimodular completely positive maps satisfying $\lim_n \|\vphi_n(x) - x\|_2 = 0$ for every $x \in S$. Construct the associated Hilbert $S$-bimodules $\cH_n$ with $T$-central unit vectors $\xi_n \in (\cH_n)_T$ satisfying $\langle x \xi_n y,\xi_n \rangle = \tau(x \vphi_n(y))$ for all $x,y \in S$.
\begin{enumerate}
\item After passage to a subsequence, $c : \cS \recht \cH = \bigoplus_n \cH_n : x \mapsto \oplus_n (x \xi_n - \xi_n x)$ is a well defined $1$-cocycle.
\item If $\vphi_n$ does not converge to the identity uniformly on the unit ball of $S$, the choice in 1 can be made so that $c$ is not inner.
\item If each $\vphi_n$ belongs to the compact ideal space $\cJ(\langle S,e_T \rangle)$, the $1$-cocycle $c$ is proper.
\end{enumerate}
\end{lemma}
\begin{proof}
Denote by $\cC$ the tensor category of finite index $T$-bimodules generated by $L^2(S)$. Write $\Irr(\cC) = \bigcup_n \cF_n$ where $\cF_n \subset \Irr(\cC)$ is an increasing sequence of finite subsets. Define $\cS_n = e_{\cF_n}(\cS)$ and note that $\cS = \bigcup_n \cS_n$. After passage to a subsequence, we may assume that $\|x \xi_n - \xi_n x \|_2 \leq 2^{-n} \|x\|_2$ for all $x \in \cS_n$ and all $n \geq 0$. So, for every $x \in \cS$, the sequence $(\|x \xi_n - \xi_n x \|)_n$ is square summable and $c(x)$ is a well defined vector in $\cH$.

To prove 2, it suffices to show that if $c$ is inner, then $\vphi_n$ converges to the identity uniformly on the unit ball of $S$. So, assume that $c(x) = x \eta - \eta x$ for all $x \in \cS$, where $\eta = \oplus_n \eta_n$ is a $T$-central vector. It follows that $x \xi_n - \xi_n x = x \eta_n - \eta_n x$ for all $x \in \cS$ and all $n \geq 0$. For all $x,y \in \cS$, we get that
$$\tau(y^*(\vphi_n(x) - x)) = \langle \xi_n x - x \xi_n,y\rangle = \langle \eta_n x - x \eta_n , y \rangle$$
and we conclude that $\|\vphi_n(x)-x\|_2 \leq 2 \|\eta_n\| \, \|x\|_2$. Since $\lim_n \|\eta_n\| = 0$, it follows that $\vphi_n$ converges to the identity uniformly on the unit ball of $S$.

Finally assume that all $\vphi_n$ belong to $\cJ(\langle S,e_T \rangle)$. By Lemma \ref{lem.compact-ideal-space}, we can take finite subsets $\cF_n \subset \Irr(\cC)$ such that $\|\vphi_n(x)\|_2 \leq \|x\|_2 /2$ for all $x \in (1-e_{\cF_n})(S)$. Since
$$\|x\xi_n - \xi_n x\|^2 = 2(\|x\|_2^2 - \real \tau(x^* \vphi_n(x))) \; ,$$
we get that $\|x\xi_n - \xi_n x\|^2 \geq \|x\|_2^2$ for all $x \in (1-e_{\cF_n})(S)$. Defining the finite sets $\cG_n = \bigcup_{k=1}^n \cF_k$, it follows that $\|c(x)\|^2 \geq n \, \|x\|_2^2$ for all $x \in (1-e_{\cG_n})(\cS)$. So, $c$ is proper.
\end{proof}

We are now ready to prove Theorem \ref{thm.one-cohom-characterizations}.

\begin{proof}[Proof of Theorem \ref{thm.one-cohom-characterizations}]
1.\ Assume that $S$ has property (T) relative to $T$ and let $c : \cS \recht \cH$ be a $1$-cocycle. We have to prove that $c$ is inner. Replacing $\cH$ by $\cH \oplus \overline{\cH}$ and $c$ by $c \oplus \overline{c}$, we may assume that $c$ is real~: there exists an anti-unitary involution $J : \cH \recht \cH$ satisfying $J(x \xi y) = y^* J(\xi) x^*$ for all $x,y \in S$, $\xi \in \cH$ and $c(x^*) = J(c(x))$ for all $x \in \cS$. For the following reason, $c$ is automatically a closable map from $\cS \subset L^2(S)$ to $\cH$. When $\cH_\al \subset \cH$ is an irreducible finite index $T$-subbimodule and $P_\al : \cH \recht \cH_\al$ is the orthogonal projection, it follows from Lemma \ref{lem.well-behaved} that $P_\al \circ c$ is a multiple of a co-isometry. Therefore, $\cH_\al$ belongs to the domain of $c^*$. When $\xi \in \cH$ is orthogonal to all finite index $T$-subbimodules of $\cH$, then $\xi$ also belongs to the domain of $c^*$ with $c^*(\xi) = 0$. So, $c^*$ is indeed densely defined.

By \cite[Corollary 3.5]{Sa88}, we then find a continuous $1$-parameter family of unital normal $T$-bimodular completely positive maps $\vphi_t : S \recht S$, $t > 0$, given by $\vphi_t(x) = \exp(-t c^* c)(x)$, where we view $c^* c$ as a positive, self-adjoint, densely defined operator on $L^2(S)$ so that $\exp(-t c^* c)$ is a positive, self-adjoint contraction for every $t > 0$. Since $S$ has property (T) relative to $T$ and using Lemma \ref{lem.uniformity}, we get that $\lim_{t \recht 0} \|1 - \exp(-t c^*c)\| = 0$ in the operator norm of $B(L^2(S))$. This means that $c^*c$ is a bounded operator on $L^2(S)$. By Lemma \ref{lem.bounded-inner}, $c$ is an inner $1$-cocycle.

Conversely assume that $S$ does not have property (T) relative to $T$. Take a sequence of unital normal $T$-bimodular completely positive maps $\vphi_n : S \recht S$ that converge to the identity in $\|\,\cdot\,\|_2$ pointwise, but not uniformly on the unit ball of $S$. The construction of Lemma \ref{lem.construct-cocycle} gives a non inner $1$-cocycle.

2.\ If $S$ has the Haagerup property relative to $T$, then the construction in Lemma \ref{lem.construct-cocycle} provides a proper $1$-cocycle. Conversely, when $c : \cS \recht \cH$ is a proper $1$-cocycle, we define, as in the beginning of the proof of 1, the $1$-parameter family of unital normal $T$-bimodular completely positive maps $\vphi_t : S \recht S$, $t > 0$, given by $\vphi_t(x) = \exp(-t c^* c)(x)$. Using Lemma \ref{lem.compact-ideal-space}, it follows that each $\vphi_t$, $t > 0$, belongs to the compact ideal space $\cJ(\langle S,e_T \rangle)$. For every $x \in S$, we have that $\lim_{t \recht 0} \|\vphi_t(x) - x \|_2 = 0$. So, $S$ has the Haagerup property relative to $T$.

3.\ If $[S:T] < \infty$, every $1$-cocycle $c : S \recht \cH$ is bounded and thus inner by Lemma \ref{lem.bounded-inner}. Indeed, whenever $H_\al \subset L^2(S)$ is an irreducible $T$-subbimodule, the restriction of $c$ to $H_\al \cap S$ must be a multiple of an isometry. Since $L^2(S)$ is the direct sum of finitely many irreducible $T$-subbimodules, it follows that $c$ extends to a bounded operator from $L^2(S)$ to $\cH$.

By Proposition \ref{prop.characterization-amenable-inclusion}, $S$ is nonamenable relative to $T$ if and only if there exists a finite subset $\cG \subset \cS$ such that
$$
\|\xi\| \leq \sum_{x \in \cG} \|x \xi - \xi x \| \quad\text{for all}\;\; \xi \in \cH_T \; .
$$
So if $S$ is nonamenable relative to $T$, every $1$-cocycle $c : \cS \recht \cH$ that is approximately inner must be inner.

Finally assume that $[S:T] = \infty$ and that $S$ is amenable relative to $T$. We prove that there exists an approximately inner, but non inner, $1$-cocycle $c : \cS \recht \cH$. Equip the space $\Mor_{T-T}(\cS,\cH)$ with the topology of pointwise norm convergence. Since $\cS$ admits a countable basis as a $T$-module, $\Mor_{T-T}(\cS,\cH)$ is a Fr\'{e}chet space. Consider the continuous linear map $\partial : \cH_T \recht \Mor_{T-T}(\cS,\cH)$ given by $(\partial \xi)(x) = x \xi - \xi x$ for all $x \in \cS$. Since $[S:T] = \infty$, the map $\partial$ is injective. Since $S$ is amenable relative to $T$, there exists a sequence of unit vectors $\xi_k \in \cH_T$ such that $\partial \xi_k \recht 0$. So the open mapping theorem implies that $\partial(\cH_T)$ is not closed in $\Mor_{T-T}(\cS,\cH)$. Any $c \in \Mor_{T-T}(\cS,\cH)$ that lies in the closure of $\partial(\cH_T)$ but does not belong to $\partial(\cH_T)$ is an approximately inner, but non inner $1$-cocycle.
\end{proof}

\subsection{Stability under extensions of irreducible quasi-regular inclusions}

Since the late 1980s, it became clear that the appropriate notion of \emph{morphism} between finite index subfactors $T \subset S$ and $Q \subset P$ is encoded by a \emph{commuting square}
\begin{equation}\label{eq.my-comm-sq}
\coms{Q}{P}{T}{S}
\end{equation}
i.e.\ a square of inclusions of II$_1$ factors satisfying $E_Q(x) = E_T(x)$ for all $x \in S$, see e.g.\ \cite[Section 2]{Po92}. In this context, two conditions naturally emerged: nondegeneracy and compatibility of relative commutants.

The commuting square \eqref{eq.my-comm-sq} is said to be \emph{nondegenerate} if $Q S$ spans a dense subspace of $L^2(P)$. It can then be naturally extended to a system of commuting squares
$$
\begin{array}{cccccccc}
Q & \subset & P & \subset & P_1 & \subset & P_2 & \subset \cdots \\ \cup & & \cup & & \cup & & \cup & \\ T & \subset & S & \subset & S_1 & \subset & S_2 & \subset \cdots
\end{array}
$$
where $Q \subset P \subset P_1 \subset \cdots$ and $T \subset S \subset S_1 \subset \cdots$ are the Jones towers. The compatibility of the relative commutants, called \emph{smoothness} in \cite[Definition 2.3.1]{Po92}, is given by the condition $T' \cap S_n \subset Q' \cap P_n$ for all $n$.

A key example to illustrate this point of view is given by a crossed product inclusion $T \subset S = T \rtimes \Gamma$ where $\Gamma$ is a finite group. If \eqref{eq.my-comm-sq} is an arbitrary nondegenerate commuting square, then smoothness holds if and only if $P \cong Q \rtimes \Gamma$ where the action $\Gamma \actson Q$ extends the original action $\Gamma \actson T$.
Actually, in Definition \ref{def.extension-quasireg} below, we will impose the stronger condition $T' \cap S_n = Q' \cap P_n$ for all $n$. In the crossed product example, if the original $\Gamma \actson T$ is by outer automorphisms, this equality requires the extended action $\Gamma \actson Q$ to be outer as well.

We generalize these notions to arbitrary quasi-regular inclusions, which are typically of infinite index, and define the notion of an \emph{extension} of an irreducible quasi-regular inclusion in Definition \ref{def.extension-quasireg}. As in the case of finite groups, when $T \subset S = T \rtimes \Gamma$ is an arbitrary crossed product inclusion with $\Gamma \actson T$ being an outer action, extensions are exactly given as $Q \subset Q \rtimes \Gamma$ where $T \subset Q$ and the action $\Gamma \actson Q$ is outer and extends the original action $\Gamma \actson T$. As we explain in Remark \ref{rem.Cartan-extensions} below, when $T \subset S$ and $Q \subset P$ are Cartan subalgebra inclusions, our notion of extension corresponds to the familiar notion of extensions of countable ergodic equivalence relations.

In order to avoid infinite index inclusions with operator valued weights, we reformulate the smoothness condition directly in terms of bimodules, keeping in mind that in the finite index case, we have $T' \cap S_{2n-1} = \End_{T-T}(\cH_n)$, where $\cH_n$ equals the $n$-fold tensor product of $T$-bimodules $L^2(S) \ovt_T \cdots \ovt_T L^2(S)$.

Let $T \subset Q$ be an inclusion of II$_1$ factors. An extension of an automorphism $\al \in \Aut(T)$ to $Q$ is an automorphism $\be \in \Aut(Q)$ satisfying $\be(x) = \al(x)$ for all $x \in T$. Similarly, an extension of a Hilbert $T$-bimodule $\cH$ to $Q$ is a Hilbert $Q$-bimodule $\cK$ containing $\cH$ as a Hilbert $T$-subbimodule such that the following two conditions hold.
\begin{align*}
& p_\cH(a \xi) = E_T(a) \xi \quad\text{and}\quad p_\cH(\xi a) = \xi E_T(a) \quad\text{for all}\;\; \xi \in \cH , a \in Q \; , \; \text{and} \\
& Q \cH \;\;\text{and}\;\; \cH Q \;\; \text{span dense subsets of $\cK$.}
\end{align*}
For every Hilbert $T$-bimodule $\cH$, one can choose projections $p \in B(\cL) \ovt T$, $q \in B(\cL) \ovt T$, normal unital $*$-homomorphisms $\psi : T \recht p(B(\cL) \ovt T)p$ and $\vphi : T \recht q(B(\cL) \ovt T)q$ and $T$-bimodular unitary operators
$$U : \ {_{\psi(T)} p(\cL \ot L^2(T))_T} \recht \cH \quad , \quad V : \ {_T (\cL^* \ot L^2(T))q_{\vphi(T)}} \recht \cH \; .$$
Then, a Hilbert $Q$-bimodule $\cK$ is an extension of $\cH$ to $Q$ if and only if we can extend $\psi$ and $\vphi$ to normal unital $*$-homomorphisms $\psitil : Q \recht p(B(\cL) \ovt Q)p$ and $\vphitil : Q \recht q(B(\cL) \ovt Q)q$ and we can extend $U$ and $V$ to $Q$-bimodular unitary operators
$$\ {_{\psitil(Q)} p(\cL \ot L^2(Q))_{Q}} \recht \cK \quad , \quad \ {_{Q} (\cL^* \ot L^2(Q))q_{\vphitil(Q)}} \recht \cK \; .$$

Whenever $\cK$ is an extension of $\cH$ to $Q$, we get an identification $\cH \ovt_T L^2(Q) \cong \cK$ as $T$-$Q$-bimodules, as well as an identification $L^2(Q) \ovt_T \cH \cong \cK$ as $Q$-$T$-bimodules. In this way, we get the canonical normal faithful unital $*$-homomorphisms
\begin{align*}
& \Theta_r : \End_{\text{--}T}(\cH) \recht \End_{\text{--}Q}(\cK) : \Theta_r(V) = V \ot 1 \quad\text{and} \\
& \Theta_\ell : \End_{T\text{--}}(\cH) \recht \End_{Q\text{--}}(\cK) : \Theta_\ell(V) = 1 \ot V  \; .
\end{align*}

\begin{definition}\label{def.extension-quasireg}
An extension of an irreducible quasi-regular inclusion of II$_1$ factors $T \subset S$ is a nondegenerate commuting square
\begin{equation}\label{eq.new-cs}
\coms{Q}{P}{T}{S}
\end{equation}
where $Q \subset P$ is an irreducible quasi-regular inclusion of II$_1$ factors,
such that the $n$-th tensor powers $\cH_n = L^2(S) \ovt_T \cdots \ovt_T L^2(S)$ and $\cK_n = L^2(P) \ovt_{Q} \cdots \ovt_{Q} L^2(P)$ satisfy $\Theta_\ell(V) = \Theta_r(V)$ for all $V \in \End_{T-T}(\cH_n)$ and such that the resulting $*$-homomorphism $\Theta_n : \End_{T-T}(\cH_n) \recht \End_{Q-Q}(\cK_n)$ is bijective.
\end{definition}

Note that by the nondegeneracy of the commuting square \eqref{eq.new-cs}, the $Q$-bimodule $\cK_n$ is an extension of the $T$-bimodule $\cH_n$, so that $\Theta_\ell$ and $\Theta_r$ are well defined.

\begin{proposition}
Let $T \subset S$ and $Q \subset P$ be irreducible quasi-regular inclusions of II$_1$ factors. Assume that $Q \subset P$ is an extension of $T \subset S$. Denote by $\cC$ the tensor category generated by the finite index $T$-subbimodules of $L^2(S)$. Similarly define the tensor category $\cCtil$ of $Q$-bimodules generated by the finite index $Q$-subbimodules of $L^2(P)$. Then,
\begin{enumlist}
\item the rigid C$^*$-tensor categories $\cC$ and $\cCtil$ are naturally equivalent ,
\item there is a natural $*$-isomorphism between the tube $*$-algebras $\cA$, $\cAtil$ associated with $(T \subset S,\cC)$ and $(Q \subset P,\cCtil)$, preserving the weights on $\cA$, $\cAtil$,
\item for every $n \geq 0$, we have $\beta_n^{(2)}(T \subset S) = \beta_n^{(2)}(Q \subset P)$.
\end{enumlist}
\end{proposition}

\begin{proof}
Since for all $n,m \geq 0$ and $V \in \End_{T-T}(\cH_n), W \in \End_{T-T}(\cH_m)$, we have $\Theta_{n+m}(V \ot W) = \Theta_n(V) \ot \Theta_n(W)$, the system of maps $\Theta_n$ induces an equivalence $\Theta$ between the rigid C$^*$-tensor categories $\cC$ and $\cCtil$. This equivalence can also be applied to infinite direct sums of objects in $\cC$, as well as to intertwiners between such bimodules. By construction, we have that $\Theta$ maps the $T$-bimodule $L^2(S)$ to the $Q$-bimodule $L^2(P)$. Also by construction, $\Theta$ maps the morphism $\delta : L^2(T) \hookrightarrow L^2(S)$ to the morphism $\deltatil : L^2(Q) \hookrightarrow L^2(P)$, and maps the ``locally defined'' morphism given by multiplication $m : L^2(S) \ovt_T L^2(S)$ to the multiplication morphism $\mtil : L^2(P) \ovt_Q L^2(P) \recht L^2(P)$. So, $\Theta$ induces a $*$-isomorphism of the tube $*$-algebras $\cA$, $\cAtil$.

Whenever $\cH \subset \cH_n$ is a finite index $T$-subbimodule and if $(\eta_j)$ is a basis of $\cH$ as a right $T$-module, then $\Theta(\cH)$ is defined as the closed linear span of $\cH Q$ inside $\cK_n$. It follows that $(\eta_j)$ is still a basis of $\Theta(\cH)$ as a right $Q$-module. Therefore, the $*$-isomorphism $\Theta : \End_{T-T}(\cH) \recht \End_{Q-Q}(\Theta(\cH))$ preserves the right traces $\Tr^r$. Similarly, $\Theta$ preserves the left traces $\Tr^l$. Therefore, the $*$-isomorphism between $\cA$ and $\cAtil$ preserves the weights defined in Proposition \ref{prop.trace-tube-quasi-reg}.

It finally follows from Theorem \ref{thm.tor-ext} that also $\beta_n^{(2)}(T \subset S) = \beta_n^{(2)}(Q \subset P)$ for all $n \geq 0$.
\end{proof}

\begin{remark}\label{rem.Cartan-extensions}
Our notion of extension in Definition \ref{def.extension-quasireg} is the ``irreducible quasi-regular inclusion'' version of the notion of an extension of countable pmp equivalence relations. A countable pmp equivalence relation $\cP$ on $(Y,\nu)$ is said to be an extension of the countable pmp equivalence relation $\cR$ on $(X,\mu)$ if we are given a measure preserving Borel map $\Delta : Y \recht X$ such that for a.e.\ $y \in Y$, $\Delta$ is a bijection between the orbit of $y$ and the orbit of $\Delta(y)$. This notion was first considered in \cite[Definition 1.4.2]{Po05} under the name of \emph{local orbit equivalence}. In \cite[Definition 1.6]{Fu06}, it has been called \emph{bijective relation morphism}, but the currently preferred terminology is \emph{extension} of equivalence relations/Cartan subalgebras, see e.g.\ \cite{AP15}.

Given countable pmp equivalence relations $\cR$ on $(X,\mu)$ and $\cP$ on $(Y,\nu)$ and writing $T = L^\infty(X) \subset L(\cR) = S$ and $Q = L^\infty(Y) \subset L(\cP) = P$, one checks that turning $\cP$ into an extension of $\cR$ by a map $\Delta : Y \recht X$ is the same as defining a nondegenerate commuting square
\begin{equation}\label{eq.final-cs}
\coms{Q}{P}{T}{S}
\end{equation}
with the property that $\cN_S(T) \subset \cN_P(Q)$.

Assume now that $T \subset S$ and $Q \subset P$ are arbitrary quasi-regular inclusions of tracial von Neumann algebras with the property that $T' \cap S = \cZ(T)$ and $Q' \cap P = \cZ(Q)$, thus covering both irreducible inclusions and Cartan inclusions. We say that $Q \subset P$ is an extension of $T \subset S$ if we are given a nondegenerate commuting square \eqref{eq.final-cs} with the following property: denoting as above by $\cH_n = L^2(S) \ovt_T \cdots \ovt_T L^2(S)$ and $\cK_n = L^2(P) \ovt_Q \cdots \ovt_Q L^2(P)$ the $n$'th tensor powers, the maps $\Theta_\ell$ and $\Theta_r$ coincide on $\End_{T-T}(\cH_n)$ and the resulting $*$-homomorphisms $\Theta_n : \End_{T-T}(\cH_n) \recht \End_{Q-Q}(\cK_n)$ satisfy
$$\lambda(\cZ(Q)) \vee \Theta_n(\End_{T-T}(\cH_n)) = \End_{Q-Q}(\cK_n) = \rho(\cZ(Q)) \vee \Theta_n(\End_{T-T}(\cH_n))$$
for all $n \geq 1$, where we denote by $\lambda$ and $\rho$ the left and right module action of $Q$ on $\cK_n$.

It is easy to check that for Cartan inclusions, this definition is equivalent with the above definition of an extension of equivalence relations. One can also prove that extensions preserve $L^2$-Betti numbers in the sense of Definition \ref{def.L2-cohom-quasireg} (using the canonical quasi-normalizer as intermediate $*$-subalgebra $T \subset \cS \subset S$). We do not elaborate this further. Note however that using the bar resolution of Section \ref{sec.cohom-cartan} and formula \eqref{eq.L2-Betti-cohom-equiv-rel} as the definition of $L^2$-Betti numbers for equivalence relations, it follows immediately that $\beta_n^{(2)}(\cP) = \beta_n^{(2)}(\cR)$ for all $n$ whenever $\cP$ is an extension of $\cR$.
\end{remark}


\begin{thebibliography}{ABC90}\setlength{\itemsep}{-1mm} \setlength{\parsep}{0mm} \small

\bibitem[AP15]{AP15} A. Aaserud and S. Popa, Approximate equivalence of group actions. {\it Ergodic Theory Dynam. Systems}, to appear. arxiv:1511.00307

\bibitem[A74]{At74} M.F. Atiyah, Elliptic operators, discrete groups and von Neumann algebras. In {\it Colloque Analyse et Topologie en l'honneur de Henri Cartan (Orsay, 1974)}, Ast\'{e}risque {\bf 32}-{\bf 33} (1976), 43-72.

\bibitem[BHV08]{BHV08} B. Bekka, P. de la Harpe and A. Valette, Kazhdan's property (T). Cambridge University Press, Cambridge, 2008.

\bibitem[BJ95]{BJ95} D. Bisch and V.F.R. Jones, Algebras associated to intermediate subfactors. {\it Invent. Math.} {\bf 128} (1997), 89-157.

\bibitem[CG85]{CG85} {J. Cheeger and M. Gromov}, $L^2$-cohomology and group cohomology. {\it Topology} {\bf 25} (1986), 189-215.

\bibitem[CC$^+$01]{CCJJV01} P.-A. Cherix, M. Cowling, P. Jolissaint, P. Julg and A. Valette, Groups with the Haagerup property. Birkh\"{a}user Verlag, Basel, 2001.

\bibitem[C78]{Co78} A. Connes, Sur la th\'{e}orie non commutative de l'int\'{e}gration. In {\it Alg\`{e}bres d'op\'{e}rateurs (Les Plans-sur-Bex, 1978)}, Lecture Notes in Math. {\bf 725}, Springer, Berlin, 1979, pp.\ 19-143.

\bibitem[C12]{Cu12} S. Curran, On the planar algebra of Ocneanu's asymptotic inclusion. {\it Internat. J. Math.} {\bf 23} (2012), art.\ id.\ 1250114, 43 pp.

\bibitem[FH80]{FH80} T. Fack and P. de la Harpe, Sommes de commutateurs dans les alg\`{e}bres de von Neumann finies continues. {\it Ann. Inst. Fourier (Grenoble)} {\bf 30} (1980), 49-73.

\bibitem[FM75]{FM75} {J. Feldman and C.C. Moore}, Ergodic Equivalence Relations, Cohomology, and Von Neumann Algebras, I, II. {\it Trans. Amer. Math. Soc.} {\bf 234} (1977), 289-324, 325-359.

\bibitem[F06]{Fu06} A. Furman, On Popa's cocycle superrigidity theorem. {\it Int. Math. Res. Not. IMRN} {\bf 19} (2007), art.\ id.\ rnm073.

\bibitem[G01]{Ga01} {D. Gaboriau}, Invariants $\ell^2$ de relations d'\'{e}quivalence et de groupes. {\it Publ. Math. Inst. Hautes \'{E}tudes Sci.} {\bf 95} (2002), 93-150.

\bibitem[GJ15]{GJ15} S. Ghosh and C. Jones, Annular representation theory for rigid C$^*$-tensor categories. {\it J. Funct. Anal.} {\bf 270} (2016), 1537-1584.

\bibitem[GHJ89]{GHJ89} F.M. Goodman, P. de la Harpe and V.F.R. Jones, Coxeter graphs and towers of algebras. {\it Mathematical Sciences Research Institute Publications} {\bf 14}, Springer-Verlag, New York, 1989.

\bibitem[G80]{Gu80} A. Guichardet, Cohomologie des groupes topologiques et des alg\`{e}bres de Lie. CEDIC/Fernand Nathan, Paris, 1980.

\bibitem[H56]{H56} G. Hochschild, Relative homological algebra. {\it Trans. Amer. Math. Soc.} {\bf 82} (1956), 246-269.

\bibitem[H98]{Hu98} H.-P. Huang, Commutators associated to a subfactor and its relative commutants. {\it Ann. Inst. Fourier (Grenoble)} {\bf 52} (2002), 289-301.

\bibitem[J82]{Jo82}  V.F.R. Jones, Index for subfactors. {\it Invent. Math.} {\bf 72} (1983), 1-25.

\bibitem[J98]{Jo98} V.F.R Jones, The planar algebra of a bipartite graph. In {\it Knots in Hellas '98}, World Scientific, 1999, pp.\ 94-117.

\bibitem[J99]{jones:planar} V.F.R. Jones, Planar Algebras, Preprint, Berkeley 1999, arXiv:math.QA/9909027.

\bibitem[J01]{jones:annular} V.F.R. Jones, The annular structure of subfactors. In {\it Essays on geometry and related topics}, Vol. 1, 2, Monogr. Enseign. Math. {\bf 38}, Enseignement Math., Geneva, 2001, pp.\ 401-463.

\bibitem[K09]{Ky09} D. Kyed, An $L^2$-K\"{u}nneth formula for tracial algebras. {\it J. Operator Theory} {\bf 67} (2012), 317-327.

\bibitem[KPV13]{KPV13}  D. Kyed, H.D. Petersen and S. Vaes, $L^2$-Betti numbers of locally compact groups and their cross section equivalence relations. {\it Trans. Amer. Math. Soc.} {\bf 367} (2015), 4917-4956.

\bibitem[LR94]{LR94} R. Longo and K.-H. Rehren, Nets of subfactors. In {\it Workshop on Algebraic Quantum Field Theory and Jones Theory (Berlin, 1994)}. {\it Rev. Math. Phys.} {\bf 7} (1995), 567-597.

\bibitem[LR95]{LR95} R. Longo and J.E. Roberts, A theory of dimension. {\it K-Theory} {\bf 11} (1997), 103-159.

\bibitem[L02]{Lu02} W. L\"{u}ck, $L^2$-invariants: theory and applications to geometry and K-theory. Springer-Verlag, Berlin, 2002.

\bibitem[M99]{Ma99} T. Masuda, Generalization of Longo-Rehren construction to subfactors of infinite depth and amenability of fusion algebras. {\it J. Funct. Anal.} {\bf 171} (2000), 53-77.

\bibitem[M01]{Mu01} M. M\"{u}ger, From subfactors to categories and topology, I. Frobenius algebras in and Morita equivalence of tensor categories. {\it J. Pure Appl. Algebra} {\bf 180} (2003), 81-157.

\bibitem[NT13]{NT13} S. Neshveyev and L. Tuset, Compact quantum groups and their representation categories. {\it Cours Sp\'{e}cialis\'{e}s} {\bf 20}. Soci\'{e}t\'{e} Math\'{e}matique de France, Paris, 2013.

\bibitem[NY15a]{NY15a} S. Neshveyev and M. Yamashita, Drinfeld center and representation theory for monoidal categories. {\it Commun. Math. Phys.} {\bf 345} (2016), 385-434.

\bibitem[NY15b]{NY15b} S. Neshveyev and M. Yamashita, A few remarks on the tube algebra of a monoidal category. {\it Preprint.}  arXiv:1511.06332.

\bibitem[O93]{Oc93} A. Ocneanu, Chirality for operator algebras. In {\it Subfactors (Kyuzeso, 1993)}, World Sci. Publ., River Edge, 1994, pp.\ 39-63.

\bibitem[PP84]{PP84} M. Pimsner and S. Popa, Entropy and index for subfactors. {\it Ann. Sci. \'{E}cole Norm. Sup.} {\bf 19} (1986), 57-106.

\bibitem[P86]{Po86} {S. Popa}, Correspondences. {\it INCREST Preprint} {\bf 56} (1986). Available at\\ \href{http://www.math.ucla.edu/~popa/preprints.html}{www.math.ucla.edu/$\sim$popa/preprints.html}

\bibitem[P92]{Po92} S. Popa, Classification of amenable subfactors of type II. {\it Acta Math.} {\bf 172} (1994), 163-255.

\bibitem[P93]{Po93} S. Popa, Approximate innerness and central freeness for subfactors: a classification result.  In {\it Subfactors (Kyuzeso, 1993)}, World Sci. Publ., River Edge, 1994, pp.\ 274-293.

\bibitem[P94a]{Po94a} S. Popa, Symmetric enveloping algebras, amenability and AFD properties for subfactors. {\it Math. Res. Lett.} {\bf 1} (1994), 409-425.

\bibitem[P94b]{Po94b} S. Popa, An axiomatization of the lattice of higher relative commutants of a subfactor. {\it Invent. Math.} {\bf 120} (1995), 427-445.

\bibitem[P97a]{Po97a} S. Popa, The relative Dixmier property for inclusions of von Neumann algebras of finite index. {\it Ann. Sci. \'{E}cole Norm. Sup.} {\bf 32} (1999), 743-767.

\bibitem[P97b]{Po97b} S. Popa, On Connes' joint distribution trick and a notion of amenability for positive maps. {\it Enseign. Math.} {\bf 44} (1998), 57-70.

\bibitem[P99]{Po99} S. Popa, Some properties of the symmetric enveloping algebra of a subfactor, with applications to amenability and property T. {\it Doc. Math.} {\bf 4} (1999), 665-744.

\bibitem[P01]{Po01} S. Popa, On a class of type II$_1$ factors with Betti numbers invariants. {\it Ann. of Math.} {\bf 163} (2006), 809-899.

\bibitem[P05]{Po05} {S. Popa}, Cocycle and orbit equivalence superrigidity for malleable actions of $w$-rigid groups. {\it Invent. Math.} {\bf 170} (2007), 243-295.

\bibitem[PV14]{PV14} S. Popa and S. Vaes, Representation theory for subfactors, $\lambda$-lattices and C$^*$-tensor categories, {\it Commun. Math. Phys.} {\bf 340} (2015), 1239-1280.

\bibitem[S88]{Sa88} J.-L. Sauvageot, Quantum Dirichlet forms, differential calculus and semigroups. In {\it Quantum probability and applications, V (Heidelberg, 1988)}, Lecture Notes in Math. {\bf 1442}, Springer, Berlin, 1990, pp.\ 334-346.

\bibitem[T06]{T06} A. Thom, $L^2$-invariants and rank metric. In {\it C$^*$-algebras and elliptic theory II}, Trends Math., Birkh\"{a}user, Basel, 2008, pp.\ 267-280.

\bibitem[V07]{Va07} {S. Vaes}, Explicit computations of all finite index bimodules for a family of II$_1$ factors. \emph{Ann. Sci. \'{E}cole Norm. Sup.} {\bf 41} (2008), 743-788.

\end{thebibliography}
\end{document}